\documentclass[10pt,reqno]{amsart}
\usepackage{amsfonts}
\usepackage{amsmath}
\usepackage{CJK}
\usepackage{amssymb}
\usepackage{latexsym,bm}
\usepackage[RGB]{xcolor}
\usepackage{mathrsfs}
\usepackage{amsthm}
\usepackage{bm}
\usepackage{hyperref}
\usepackage[normalem]{ulem}
\allowdisplaybreaks[4]
\numberwithin{equation}{section}
\newtheorem{theorem}{Theorem}[section]

\newtheorem{lemma}[theorem]{Lemma}
\newtheorem{remark}{Remark}[section]

\newtheorem{proposition}[theorem]{Proposition}

\begin{document}{\tiny }
	\title[Stability of conservation law]{Time-asymptotic stability of composite waves of degenerate Oleinik shock and rarefaction for non-convex conservation laws with Cattaneo's law}
	%\titlerunning{Stability for non-convex conservation law}
	\author{Yuxi Hu and Ran Song}
	\thanks{\noindent Yuxi Hu, Department of Mathematics, China University of Mining and Technology, Beijing, 100083, P.R. China,
		yxhu86@163.com\\
		\indent Ran Song, Department of Mathematics, China University of Mining and Technology, Beijing, 100083, P.R. China, 
		ran.song.math@foxmail.com.
	}
	\date{}
	\maketitle
	\renewcommand{\abstractname}{\textbf{Abstract}}
	\begin{abstract}
		This paper examines the large-time behavior of solutions to a one-dimensional conservation law featuring a non-convex flux and an artificial heat flux term regulated by Cattaneo's law, forming a 2$\times$2 system of hyperbolic equations. Under the conditions of small wave strength and sufficiently small initial perturbations, we demonstrate the time-asymptotic stability of a composite wave that combines a degenerate Oleinik shock and a rarefaction wave. The proof utilizes the Oleinik entropy condition, the a-contraction method with time-dependent shifts, and weighted energy estimates.\par
		\par
		\ \\
		\noindent{{\bf AMS classification code}: 35B35, 35L65}
	\end{abstract}

	\section{Introduction}
	This paper investigates the large-time behavior of solutions to the one-dimensional scalar conservation law with Cattaneo's law, governed by the system \cite{ref29,ref5,ref7}
	\begin{equation}\label{yfc}
		\left\{
			\begin{aligned}
				&u_t+f(u)_x=q_{x},\\
				&\tau q_{t}+q=\mu u_{x},
			\end{aligned}\right.
		\end{equation}
		under the initial condition
		\begin{align}\label{IC}
			\left(u(t,x),q(t,x)\right)|_{t=0}=(u_{0}(x),q_{0}(x))\to(u_{\pm},0),\quad x\to \pm \infty,
		\end{align}
		where $(u(t,x),q(t,x))\in\mathbb{R}\times \mathbb{R}$ are the unknown functions, $(u_0,q_0)$ denotes the initial data, $\tau$ and $\mu$ are the relaxation time and the viscosity coefficient, and $u_{\pm}$ are prescribed constants. Without loss of generality, we assume that $\mu=1$ in the sequel.
		
		Equation \eqref{yfc}$_2$ is known as Cattaneo's law, in order to describe the
		finite speed of heat conduction. The relaxation time $\tau$ embodies a finite lag between a change in temperature gradient and the heat flux response. For the conservation law with Cattaneo's law, we refer to \cite{ref8,ref3,ref5} and the references therein.
		
		For solutions with distinct far-field states $u_- \neq u_+$, the large-time behavior of the system \eqref{yfc}-\eqref{IC} is governed by the corresponding inviscid Riemann problem:
		\begin{align}\label{RP}
			\begin{cases}
				u_t + f(u)_x = 0, \\
				u(0, x) = u_0(x) = 
				\begin{cases}
					u_-, & x < 0, \\
					u_+, & x > 0.
				\end{cases}
			\end{cases}
		\end{align}
		
		For a strictly convex flux $f(u)$, the solution to \eqref{RP} is either a single shock or a rarefaction wave. The time-asymptotic stability of such viscous wave patterns for conservation laws has been extensively studied since the pioneering work of Il'in and Ole\u{i}nik \cite{ref9}. In contrast, for non-convex fluxes, the Riemann solution exhibits a richer and more complex structure, which may include composite waves consisting of Oleinik shocks, rarefactions, and contact discontinuities.
		
		A shock wave propagating with speed $\sigma$ and connecting the Riemann states $u_{\pm}$ is admissible if it satisfies the Rankine-Hugoniot condition
		\begin{equation}\label{RH}
			\sigma=\frac{f\left(u_{+}\right)-f\left(u_{-}\right)}{u_{+}-u_{-}},
		\end{equation}
		and the Oleinik entropy condition
		\begin{equation}\label{Oleinik}
			\sigma \leq \frac{f(u)-f\left(u_{-}\right)}{u-u_{-}}, \quad \text{for all } u \text{ between } u_{-} \text{ and } u_{+}.
		\end{equation}
		For the specific cubic flux $f(u) = u^3$, condition \eqref{Oleinik} implies the following admissibility criteria: \\if $u_{-} < 0$, then 
		$$u_{-}<u_{+}\leq -\frac{u_{-}}{2};$$
		or if $u_{-} > 0$, then 
		$$-\frac{u_{-}}{2} \leq u_{+}<u_{-}.$$ The limiting case $u_{+}=-\frac{u_{-}}{2}$ is termed a degenerate Oleinik shock. A direct calculation for this degenerate case yields the characteristic speed relation
		\begin{align}
			f'(u_{-}) > \sigma = f'(u_{+}).
		\end{align}
		
		We note that when $\tau=0$ (i.e. Fourier's law), the system \eqref{yfc} reduces formally to the classical viscous conservation law
		\begin{align}\label{tuihua}
			u_t+f(u)_x=\mu u_{xx},
		\end{align}
		with the initial condition
		\begin{align}\label{thic}
			u(0,x)=u_0(x)\to u_\pm,\quad x\to\pm\infty.
		\end{align}
		
		The asymptotic behavior of the solutions to equation \eqref{tuihua} has been widely studied. When the flux $f(u)$ is strictly convex, there is a vast literature on the time-asymptotic stability of a single viscous shock or rarefaction wave for $\eqref{tuihua}$-$\eqref{thic}$, see \cite{ref1,ref2,ref21,ref27,ref10}. For viscous shock waves, stability is classical for one-dimensional scalar equations \cite{ref1,ref2,ref27} and has been extended to multi-dimensional scalar cases \cite{ref21}. For rarefaction waves, their asymptotic behavior and stability in scalar conservation laws have been thoroughly analyzed \cite{ref10}. Matsumura and Yoshida \cite{ref14} investigate the asymptotic behavior in time of solutions to the Cauchy problem for a one-dimensional viscous conservation law in which the ﬂux function is convex but linearly degenerate on some intervals. Hattori \cite{ref44} established decay rates for the Burgers equation.
		
		When the flux $f(u)$ is non-convex, there has been extensive research on the stability of a single viscous Oleinik shock. In the case where the non-convex flux $f(u)$ has one inflection point, Kawashima \cite{ref34} established the stability of non-degenerate Oleinik shocks using the anti-derivative and weighted $L^2$-energy methods, and Mei \cite{ref19} subsequently proved the stability for the degenerate case. For a general non-convex flux $f(u)$, Matsumura and Nishihara \cite{ref38} proved the stability of both non-degenerate and degenerate Oleinik shock waves. Liu \cite{ref30} also investigated the stability of non-degenerate shock waves with a different weight function, while Huang and Xu \cite{ref18} established the decay rate toward the viscous shock profile. Furthermore, Freist\"uhler and Serre \cite{ref17} derived an interesting $L^1$ stability theorem, and Jones-Gardner-Kapitula \cite{ref12} and Weinberger \cite{ref25} also studied the time-asymptotic stability of viscous shock waves using various methods. For a single rarefaction wave, we note that the rarefaction wave only is formed on the convex part of the flux $f(u)$. Consequently, when directly applying the $L^2$-energy method, the stability analysis of such a wave reduces to the case of a strictly convex. Although the stability of either a single non-convex shock wave or a single rarefaction wave can be addressed separately, it must be noted that the anti-derivative method for shocks and the $L^2$-energy method for rarefaction waves are not compatible with each other. This problem was recently resolved by Huang and Wang \cite{ref16} for $f(u)=u^3$. Their proof relies on a novel weighted $a$-contraction framework coupled with a time-dependent shift function.
		
		%Concerning the stability of rafaction waves in equation \eqref{tuihua}, results exist for both the whole line and the half-line. For the whole space, Hattori\cite{ref44} established decay rates for the Burgers equation, while Nakamura\cite{ref7} extended the analysis to more general hyperbolic systems with balance laws. For the half-space, both Liu\cite{ref45} and Nakamura\cite{ref46} studied scalar viscous conservation laws but with different emphases: \cite{ref45} revealed new phenomena such as boundary-layer shock and rarefaction superposition under large initial data, whereas \cite{ref46} precisely quantified the convergence rate, showing it matches the known whole-space rate. It is noteworthy that \cite{ref7} provides a unified analysis covering both domains.
		%A prominent gap in this landscape was the asymptotic stability of composite waves, specifically those involving both an Oleinik shock and a rarefaction wave in non-convex Riemann solutions. This gap was recently filled by Huang and Wang \cite{ref16} for the scalar conservation law with $f(u)=u^3$. They developed a novel weighted $a$-contraction method, combined with a carefully designed time-dependent shift function, to successfully establish the stability of such a composite wave involving a degenerate Oleinik shock and a rarefaction.
		
		When Cattaneo’s law is introduced with $\tau>0$, the system \eqref{tuihua} becomes the system \eqref{yfc}, which necessitates the development of new analytical tools. As in \cite{ref3}, the large-time behavior of solutions to \eqref{yfc}-\eqref{IC} is governed by the solution of the corresponding Riemann problem
		\begin{align*}
			\begin{cases}
				u_t + f(u)_x = 0, \\
				u(0, x) = u_0(x) = \begin{cases}
					u_-, & x < 0, \\
					u_+, & x > 0.
				\end{cases}
			\end{cases}
		\end{align*}
		For strictly convex fluxes in this setting, Nakamura and Kawashima \cite{ref5} established the stability of viscous shock profiles using the anti-derivative method and weighted $L^2$-energy estimates, while Bai, He, and Zhao \cite{ref8} proved the nonlinear stability of rarefaction waves. For non-convex fluxe, Deng \cite{ref3} obtained stability results via the standard energy method and a shift theory under the Oleinik entropy condition.
		%这里写tau>0与tau=0解的性质上的区别
		
		It is not evident a priori whether properties known for classical systems extend to their relaxed counterparts. A key distinction is illustrated by finite-time blow-up phenomena: for the relaxed system, results such as those by Hu and Wang \cite{ref23} and Hu, Racke, and Wang \cite{ref26} demonstrate that solutions can blow up in finite time under certain large initial data. In contrast, for the classical compressible Navier-Stokes system, global existence has been established for arbitrarily large initial data in \cite{ref28}.
		
		Based on the methodology established in \cite{ref16}, we investigate the time-asymptotic stability of a composite wave for the system \eqref{yfc} with a cubic flux $f(u)=u^3$. Clearly, the dissipation structure of the relaxed system \eqref{yfc} is weaker than that of the classical system \eqref{tuihua}, and it introduces substantial challenges in the energy estimates. In contrast to the results in \cite{ref16}, the entropy estimate fails to control the derivative of $(u-\tilde{u})$ due to the constraint imposed by Cattaneo's law. To obtain this missing estimate, higher-order energy estimates become necessary. However, performing the higher estimates requires the shock strength to be sufficiently small. This condition differs from the setting in \cite{ref16}. Moreover, compared with the estimates presented in \cite{ref16}, we now require $H^2$-estimates of the solutions to close the energy. 
		
		The remainder of the paper is structured as follows. Section \ref{sec2} introduces the viscous Oleinik shock profile and the rarefaction wave, and presents the main stability theorem. Section \ref{sec3} begins with the construction of an approximate rarefaction wave and the definition of the essential weight function and time-dependent shift. Using these constructs, the original stability problem is reformulated in terms of the shifted wave profiles. Finally, Section \ref{sec4} is devoted to establishing the necessary a priori estimates, divided into three parts: basic $L^2$ estimates, high-order energy estimates, and key dissipative estimates.\par
	\ \\
	\noindent{\textbf{Notations.}} For $s\in \mathbb{Z}$ and $p\in[1,\infty]$, $L^p(\mathbb{R})$ and $H^s(\mathbb{R})$ denote the standard Lebesgue space and s-th order Sobolev space in $\mathbb{R}$ with norm
	\begin{align*}
		\| f \|_{L^{p}(\mathbb{R})} := \left( \int_{\mathbb{R}} |f|^{p} dx \right)^{\frac{1}{p}},\ \ \|f\|_{H^s(\mathbb{R})}=\left(\sum_{i=1}^{s}\|\partial_{x}^if \|_{L^{p}(\mathbb{R})}^2\right)^{\frac{1}{2}}.
	\end{align*}
	$C^j((0,T);H^k(\mathbb{R}))$ denotes the space of $j$-order continuous differentiable functions on $(0,T)$ with values in $H^k(\mathbb{R})$, and $L^p((0,T);H^k(\mathbb{R}))$ denotes the space of $p$-th power integrable functions on $(0,T)$ with values in $H^k(\mathbb{R})$. For any function $g:\mathbb{R}\times\mathbb{R}^+\to\mathbb{R}$, we define $g^{\pm\mathbf{X}}(t,x)=g(t,x\pm\mathbf{X}(t))$.

	\section{Preliminaries and main result}\label{sec2}
	In this section, we first introduce the viscous Oleinik shock profile and the rarefaction wave solution. We note that two distinct cases arise in the analysis (corresponding to, whether $u_{-}<0<-\frac{u_{-}}{2}\leq u_{+}$ or $u_{+}\leq -\frac{u_{-}}{2}<0<u_{-}$). The present work focuses on the former case; the latter case can be handled by similar arguments and will not be discussed in detail here. Finally, we state the main theorem concerning the time-asymptotic stability of a composite wave of a degenerate viscous Oleinik shock and a rarefaction wave for system \eqref{yfc}.
	
	\subsection{Oleinik shock wave and rarefaction wave}
	In the subsequent discussion, we focus on the simple non-convex case of $f(u)=u^3$. Let $\xi=x-\sigma t$ with $\sigma=f'(u_{+})$ be the speed of shock wave. Plugging the form $(u^S,q^S)(\xi)$ into the system \eqref{yfc}, we have the following ODEs:
	\begin{equation}\label{tws} %tws=traveling wave solution
		\left\{ \begin{aligned}
			&-\sigma u_{\xi}^{S}+f(u^{S})_{\xi}=q_{\xi}^{S}, \\
			&-\tau\sigma q_{\xi}^{S}+q^{S}=u_{\xi}^{S} ,
		\end{aligned}\right.
	\end{equation}
	with $(u^{S},q^{S})(\xi)\to (u_{\pm},0),\ \ \xi\to \pm\infty.$
	We focus on the degenerate viscous shock profile, corresponding to the case $u_{+}=-\frac{u_{-}}{2}$ with the shock speed $\sigma=f'(u_{+})=3u_{+}^{2}.$  Substituting this into the traveling wave equation $\eqref{tws}$ and integrating over $(-\infty,\xi)$, we get
	\begin{equation}\label{1S}
		u_{\xi}^{S}=\frac{(u^{S}-u_{-})(u^{S}-u_{+})^{2}}{1+3\tau \sigma\left[(u^{S})^{2}-u_{+}^{2}\right]}.
	\end{equation}
	Hence, we present the following lemma.
	\begin{lemma}\label{lem0}
		For any states $u_{\pm}$ satisfying $u_{-}< 0$ and $u_{+}=-\frac{u_{-}}{2}>0$, and
		$\sigma>0$ satisfying R-H condition \eqref{RH} and Oleinik condition \eqref{Oleinik}, and $\tau<\frac{4}{9\sigma u_{-}^2}$, and let $\delta_{S} := |u_{+}-u_{-}|$ be the wave strength of the degenerate Oleinik shock. Then there exists a constant $C > 0$ independent of $\tau$ such that the following properties hold true:
		\begin{align*}
			u_{\xi}^{S}>0,\ \  \ \ \forall \xi\in \mathbb{R},
		\end{align*}
		and
		\begin{equation*}
			\begin{array}{ll}
				\left|u^S(\xi)-u_{-}\right| \leq C \delta_S e^{-C \delta_S^2|\xi|}, & \text {if } \xi<0, \\
				\left|u^S(\xi)-u_{+}\right| \leq \frac{C \delta_S}{1+C \delta_S^2|\xi|}, & \text {if } \xi>0, \\
				\left|u_{\xi}^S\right| \leq C \delta_S^3 e^{-C \delta_S^2|\xi|}, & \text {if } \xi<0, \\
				\left|u_{\xi}^S\right| \leq \frac{C \delta_S^3}{\left(1+C \delta_S^2|\xi|\right)^2}, & \text {if } \xi>0, \\
				\left|\partial_{\xi}^{k}u^S\right| \leq C \delta_S^2\left|\partial_{\xi}^{k-1}u^S\right|,\ k=2,3,4 &\forall \xi \in \mathbb{R}, \\
				\left|\partial_{\xi}^{k}q^{S}\right|\leq C\delta_{S}^2\left|u^{S}_{\xi}\right|,\ k=1,2,3,4 &\forall \xi \in \mathbb{R}.
			\end{array}
		\end{equation*}
	\end{lemma}
	\begin{proof}
		First, we notice that $\tau<\frac{4}{9\sigma u_{-}^2}$, then we get the sign of $u^S_{\xi}$ is determined by $(u^S-u_{-})$ and 
		\begin{align}\label{13}
			\frac{2}{3}<\left|1+3\tau \sigma\left[(u^{S})^{2}-u_{+}^{2}\right]\right|<2.
		\end{align}
		It is easily shown that $u^S-u_{-}>0$, i.e.,  $$u^S_{\xi}>0.$$
		Next, we assume that $\xi<0$, there exists a positive constant $C_{*}<1$ such that 
		\begin{equation}\label{s-}
			-\delta_{S}\leq u^S-u_{+}\leq u^S(0)-u_{+}\leq -C_{*}\delta_{S},\ \ 
			0< u^S-u_{-}\leq u^S(0)-u_{-}\leq \delta_{S}.
		\end{equation}
		By \eqref{1S}-\eqref{s-}, we obtain
		$$\left(u^S-u_{-}\right)_{\xi}\geq \frac{1}{2}(u^S-u_{-})(u^S-u_{+})^2\geq\frac{C_{*}}{2}\delta_{S}^2(u^S-u_{-}),$$
		using the Gronwall's inequality and \eqref{s-}, we get
		$$\left|u^S(\xi)-u_{-}\right| \leq C \delta_S e^{-C \delta_S^2|\xi|}.$$
		On the other hand, if $\xi>0$, there exists a positive constant $C_{*}<1$ such that 
		\begin{equation}\label{s+}
			C_{*}\delta_{S}\leq u^S(0)-u_{-}\leq u^S-u_{-}\leq\delta_{S}  ,\ \ 
			0< u_{+}-u^S\leq u_{+}-u^S(0)\leq \delta_{S}.
		\end{equation}
		By \eqref{1S}, \eqref{13} and \eqref{s+}, we have
		$$(u^S-u_{+})_{\xi}\geq \frac{1}{2}(u^S-u_{-})(u^S-u_{+})^2\geq\frac{C_{*}}{2} \delta_{S}(u^S-u_{+})^2,$$
		using the Gronwall's inequality and \eqref{s+}, we get 
		$$\left|u^S(\xi)-u_{+}\right| \leq \frac{C \delta_S}{1+C \delta_S^2|\xi|},$$
		where $C=\max\{\frac{C_{*}}{2},1\}.$ 
		
		Then, by \eqref{1S} and \eqref{13}, we get
		\begin{align*}
		\left|u^S_{\xi}\right|\leq C|u^S-u_{-}|\cdot|u^S-u_{+}|^2\leq C\delta_{S}^2|u^S-u_{-}|\leq C\delta_S^3 e^{-C \delta_S^2|\xi|},\ \ &\text{if}\  \xi<0,\\
		\left|u^S_{\xi}\right|\leq C|u^S-u_{-}|\cdot|u^S-u_{+}|^2\leq C\delta_{S}|u^S-u_{+}|^2\leq\frac{C \delta_S^3}{\left(1+C \delta_S^2|\xi|\right)^2},\ \ &\text{if}\ \xi>0.
		\end{align*}
		Using \eqref{1S}, we obtain
		\begin{align*}
		\left|u^S_{\xi\xi}\right|\leq&\left|\frac{\left[(u^S-u_{+})^2+2(u^S-u_{-})(u^S-u^{+})\right]\left(1+3\tau\sigma\left[(u^S)^2-u_{+}^2\right]\right)}{\left(1+3\tau\sigma\left[(u^S)^2-u_{+}^2\right]\right)^2}u^S_{\xi}\right|\\
		&+\left|\frac{6\tau\sigma u^S(u^S-u_{-})(u^S-u_{+})^2}{\left(1+3\tau\sigma\left[(u^S)^2-u_{+}^2\right]\right)^2}u^S_{\xi}\right|\\
		\leq&C\delta_{S}^2\left|u^S_{\xi}\right|.
		\end{align*}
		Similarly, for $k=3,4$, we can get $\left|\partial_{\xi}^{k}u^S\right| \leq C \delta_S^2\left|\partial_{\xi}^{k-1}u^S\right|.$\par
		Finally, by the equation \eqref{tws}$_1$ and $\sigma=3u_{+}^2$, we have $$q^S_{\xi}=-\sigma u^S_{\xi}+3(u^S)^2u^S_{\xi}=3(u^S-u_{+})(u^S+u_{+})u^S_{\xi},$$
		and $|u^S+u_{+}|\leq \max\{2u_{+},\left|u_{-}+u_{+}\right|\}=\frac{2\delta_{S}}{3}$.
		Then we get $\left|q^S_{\xi}\right|\leq C\delta_{S}^2\left|u^S_{\xi}\right|$.
		Similarly, for $k=2,3,4$, we get $\left|\partial_{\xi}^{k}q^{S}\right|\leq C\delta_{S}^2\left|u^{S}_{\xi}\right|$.
	\end{proof}
	We now turn to the rarefaction wave solutions of equation \eqref{yfc}. Our analysis begins by recalling the Riemann problem for the inviscid Burgers equation:
	\begin{equation}\label{IBE}
		\begin{cases}
			w_t^r+w^r w_x^r=0, \\
			w^r(0, x)=w_0^r(x)=\begin{cases}
				w_{-},\ x<0, \\
				w_{+},\ x>0.
			\end{cases}
		\end{cases}
	\end{equation}
	If $w_{-}<w_{+}$,then the Riemann problem  has a self-similar solution $w^{r}(t,x):= w^{r}(\frac{x}{t};w_{-},w_+)$ given by
	\begin{equation}\label{wr}
		w^r(t, x)=w^r\left(\frac{x}{t} ; w_{-}, w_{+}\right):=\begin{cases}
			w_{-}, \quad x \leq w_{-} t, \\
			\frac{x}{t}, \quad w_{-} t \leq x \leq w_{+} t, \\
			w_{+}, \quad x \geq w_{+} t.
		\end{cases}
	\end{equation}
	If $f''(u)>0$ on $u\in [u_-,u_+]$, then the self-similar rarefaction wave solution $u^{r}(t, x):=u^{r}(\frac{x}{t};u_-,u_+)$ can be given by
	\begin{equation}\label{RW}
		\begin{aligned}
			&u^r\left(\frac{x}{t} ; u_{-}, u_{+}\right)=(\lambda)^{-1}\left(w^r\left(\frac{x}{t} ; \lambda_{-}, \lambda_{+}\right)\right),
		\end{aligned}
	\end{equation}
	where $\lambda(u):=f^{\prime}(u) \text { and } \lambda_{ \pm}:=\lambda\left(u_{ \pm}\right)=f^{\prime}\left(u_{ \pm}\right)$.

	\subsection{Composite wave and main result}
	Let $u_{m}=-\frac{u_{-}}{2}$, then the Riemann problem $(\ref{RP})$ is solved by composite wave of a degenerate Oleinik shock connecting $u_-$ and $u_m$ and a rarefaction wave connecting $u_m$ and $u_+$. Thus, it is speculated that the time-asymptotic stability of the Cauchy problem $(\ref{yfc})$-$(\ref{IC})$ is determined by composite wave of a degenerate viscous Oleinik
	shock wave and a rarefaction wave. Then the composite wave is determined as follows:
	\begin{equation}\label{CW}
		\left(u^S\left(x-\sigma t+\mathbf{X}(t)\right)+u^r\left(\frac{x}{t} \right)-u_m, q^{S}(x-\sigma t+\mathbf{X}(t))+\left(u^{r}\left(\frac{x}{t}\right)\right)_{x}\right),
	\end{equation}
	where $u^S\left(x-\sigma t+\mathbf{X}(t)\right),\ q^S(x-\sigma t+\mathbf{X}(t))$ are defined by $(\ref{tws})$, $u^r\left(\frac{x}{t} \right)=u^r\left(\frac{x}{t} ;u_{m},u_{+}\right)$ is given in $(\ref{RW})$, and $\mathbf{X}(t)$ is defined later by \eqref{shift}.\par
	Now, we state our main result as follows.
	\begin{theorem}\label{theo1}
		Assume that the relaxation parameter $\tau$ satisfies
		\begin{align}\label{tau}
			\tau\leq\min\left\{\frac{1}{63\sigma u_{-}^2},\frac{8}{7785}\right\}.
		\end{align}
		For given constant states $u_{\pm}$ satisfy $u_{-}<0<-\frac{u_{-}}{2}<u_{+}$, there  exist positive constants $\delta_{0}, \varepsilon_{0}$ such that if $\delta_{R}=\delta_{S}^2$ and the initial data $(u_0,q_0)$ satisfies
		\begin{equation}\label{C1}
			\begin{aligned}
				& \left\|u_0(\cdot)-u^S\left(\cdot ; u_{-}, u_m\right)\right\|_{L^2\left(\mathbb{R}_{-}\right)}+\left\|u_0(\cdot)-\left(u^S\left(\cdot ; u_{-}, u_m\right)+u_{+}-u_m\right)\right\|_{L^2\left(\mathbb{R}_{+}\right)} \\
				& \quad+\left\|u_{0 x}(\cdot)-u_x^S\left(\cdot ; u_{-}, u_m\right)\right\|_{H^1(\mathbb{R})}+\sqrt{\tau}\left\|q_{0}(\cdot)\right\|_{H^2(\mathbb{R})}<\varepsilon_0,
			\end{aligned}
		\end{equation}
		where $\mathbb{R}_{+}:=-\mathbb{R}_{-}=(0,+\infty)$, and the shock wave strength satisfies $\delta_S\leq\delta_{0}$, then the initial value problem $(\ref{yfc})$-$(\ref{IC})$ has a unique global-in-time classical solution $(u(t,x),q(t,x))$. Moreover, there exist an absolutely continuous shift function $\mathbf{X}(t)$ such that
		\begin{equation}
			\begin{aligned}
				&u(t, x)-\left(u^S\left(x-\sigma t+\mathbf{X}(t)\right)+u^r\left(\frac{x}{t}\right)-u_m\right) \in C\left(\left(0,+\infty\right) ; H^2(\mathbb{R})\right),\\
				&q(t,x)-\left(q^{S}(x-\sigma t+\mathbf{X}(t))+\left(u^{r}\left(\frac{x}{t}\right)\right)_{x}\right) \in C\left(\left(0,+\infty\right); H^2(\mathbb{R})\right).
			\end{aligned}
		\end{equation}
		In addition,
		\begin{equation}\label{time}
			\begin{aligned}
				\sup _{x \in \mathbb{R}}\bigg|u(t, x)-\left(u^{S}\left(x-\sigma t+\mathbf{X}(t)\right)+u^{r}\left(\frac{x}{t}\right)-u_{m}\right)\bigg|&\to0,\ &t \rightarrow+\infty,\\
				 \sup _{x \in \mathbb{R}}\bigg|q(t,x)-q^{S}(x-\sigma t+\mathbf{X}(t))\bigg|&\to0,\ &t \rightarrow+\infty,
			\end{aligned} 
		\end{equation}
		and
		\begin{align}\label{X(t)}
			\lim _{t \rightarrow+\infty}\left|\dot{\mathbf{X}}(t)\right|=0.
		\end{align}
	\end{theorem}
	\begin{remark}
		It is straightforward  to show that Theorem $\ref{theo1}$ in fact holds under the condition $\delta_{R}=\delta_{S}^{1+\varepsilon}$, with $\varepsilon>0$.  We note that, for the classical conservation laws \eqref{tuihua}-\eqref{thic}, the strength of the rarefaction wave $\delta_R$ was also  required to be of lower order than $\delta_S$ in an implicit manner.
	\end{remark}
	\begin{remark}
		Based on $(\ref{X(t)})$, the shift function $\mathbf{X}(t)$ satisfies
		\begin{align*}
			\lim _{t \rightarrow+\infty}\left|\frac{\mathbf{X}(t)}{t}\right|=0.
		\end{align*}
		This implies that the shifted viscous Oleinik shock preserves the original Oleinik shock profile time-asymptotically.
	\end{remark}

	\section{Reformulation of the problem}\label{sec3}
	In this section, we first construct an approximate rarefaction wave. Secondly, we reformulate the stability problem by introducing a time-dependent shift, $\mathbf{X}(t)$, and a weight function $w(u^S(\xi))$. This centers the analysis on the shifted viscous Oleinik shock and the shifted approximate rarefaction wave. Thirdly, we recall the local existence theorem and present a priori estimates. Finally, we prove Theorem \ref{theo1} by combining the local existence of solutions with uniform-in-time a priori estimates.
	
	\subsection{Construction of approximate smooth rarefaction wave}
	We construct smooth approximate solution of the rarefaction wave by the smooth solutions of the following Cauchy problem of Burgers equation as in \cite{ref24}:
	\begin{equation}\label{ABE}                 \left\{\begin{array}{l}w_{t}^{R}+w^{R} w_{x}^{R}=0, \\ w^{R}(0, x)=w_{0}^{R}(x)=\frac{w_{+}+w_{-}}{2}+\frac{w_{+}-w_{-}}{2} \tanh x \rightarrow w_{ \pm},\ \text {as } x \rightarrow \pm \infty .\end{array}\right. 
	\end{equation}
	Since $w_{-}<w_{+}$ and $\left(w_{0}^{R}\right)'(x)>0$, the Cauchy problem $(\ref{ABE})$ admits a unique smooth solution $w^{R}(t,x)=w^{R}(t,x;w_{-},w_{+})$ defined by
	\begin{equation}
		\left\{\begin{array}{l}w^{R}(t,x)=w_{0}^{R}\left(x_{0}(t,x)\right)=\frac{w_{+}+w_{-}}{2}+\frac{w_{+}-w_{-}}{2} \tanh \left(x_{0}(t, x)\right), \\ x=x_{0}(t,x)+w_{0}^{R}\left(x_{0}(t, x)\right) t,\end{array}\right.
	\end{equation}
	where $x_{0}(t,x)$ is the unique intersection point with $x$-axis of the straight characteristic line of Burgers equation. Correspondingly, we construct the smooth approximate rarefaction wave $u^{R}(t,x):=u^{R}(t,x;u_{-},u_{+})$ of $u^{r}(t,x):=u^{r}(\frac{x}{t};u_{-},u_{+})$ by
	\begin{align*}
		u^{R}(t,x;u_{-},u_{+}):=(\lambda)^{-1}\left(w^{R}(t,x;\lambda_{-},\lambda_{+})\right),
	\end{align*}
	where $\lambda(u):=f'(u),\ \lambda_{\pm}=\lambda(u_{\pm})=f'(u_{\pm})$. It is easy to check that $u^{R}(t,x)$ satisfies the equation
	\begin{equation}  \left\{\begin{array}{l}u_{t}^{R}+f\left(u^{R}\right)_{x}=0, \\ u^{R}(0, x)=(\lambda)^{-1}\left(\frac{\lambda_{+}+\lambda_{-}}{2}+\frac{\lambda_{+}-\lambda_{-}}{2} \tanh x\right) \rightarrow u_{ \pm}, \ \text {as }\  x \rightarrow \pm \infty .\end{array}\right. 
	\end{equation}
	We have the following lemma to estimate $u^{R}$, see \cite{ref16,ref24}.
	\begin{lemma}\label{lem1}
		Assume \( u_{-} < u_{+} \)and \( f^{\prime \prime}(u) > 0 \) on \( u \in\left[u_{-}, u_{+}\right] \)and let \( \delta_{R}:=\left|u_{+}-u_{-}\right| \) be the rarefaction wave strength. Then it holds that\\
		$(1)$ \( u_{-} < u^{R}(t, x) < u_{+}, \ u_{x}^{R}(t, x) > 0 \), for \( t \geq 0 \) and \( x \in \mathbb{R} \);\\
		$(2)$ For all \( t \geq 1 \) and \( p \in[1,+\infty] \), there exists a positive constant \( C_{p} \) such that
		\begin{align*}
			\begin{array}{l}
				\left\|u_{x}^{R}(t)\right\|_{L^{p}(\mathbb{R})} \leq C_{p} \min \left\{\delta_{R}, \delta_{R}^{\frac{1}{p}} t^{-1+\frac{1}{p}}\right\}, \\
				\left\|\partial_{x}^{k}u^{R}(t)\right\|_{L^{p}(\mathbb{R})} \leq C_{p} \min \left\{\delta_{R}, t^{-1}\right\},\ \ k=2,3,4;
			\end{array}
		\end{align*}
		$(3)$ 
		\begin{align*}
			\lim_{t \to \infty} \sup_{x \in \mathbb{R}}\left|u^{R}(t, x)-u^{r}\left(\frac{x}{t}\right)\right|=0,
		\end{align*}
		that is, the approximate rarefaction wave \( u^{R}(t, x) \) and the inviscid self-similar rarefaction wave fan \( u^{r}\left(\frac{x}{t}\right) \) are equivalent to every time-asymptotically for \( x \in \mathbb{R} \) uniformly;\\
		$(4)$ For all \( t \geq 0 \), it holds that
		\begin{align*}
			\begin{array}{ll}
				\left|u^{R}(t, x)-u_{+}\right| \leq C \delta_{R} e^{-2\left|x-\lambda_{+} t\right|}, & x \geq \lambda_{+} t, \\
				\left|u^{R}(t, x)-u_{-}\right| \leq C \delta_{R} e^{-2\left|x-\lambda_{-} t\right|}, & x \leq \lambda_{-} t.
			\end{array}
		\end{align*}
		$(5)$ For all \( t \geq 1 \) and any \( \varepsilon \in(0,1) \), there exists a positive constant \( C_{\varepsilon} \) such that
		\begin{align*}
			\begin{array}{l}
				\left|u^{R}(t, x)-u_{+}\right| \leq C_{\varepsilon} \delta_{R}^{\frac{2 \varepsilon}{2+\varepsilon}} t^{-1+\varepsilon} e^{-\varepsilon\left|x-\lambda_{+} t\right|}, \quad x \geq \lambda_{+} t, \\
				\left|u^{R}(t, x)-u_{-}\right| \leq C_{\varepsilon} \delta_{R}^{\frac{2 \varepsilon}{2+\varepsilon}} t^{-1+\varepsilon} e^{-\varepsilon\left|x-\lambda_{-} t\right|}, \quad x \leq \lambda_{-} t.
			\end{array}
		\end{align*}
		$(6)$ For all \( t \geq 1 \) and any \( \varepsilon \in(0,1) \), there exists a positive constant \( C_{\varepsilon} \) such that
		\begin{align*}
			\left|u^{R}(t, x)-u^{r}\left(\frac{x}{t}\right)\right| \leq C_{\varepsilon} \delta_{R}^{\varepsilon} t^{-1+\varepsilon}, \quad \lambda_{-} t \leq x \leq \lambda_{+} t .
		\end{align*}
	\end{lemma}
	
	\subsection{Reformulation of the problem}
	In order to simplify our analysis, we rewrite the system $(\ref{yfc})$ through the coordinates transformation $(t,x)\to (t,\xi=x-\sigma t)$:
	\begin{equation} \label{TFC}
		\begin{cases}
			u_t - \sigma u_{\xi} + f(u)_{\xi} = q_{\xi}, \\
			\tau q_t - \tau\sigma q_{\xi} + q = u_{\xi}.
		\end{cases} 
	\end{equation}\par
	Next, we analyze the composite wave $$\left(u^S(x-\sigma t+\mathbf{X}(t))+u^r(\frac{x}{t})-u_{m}, q^S(x-\sigma t+\mathbf{X}(t))+\left(u^r(\frac{x}{t})\right)_{x}\right).$$ Since the rarefaction wave $u^r(\frac{x}{t};u_m,u_+)$ is only Lipschitz continuous. We replace $u^r(\frac{x}{t};u_m,u_+)$ with a shifted approximate rarefaction wave  $u^R(1+t,x+\mathbf{X}(t);u_m,u_+)$ to address this problem for the subsequent stability analysis. Here the shift function $\mathbf{X}(t)$ remains to be determined. Thus, the stability ansatz is defined by
	 \begin{align}
	 	&\tilde{u}(t,\xi+\mathbf{X}(t)):=u^S\left(\xi+\mathbf{X}(t)\right)+u^R\left(1+t,\xi+\sigma t+\mathbf{X}(t) \right)-u_m, \\
	 	&\tilde{q}(t,\xi+\mathbf{X}(t)):=q^{S}(\xi+\mathbf{X}(t))+\left(u^R\left(1+t,\xi+\sigma t+\mathbf{X}(t) \right)\right)_{\xi}.
	 \end{align}
	For simplification, we denote that
	\begin{align*}
		\tilde{u}^{\mathbf{X}}(t,x):=\tilde{u}(t,x+\mathbf{X}(t)),&\ \tilde{q}^{\mathbf{X}}(t,x):=\tilde{q}(t,x+\mathbf{X}(t)),\\
		\left(u^{S}\right)^{\mathbf{X}}(t,\xi):=u^{S}(t,\xi+\mathbf{X}(t)),&\ \left(u^{R}\right)^{\mathbf{X}}(t,\xi):=u^{R}(t,\xi+\mathbf{X}(t)).
	\end{align*}\par
	Meanwhile, the ansatz $\left(\tilde{u}^{\mathbf{X}},\tilde{q}^{\mathbf{X}}\right)$ satisfies the following system:
	\begin{equation} \label{XFC}
		\begin{cases}
			\tilde{u}_t^{\mathbf{X}} - \sigma \tilde{u}_{\xi}^{\mathbf{X}} + f\left(\tilde{u}^{\mathbf{X}}\right)_{\xi} - \tilde{q}_{\xi}^{\mathbf{X}} = \dot{X}(t)\left(\left(u^S\right)_{\xi}^{\mathbf{X}}+ \left(u^R\right)_{\xi}^{\mathbf{X}}\right) + F^{\mathbf{X}}, \\
			\tau\tilde{q}_t^{\mathbf{X}} - \tau \sigma \tilde{q}_{\xi}^{\mathbf{X}} + \tilde{q}^{\mathbf{X}} - \tilde{u}_{\xi}^{\mathbf{X}} = \tau \dot{X}(t) \tilde{q}_{\xi}^{\mathbf{X}} + \tau \left(u^R\right)_{t\xi}^{\mathbf{X}},
		\end{cases}
	\end{equation}
	where the error term $F^{\mathbf{X}}(t,\xi)=F(t,\xi+\mathbf{X}(t))$ is given by
	\begin{equation}
		F^{\mathbf{X}}= \left(f\left(\tilde{u}^{\mathbf{X}}\right) - f\left(\left(u^S\right)^{\mathbf{X}}\right) - f\left(\left(u^R\right)^{\mathbf{X}}\right)\right)_{\xi}- \left(u^R\right)_{\xi\xi}^{\mathbf{X}}. 
	\end{equation}
	We define the perturbation
	\begin{equation}\label{P} 
		\phi(t, \xi) := u(t, \xi) - \tilde{u}^{\mathbf{X}}(t, \xi),\ \ r(t,\xi):=q(t,\xi)-\tilde{q}^{\mathbf{X}}(t,\xi). 
	\end{equation} 
	The perturbation $(\phi, r)$ satisfies the following equation, derived from systems (\ref{TFC}) and (\ref{XFC}):
	\begin{equation}\label{PFC}
		\left\{\begin{array}{l}\phi_{t}-\sigma \phi_{\xi}+\left(f\left(\phi+\tilde{u}^{\mathbf{X}}\right)-f\left(\tilde{u}^{\mathbf{X}}\right)\right)_{\xi}+\dot{\mathbf{X}}(t)\left(\left(u^{S}\right)_{\xi}^{\mathbf{X}}+\left(u^{R}\right)_{\xi}^{\mathbf{X}}\right)-r_{\xi}=-F^{\mathbf{X}}, \\ 
			\tau r_{t}-\tau\sigma r_{\xi}+r-\phi_{\xi}+\tau \dot{\mathbf{X}}(t) \tilde{q}_{\xi}^{\mathbf{X}}+\tau \left(u^{R}\right)_{t\xi}^{\mathbf{X}}=0,\end{array}\right. 
	\end{equation}
	with initial condition
	\begin{align}\label{initial}
		\left(\phi(0,\xi),r(0,\xi)\right)=\left(\phi_{0}(\xi),r_0(\xi)\right)=\left(u_0(\xi)-\tilde{u}(0,\xi),q_0(\xi)-\tilde{q}(0,\xi)\right).
	\end{align}
	
	\subsection{Construction of weight function and shift function}
	There exist unique values $\xi_{1}$ and $\xi_{2}$ such that $u^S(\xi_1)=0$ and $u^S(\xi_2)= \frac{u_m}{2}$, respectively. We introduce the weight function $w(\xi)=w(u^{S}(\xi))$ in \cite{ref16}:
	\begin{equation} \label{WF}
		w(\xi)=w(u^{S}(\xi)) := \begin{cases}
			\frac{5}{2}u_{m}(u_{m} - u^{S}(\xi)), & \xi \in (-\infty, \xi_{1}) \iff u^{S} \in (u_{-}, 0), \\
			\frac{5}{2u_{m}^{2}}(u_{m} - u^{S})[4(u^{S})^{3} + u_{m}^{3}], & \xi \in [\xi_{1}, \xi_{2}) \iff u^{S} \in [0, \frac{u_{m}}{2}), \\
			\frac{15}{8}u_{m}^{2}, & \xi \in [\xi_{2}, +\infty) \iff u^{S} \in [\frac{u_{m}}{2}, u_{m}).
		\end{cases} 
	\end{equation}
	Notice that
	\begin{equation} \label{WFP}
		\begin{aligned}
			\frac{15}{8}u_m^2 \leq w < \frac{15}{2}u_m^2,\ \ 
			-\frac{5}{2}u_m \leq w' \leq 0,\ \ 
			0 \leq w'' \leq \frac{15}{2}, 
		\end{aligned}
	\end{equation}
	where $w':=\frac{dw(u^{S})}{du^S}.$ It is clear that $w\in \mathbf{C}^{2}(\mathbb{R}).$\par
	The ensuing analysis relies critically on the shift function $\mathbf{X}(t)$, which is defined as the solution to an ordinary differential equation, following the approach in \cite{ref16}:
	\begin{equation} \label{shift}
		\begin{cases}
			\dot{\mathbf{X}}(t) = \frac{32}{25u_m^2}\int_{\mathbb{R}}\phi(t, \xi)w^{\mathbf{X}}(u^S(\xi))(u^S)_{\xi}^{\mathbf{X}}(\xi)d\xi, \\
			\mathbf{X}(0) = 0,
		\end{cases} 
	\end{equation}
	where $w^{\mathbf{X}}(u^{S}(\xi)):=w\left(\left(u^{S}\right)^{\mathbf{X}}(\xi)\right).$
	We now state the following lemma, as shown in \cite{ref3,ref16}.
	\begin{lemma}\label{le0}
		For any $T,C_0>0$, there exists a positive constant $C$ such that for $\forall u\in L^{\infty}((0,T);\mathbb{R})$ with $|u(t,x)|\leq C_0,\ (t,x)\in [0,T]\times\mathbb{R}$, then the ODE $(\ref{shift})$ has a unique absolutely continuous solution on $[0,T]$. Moreover,
		\begin{align*}
			|\dot{\mathbf{X}}(t)|\leq Ct,\ \forall t\in[0,T].
		\end{align*}
	\end{lemma}
	Then, we prove Theorem $\ref{theo1}$ by combining a local existence result with a priori estimates and dissipative estimates through a continuity argument. We note that system $\eqref{PFC}$ constitutes a quasi-linear symmetric hyperbolic system of $\left(\phi,\sqrt{\tau}r\right)$, see \cite{ref40}.
	\begin{theorem}\label{local}
		 (Local existence)
		For any $M>0$, there exists a constant $T_0=T_0(M)>0$ such that if
		\begin{align*}
			\left\lVert\left(u_0(\xi)-\tilde{u}(0,\xi), \sqrt{\tau}q_0(\xi)-\sqrt{\tau}\tilde{q}(0,\xi)\right)\right\rVert_{H^2(\mathbb{R})} \leq M,
		\end{align*}
		then the Cauchy problem $(\ref{PFC})$-$(\ref{initial})$ has a unique solution $(\phi,r)$ on $[0,T_0]$ such that 
		\begin{align*}
			\left(u-\tilde{u}, \sqrt{\tau}q-\sqrt{\tau}\tilde{q}\right)\in C^0\left([0,T_0];H^2\right)\cap C^1\left([0,T_0];H^1\right),
		\end{align*}
		and
		\begin{align*}
			\sup_{0\leq t\leq T_0} \|u(t,\cdot)-\tilde{u}(t,\cdot), \sqrt{\tau}q(t,\cdot)-\sqrt{\tau}\tilde{q}(t,\cdot)\|_{H^2(\mathbb{R})} \leq 2M.
		\end{align*}
	\end{theorem}
	The following proposition provides a priori estimates of the perturbation $(\phi,r)$.
	\begin{proposition}\label{p1}
		(A priori estimates)
		For given $u_{-}<0$ and $u_{\pm}$ satisfying $u_{-}<u_m=-\frac{u_{-}}{2}<u_{+}$. Let $\phi(t,\xi)$ and $r(t,\xi)$ being local solutions to system \eqref{PFC}, then there exist positive constants $\delta_{1},\varepsilon_{1}$ such that if the shock wave strength $\delta_{S}\leq \delta_{1}$ and the rarefaction wave strength $\delta_{R}=\delta_S^2$, and 
		\begin{equation}\label{js}
			\begin{aligned}
				\sup_{0\leq t\leq T} \|\phi(t,\cdot), \sqrt{\tau}r(t,\cdot)\|_{H^2(\mathbb{R})} \leq \varepsilon_1,
			\end{aligned} 
		\end{equation}
		then there exists a uniform constant $C>0$ (independent of $\tau$ and $T$) such that for $0\leq t\leq T$ 
		\begin{equation} 
			\begin{aligned}
				\sup_{0\leq t\leq T}&\left[\|\phi(t, \cdot)\|_{H^{2}}^2+ \tau\|r(t, \cdot)\|_{H^{2}}^2\right]+ \int_0^t \|r(s, \cdot)\|_{H^{2}}^2 ds + \int_0^t\|\phi_{\xi}\|_{H^2}^2 ds \\
				&\leq C \left(\|\phi_0\|_{H^2}^2+\tau\|r_0\|_{H^2}^2 +\delta_R^{\frac{8}{33}}\right),
			\end{aligned}
		\end{equation}
		where $\delta_S=|u_m-u_-|,\ \delta_{R}=|u_+-u_m|$. 
	\end{proposition}
	The proof of Proposition \ref{p1} will be presented in Section \ref{sec4}.
	
	\subsection{The continuity arguments}
	Under the assumption $\eqref{js}$, we can therefore extend the local solution to a global solution for all $t\in\left[0,+\infty\right)$ via the continuity argument. We define
	\begin{align*}
		\varepsilon^{*}:=\min\left\{\frac{\varepsilon_{1}}{4},\sqrt{\frac{\varepsilon_{1}^2}{36C_0}-\delta_R^{\frac{8}{33}}}\right\},
	\end{align*}
	where $\varepsilon_1,C_0$ are given by Proposition $\ref{p1}$. Clearly, $\varepsilon^*>0$ by the smallness of $\delta_{1}$. Let $M=\frac{\varepsilon_1}{4}$, the local existence result established in Theorem $\ref{local}$ ensures the existence of a positive constant $T_0$ such that 
	\begin{align}\label{000}
		\sup_{0\leq t\leq T_0} \|u(t,\cdot)-\tilde{u}(t,\cdot), \sqrt{\tau}q(t,\cdot)-\sqrt{\tau}\tilde{q}(t,\cdot)\|_{H^2(\mathbb{R})} \leq \frac{\varepsilon_1}{2}.
	\end{align}
	By the Lemma $\ref{le0}$, we have $$\sup_{0\leq t\leq T_0,\ x\in\mathbb{R}}|u(t,x)|\leq C,$$ 
	and 
	$$|\dot{\mathbf{X}}(t)|\leq Ct,\ \ \forall t\in[0,T_0].$$
	Consequently, we can choose $T_0$ sufficiently small so that for all $t\in[0,T_0]$, we obtain
	\begin{align}\label{111}
		\lVert\tilde{u}^{\mathbf{X}}-\tilde{u}\rVert_{H^2(\mathbb{R})}+\sqrt{\tau}\lVert\tilde{q}^{\mathbf{X}}-\tilde{q}\rVert_{H^2(\mathbb{R})}
		&\leq Ct\lVert u_{\xi}^S+u_{\xi}^R\rVert_{H^2(\mathbb{R})}+C\sqrt{\tau}t\lVert q_{\xi}^S+u_{\xi\xi}^R\rVert_{H^2(\mathbb{R})}\nonumber\\
		&\leq \frac{\varepsilon_1}{8}.
	\end{align}
	Let energy term $E(t)$ be defined as
	\begin{align*}
		E(t)=\sup_{0\leq s\leq t}\left\lVert\left(\phi,\sqrt{\tau}r\right)\right\rVert_{H^2(\mathbb{R})}.
	\end{align*}
	From $(\ref{000})$ and $(\ref{111})$, we have 
	\begin{align*}
		E(T_0)\leq \frac{\varepsilon_{1}}{2}+\frac{\varepsilon_1}{8}<\varepsilon_1.
	\end{align*}
	We next turn to the maximal existence time:
	\begin{align*}
		T_{max}:=\sup\left\{t>0|E(t)\leq \varepsilon_1\right\}.
	\end{align*}
	For now, we assume $E(0)\leq \varepsilon_0$. If $T_{max}<+\infty$, then $E(T_{max})=\varepsilon_1.$ While it holds from Proposition $\ref{p1}$ that
	\begin{align*}
		E(T_{max})\leq \sqrt{C_0\left(\|\phi_0\|_{H^2}^2+\tau\|r_0\|_{H^2}^2 +\delta_R^{\frac{8}{33}}\right)}=\sqrt{C_0E(0)+C_0\delta_{R}^{\frac{8}{33}}}\leq \frac{\varepsilon_1}{4},
	\end{align*}
	which contradicts $E(T_{max})=\varepsilon_1.$ Thus, we have $T_{max}=+\infty$. Together with Proposition $\ref{p1}$, this yields
	\begin{equation} \label{T>0}
		\begin{aligned}
			\sup_{t\geq0 }&\left[\|\phi(t, \cdot)\|_{H^{2}}^2+ \tau\|r(t, \cdot)\|_{H^{2}}^2\right]+ \int_0^t \|r(s, \cdot)\|_{H^{2}}^2 ds + \int_0^t \int_{\mathbb{R}}\|\phi_{\xi}\|_{H^2}^2 d\xi ds \\
			&\leq C \left(\|\phi_0\|_{H^2}^2+\tau\|r_0\|_{H^2}^2 +\delta_R^{\frac{8}{33}}\right).
		\end{aligned}
	\end{equation}
	In addition, we need to check that the initial energy $E(0)=\|\phi_0\|_{H^2}^2+\tau\|r_0\|_{H^2}^2$ is sufficiently small.
	\begin{align*}
		\small
		E(0)=&\left\lVert\left(\phi_0,\sqrt{\tau}r_0\right)\right\rVert_{H^2(\mathbb{R})}\\=&\left\| u_0 - (u^S(x; u_-, u_m) + u^R(1, x; u_m, u_+) - u_m) \right\|_{H^2(\mathbb{R})}\\
		&+\sqrt{\tau}\left\| q_0 - (q^S(x) + u_{\xi}^R(1, x; u_m, u_+)) \right\|_{H^2(\mathbb{R})}\\
		\leq& \left\| u_0 - u^S(x; u_-, u_m) \right\|_{L^2(\mathbb{R}_-)} + \left\| u^R(1, x; u_m, u_+) - u_m \right\|_{L^2(\mathbb{R}_-)}\\
		&+ \left\| u_0 - (u^S(x; u_-, u_m) + u_+ - u_m) \right\|_{L^2(\mathbb{R}_+)} + \left\| u^R(1, x; u_m, u_+) - u_+ \right\|_{L^2(\mathbb{R}_+)}\\
		&+ \left\| u_{0x}(\cdot) - u^S_x(\cdot; u_-, u_m) \right\|_{H^1(\mathbb{R})} + \left\| u^R_x(1, x; u_m, u_+) \right\|_{H^1(\mathbb{R})}\\
		&+\sqrt{\tau}\left\| q_0  \right\|_{H^2(\mathbb{R})}+(1+C(\tau))\sqrt{\tau}\left\| u_{\xi}^S(x;u_-,u_m)\right\|_{H^2(\mathbb{R})}\\
		&+\sqrt{\tau}\left\|u_{\xi}^R(1, x; u_m, u_+)) \right\|_{H^2(\mathbb{R})}\\
		\leq& \left\| u_0 - u^S(x; u_-, u_m) \right\|_{L^2(\mathbb{R}_-)} + \left\| u_0 - (u^S(x; u_-, u_m) + u_+ - u_m) \right\|_{L^2(\mathbb{R}_+)}\\
		&+ \left\| u_{0x}(\cdot) - u^S_x(\cdot; u_-, u_m) \right\|_{H^1(\mathbb{R})} +\sqrt{\tau}\left\| q_0  \right\|_{H^2(\mathbb{R})} +C_1\left( \delta_S+\delta_R\right),
	\end{align*}
	where $C_1$ is given by the Lemma $\ref{lem0}$ and $\ref{lem1}$. Thus, we take $\delta_{0}=\min\{\delta_{1},\frac{\varepsilon_{1}}{2C_1}\},\varepsilon_{0}=\frac{\varepsilon_{1}^2}{32C_0}$. Then, in view of \eqref{C1}, we arrive at the estimate $E(0)\leq \varepsilon_0$.
	
	\subsection{Time-asymptotic behavior}
	To establish the time-asymptotic behavior $(\ref{time})$ and $(\ref{X(t)})$, we define $g(t)=\lVert\phi_{\xi}\rVert_{L^{2}(\mathbb{R})}^2+\tau\lVert r_{\xi}\rVert_{L^{2}(\mathbb{R})}^2$. From the estimate $(\ref{T>0})$, we have
	\begin{align*}
		\int_{0}^{+\infty} |g'(t)| dt = \int_{0}^{+\infty} \left|\frac{d}{dt} \|\phi_{\xi}\|_{L^2(\mathbb{R})}^{2}+\tau\frac{d}{dt} \|r_{\xi}\|_{L^2(\mathbb{R})}^{2}\right| dt\leq C.
	\end{align*}
	Hence, $$\lim_{t \to +\infty}g(t)=\lim_{t \to +\infty}\left(\lVert\phi_{\xi}\rVert_{L^{2}(\mathbb{R})}^2+\tau\lVert r_{\xi}\rVert_{L^{2}(\mathbb{R})}^2\right)=0.$$
	Applying the Sobolev inequality yields
	\begin{align*}
		&\lim_{t \to +\infty} \|\phi(t, \cdot)\|_{L^{\infty}(\mathbb{R})} \leq C\lim_{t \to +\infty}  \|\phi(t, \cdot)\|_{L^2(\mathbb{R})}^{\frac{1}{2}} \|\phi_{\xi}(t, \cdot)\|_{L^2(\mathbb{R})}^{\frac{1}{2}} = 0,\\
		&\lim_{t \to +\infty} \sqrt{\tau}\|r(t, \cdot)\|_{L^{\infty}(\mathbb{R})} \leq C\sqrt{\tau}\lim_{t \to +\infty}  \|r(t, \cdot)\|_{L^2(\mathbb{R})}^{\frac{1}{2}} \|r_{\xi}(t, \cdot)\|_{L^2(\mathbb{R})}^{\frac{1}{2}} = 0.
	\end{align*}
	Moreover, by referring to the definition of the shift $\dot{\mathbf{X}}(t)$ (see $(\ref{shift})$), it follows that
	\begin{align*}
		|\dot{\mathbf{X}}(t)| \leq C \|\phi\|_{L^{\infty}(\mathbb{R})}\to 0, \ as\ t\to +\infty.
	\end{align*}
	By Lemma $(\ref{lem1})$ and $(\ref{wr})$, we have 
	\begin{align*}
		&\sup_{x \in \mathbb{R}} \left| u^{r} \left(\frac{x}{t}; u_{m}, u_{+}\right) - u^{R} \left(1 + t, x + \mathbf{X}(t); u_{m}, u_{+}\right) \right|\\
		\leq& \sup_{x \in \mathbb{R}} \left| u^{r} \left(\frac{x}{t}; u_{m}, u_{+}\right) - u^{r} \left(\frac{x}{1+t}; u_{m}, u_{+}\right) \right|\\
		&+ \sup_{x \in \mathbb{R}} \left| u^{r} \left(\frac{x}{1+t}; u_{m}, u_{+}\right) - u^{r} \left(\frac{x + \mathbf{X}(t)}{1+t}; u_{m}, u_{+}\right) \right|\\
		&+ \sup_{x \in \mathbb{R}} \left| u^{r} \left(\frac{x + \mathbf{X}(t)}{1+t}; u_{m}, u_{+}\right) - u^{R} \left(1 + t, x + \mathbf{X}(t); u_{m}, u_{+}\right) \right|\\
		\leq&C(1 + t)^{-1}+C\left|\dot{\mathbf{X}}(t)\right|+C(1+t)^{-\frac{1}{2}}
		\to0, \ (t\to+\infty).
	\end{align*}
	Therefore, the proof of Theorem $\ref{theo1}$ is finished.

	\section{A priori estimates}\label{sec4}
	This section is devoted to establishing the uniform-in-time a priori estimates stated in Proposition \ref{p1}. We assume that the Cauchy problem (\ref{PFC}) under the initial condition $$\left(\phi(t,\xi),r(t,\xi)\right)|_{t=0}=\left(\phi_{0},r_{0}\right)$$ admits a solution $\phi,r\in \mathbf{C}\left([0,T];H^{2}(\mathbb{R})\right)$ for some given time $T>0.$ We start with $L^2$ relative entropy estimate for $\left(\phi,r\right)$.
	
	\subsection{$L^2$ estimates}
	Let $U(t,\xi):=(u,q)^{T}$ and $\tilde{U}(t,\xi):=(\tilde{u},\tilde{q})^{T}$, and we define a relative entropy quantity:
	\begin{align*}
		\eta(U|\tilde{U})=\frac{|u-\tilde{u}|^{2}+\tau|q-\tilde{q}|^{2}}{2}.
	\end{align*}
	\begin{lemma}\label{lem2}
		Assume that $\tau$ satisfies \eqref{tau}, there exist positive constants $\delta_{0},\varepsilon_{2}>0$ such that if the rarefaction wave strength $\delta_{R}=\delta_{S}^2$ and shock wave strength $\delta_{S}<\delta_{0}$, and $$\sup_{0\leq t\leq T} \|\phi(t,\cdot), \sqrt{\tau}r(t,\cdot)\|_{H^2(\mathbb{R})} \leq \varepsilon_2,$$ then there exists a positive constant C (independent of $\tau$ and $T$) such that for $t\in[0,T]$
		\begin{align*}
			&\int_{\mathbb{R}}w\eta(U|\tilde{U}^{\mathbf{X}})d\xi+ \left(2-2k_{1}\right)\int_0^t\int_{\mathbb{R}}r^2wd\xi ds + \frac{25u_m^2}{64}\int_0^t |\dot{\mathbf{X}}(s)|^2 ds \\
			&+ \frac{4u_m^3}{5}\int_0^t \int_{\mathbb{R}} \phi^2(u^S)_\xi^{\mathbf{X}} d\xi ds+3u_m\int_0^t \int_{\mathbb{R}} \phi^2w(u^R)_\xi^{\mathbf{X}} d\xi ds+2\int_0^t\mathbf{G}^{SR}(s)ds \\
			\leq& Cu_m^2 (\|\phi_0\|_{L^2}^2+\tau\|r_0\|_{L^2}^2 )+2k_2\int_{0}^t\int_{\mathbb{R}}\left(\phi_{\xi}\right)^2wd\xi ds+Cu_m^2\delta_R^{\frac{8}{33}},
		\end{align*}
		where $k_1,k_2$ satisfy the condition $(\ref{k1k2})$ below and
		\begin{align*}
			\mathbf{G}^{SR}(t):=3 \int_{\mathbb{R}} (\phi^{-X})^2 w(u^R - u_m) u_{\xi}^S d\xi - \frac{3}{2} \int_{\mathbb{R}} (\phi^{-X})^2 (u^R - u_m)^2 w' u_{\xi}^S d\xi.
		\end{align*}
		\begin{remark}
			Unlike the case in \cite{ref16}, Lemma $\ref{lem2}$ fails to provide an estimate for $\phi_{\xi}$. Consequently, we must leverage both the dissipation estimates and the high-order energy estimates to establish control over this first-order term.
		\end{remark}
	\end{lemma}
	The proof of Lemma $\ref{lem2}$ will be established by means of the following lemmas.
	\begin{lemma}\label{lem2.5}
		Let $\mathbf{X}(t)$ and $w(u^{S}(\xi))$ be respectively the shift function in $(\ref{shift})$ and the weight function in $(\ref{WF})$. Then the following equation hold:
		\begin{equation}\label{yds} 
			\frac{d}{dt} \int_{\mathbb{R}} w \eta(U^{-\mathbf{X}}|\tilde{U}) d\xi + \dot{\mathbf{X}}(t)\mathbf{Y}(t) + J^{good}(t) + J^{bad}(t) = \mathbf{Z}(t), 
		\end{equation}
		where
		\small
		\begin{equation*} 
			\begin{aligned}
				\mathbf{Y}(t) &:= \int_{\mathbb{R}} \phi^{-\mathbf{X}} w u_{\xi}^{S} d \xi+\int_{\mathbb{R}} \phi^{-\mathbf{X}} w u_{\xi}^{R} d \xi-\frac{1}{2} \int_{\mathbb{R}}\left(\phi^{-\mathbf{X}}\right)^{2} w^{\prime} u_{\xi}^{S} d \xi-\tau\int_{\mathbb{R}}\phi^{-\mathbf{X}}q_{\xi}^{S}w'u_{\xi}^{S}d\xi\\
				&-\tau \int_{\mathbb{R}}r^{-\mathbf{X}}\left(u_{\xi}^{S}+u_{\xi}^{R}\right)w'u_{\xi}^{S}d\xi+\tau\int_{\mathbb{R}}r^{-\mathbf{X}}w\tilde{q}_{\xi}d\xi-\tau\int_{\mathbb{R}}\phi^{-\mathbf{X}}u_{\xi\xi}^{R}w'u_{\xi}^{S}d\xi\\
				&-\frac{\tau}{2}\int_{\mathbb{R}}\left(r^{-\mathbf{X}}\right)^{2} w^{\prime} u_{\xi}^{S} d \xi,
			    \end{aligned}
			\end{equation*}
			\begin{equation*}
				\begin{aligned}
				J^{\text {good }}(t)&:=\int_{\mathbb{R}} \left(r^{-\mathbf{X}}\right)^{2} w d\xi-3\int_{\mathbb{R}}\left(\phi^{-\mathbf{X}}\right)^{2}\left(u^{S}\right)^{2} w^{\prime} u_{\xi}^{S} d \xi+3\int_{\mathbb{R}}\left(\phi^{-\mathbf{X}}\right)^{2}\left(u^{R}-u_{m}\right) w u_{\xi}^{S} d \xi\\
				&+3 \int_{\mathbb{R}}\left(\phi^{-\mathbf{X}}\right)^{2} u^{R} w u_{\xi}^{R} d \xi-\frac{3}{2} \int_{\mathbb{R}}\left(\phi^{-\mathbf{X}}\right)^{2}\left(u^{R}-u_{m}\right)^{2} w^{\prime} u_{\xi}^{S} d \xi-\frac{3}{4} \int_{\mathbb{R}}\left(\phi^{-\mathbf{X}}\right)^{4} w^{\prime} u_{\xi}^{S} d \xi,
		    	\end{aligned}
	     	\end{equation*}
	        \begin{equation*}
	            \begin{aligned}
				J^{\text {bad }}(t)&:=3 \int_{\mathbb{R}}\left(\phi^{-\mathbf{X}}\right)^{2} w u^{S} u_{\xi}^{S} d \xi+3 \int_{\mathbb{R}}\left(\phi^{-\mathbf{X}}\right)^{2} w\left(u^{S}-u_{m}\right) u_{\xi}^{R} d \xi\\
				&-3\int_{\mathbb{R}}\left(\phi^{-\mathbf{X}}\right)^{2}\left(u^{R}-u_{m}\right) u^{S} w^{\prime} u_{\xi}^{S} d \xi+\sigma \int_{\mathbb{R}}\left(\phi^{-\mathbf{X}}\right)^{2} w^{\prime} u_{\xi}^{S} d \xi\\
				&-\frac{1}{2} \int_{\mathbb{R}}\left(\phi^{-\mathbf{X}}\right)^{2} \left(u_{\xi}^{S}\right)^{2} w^{\prime \prime} d \xi+\int_{\mathbb{R}}\left(\phi^{-\mathbf{X}}\right)^{3} w\left(u_{\xi}^{S}+u_{\xi}^{R}\right) d \xi\\
				&-2\int_{\mathbb{R}}\left(\phi^{-\mathbf{X}}\right)^{3}\left(u^{S}+u^{R}-u_{m}\right) w^{\prime} u_{\xi}^{S} d \xi+\frac{\tau\sigma}{2}\int_{\mathbb{R}}\left(r^{-\mathbf{X}}\right)^{2}w'u_{\xi}^{S}d\xi+\tau\sigma\int_{\mathbb{R}}r^{-\mathbf{X}}u_{\xi\xi}^Rwd\xi\\
				&+6\tau\int_{\mathbb{R}}r^{-\mathbf{X}}u^R(u_{\xi}^R)^2wd\xi+3\tau\int_{\mathbb{R}}r^{-\mathbf{X}}(u^R)^2u_{\xi\xi}^Rwd\xi+R_{0}^{\tau}(t), \\
				\mathbf{Z}(t)&:=- \int_{\mathbb{R}} F \phi^{-\mathbf{X}} w d\xi+\tau\int_{\mathbb{R}}r^{-\mathbf{X}}Fw'u_{\xi}^{S}d\xi+\tau\frac{d}{dt}\int_{\mathbb{R}}r^{-\mathbf{X}}\phi^{-\mathbf{X}}w'u_{\xi}^{S}d\xi,
			\end{aligned}
		\end{equation*}
		where
		\begin{equation}
			\begin{aligned}
				R_{0}^{\tau}(t):=&- \tau\sigma\int_{\mathbb{R}}r^{-\mathbf{X}}\phi^{-\mathbf{X}}w''\left(u_{\xi}^{S}\right)^{2}d\xi-\tau\sigma\int_{\mathbb{R}}r^{-\mathbf{X}}\phi^{-\mathbf{X}}w'u^S_{\xi\xi}d\xi\\
				&-3\tau\int_{\mathbb{R}}r^{-\mathbf{X}}(\phi^{-\mathbf{X}})^{2}\phi_{\xi}^{-\mathbf{X}}w'u_{\xi}^{S}d\xi -6\tau\int_{\mathbb{R}}(u^{S}+u^{R}-u_{m})r^{-\mathbf{X}}\phi^{-\mathbf{X}}\phi_{\xi}^{-\mathbf{X}}w'u_{\xi}^{S}d\xi\\
				&-3\tau\int_{\mathbb{R}}(u^{S}+u^{R}-u_{m})^{2}r^{-\mathbf{X}}\phi_{\xi}^{-\mathbf{X}}w'u_{\xi}^{S}d\xi-3\tau\int_{\mathbb{R}}(u_{\xi}^{S}+u_{\xi}^{R})r^{-\mathbf{X}}(\phi^{-\mathbf{X}})^{2}w'u_{\xi}^{S}d\xi\\
				&-6\tau\int_{\mathbb{R}}(u^{S}+u^{R}-u_{m})(u_{\xi}^{S}+u_{\xi}^{R})r^{-\mathbf{X}}\phi^{-\mathbf{X}}w'u_{\xi}^{S}d\xi-\frac{\tau}{2}\int_{\mathbb{R}}(r^{-\mathbf{X}})w''(u_{\xi}^{S})^{2}d\xi\\
				&-\frac{\tau}{2}\int_{\mathbb{R}}(r^{-\mathbf{X}})^{2}w'u_{\xi\xi}^Sd\xi+\frac{1}{2}\int_{\mathbb{R}}(\phi^{-\mathbf{X}})^{2}w'R_{1}^{\tau}u_{\xi}^{S}d\xi-\tau\sigma\int_{\mathbb{R}}u_{\xi\xi}^R\phi^{-\mathbf{X}}w'u_{\xi}^Sd\xi\\
				&- 6\tau\int_{\mathbb{R}}(u^{R})(u_{\xi}^{R})^{2}\phi^{-\mathbf{X}}w'u_{\xi}^{S}d\xi-3\tau\int_{\mathbb{R}}(u^{R})^{2}u_{\xi\xi}^{R}\phi^{-\mathbf{X}}w'u_{\xi}^{S}d\xi,
			\end{aligned}
		\end{equation}
		and
		\begin{equation}
			\begin{aligned}
				&R_{1}^{\tau}(u^{S})
				:=\frac{9\tau\sigma\left((u^{S})^{2}-u_{m}^{2}\right)^{2}\left(1+3\tau\sigma\left((u^{S})^{2}-u_{m}^{2}\right)\right)+6\tau\sigma u^{S}(u^{S}+2u_{m})(u_{m}-u^{S})}{\left(1+3\tau\sigma\left((u^{S})^{2}-u_{m}^{2}\right)\right)^{2}},
			\end{aligned}
		\end{equation}
	\end{lemma}
	\begin{remark}
		Given the conditions $w'\leq0,\ w''\geq0,\ u_{\xi}^{S}>0,\ u_{\xi}^{R}>0,\ u^{R}>u_{m},$ thus $J^{\text{good}}(t)$ consists of good terms, while $J^{\text{bad}}(t)$ consists of bad terms and $\mathbf{Z}(t)$ can be controlled by the wave interaction estimates and the properties of approximate rarefaction wave.
	\end{remark}
	\begin{proof}
		By the direct calculation, we can get
		\begin{align*}
			\frac{d}{dt} \int_{\mathbb{R}} w^{\mathbf{X}} \eta(U|\tilde{U}^{\mathbf{X}}) d\xi=\dot{\mathbf{X}}(t)\int_{\mathbb{R}}w_{\xi}^{{\mathbf{X}}}\eta\left(U|\tilde{U}^{\mathbf{X}}\right)d\xi+\int_{\mathbb{R}}w^{\mathbf{X}}\left(\phi_{t}\phi+\tau r_{t}r\right)d\xi.
		\end{align*}	
		Changing variable $\xi\to\xi-\mathbf{X}(t), $ we have
		\small
		\begin{align*}
			&\frac{d}{dt} \int_{\mathbb{R}} w \eta(U^{-\mathbf{X}}|\tilde{U}) d\xi
			=\dot{\mathbf{X}}(t)\int_{\mathbb{R}}w_{\xi}\eta\left(U^{-\mathbf{X}}|\tilde{U}\right)d\xi+\int_{\mathbb{R}}w\left(\phi_{t}^{-\mathbf{X}}\phi^{-\mathbf{X}}+\tau r_{t}^{-\mathbf{X}}r^{-\mathbf{X}}\right)d\xi\\
			=&\dot{\mathbf{X}}(t)\int_{\mathbb{R}}\left(\frac{\left(\phi^{-\mathbf{X}}\right)^{2}}{2}+\frac{\left(r^{-\mathbf{X}}\right)^{2}}{2}\right)w'u_{\xi}^{S}d\xi-\int_{\mathbb{R}}F\phi^{-\mathbf{X}}wd\xi\\
			&+\int_{\mathbb{R}}\left[\sigma \phi_{\xi}^{-\mathbf{X}}-\left(f\left(\phi^{-\mathbf{X}}+\tilde{u}\right)-f\left(\tilde{u}\right)\right)_{\xi}-\dot{\mathbf{X}}(t)\left(u_{\xi}^{S}+u_{\xi}^{R}\right)+r_{\xi}^{-\mathbf{X}}\right]\phi^{-\mathbf{X}} wd\xi\\
			&+\int_{\mathbb{R}}\left[\tau\sigma r_{\xi}^{-\mathbf{X}}-r^{-\mathbf{X}}+\phi_{\xi}^{-\mathbf{X}}-\tau \dot{\mathbf{X}}(t) \tilde{q}_{\xi}-\tau \left(u^{R}\right)_{t\xi}\right]r^{-\mathbf{X}}wd\xi\\
			=&\dot{\mathbf{X}}(t)\int_{\mathbb{R}}\left(\frac{\left(\phi^{-\mathbf{X}}\right)^{2}}{2}+\frac{\left(r^{-\mathbf{X}}\right)^{2}}{2}\right)w'u_{\xi}^{S}d\xi-\int_{\mathbb{R}}F\phi^{-\mathbf{X}}wd\xi+\sigma\int_{\mathbb{R}}\phi_{\xi}^{-\mathbf{X}}\phi^{-\mathbf{X}}wd\xi\\
			&-\int_{\mathbb{R}}\left(f\left(\phi^{-\mathbf{X}}+\tilde{u}\right)-f\left(\tilde{u}\right)\right)_{\xi}\phi_{\xi}^{-\mathbf{X}}\phi^{-\mathbf{X}}wd\xi-\dot{{\mathbf{X}}}(t)\int_{\mathbb{R}}\left(u_{\xi}^{S}+u_{\xi}^{R}\right)\phi^{-\mathbf{X}}wd\xi+\tau\sigma\int_{\mathbb{R}}r_{\xi}^{-\mathbf{X}}r^{-\mathbf{X}}wd\xi\\
			&-\int_{\mathbb{R}}\left(r^{-\mathbf{X}}\right)^2wd\xi-\tau\dot{\mathbf{X}}(t)\int_{\mathbb{R}}r^{-\mathbf{X}}w\tilde{q}_{\xi}d\xi-\tau\int_{\mathbb{R}}r^{-\mathbf{X}}\left(u^{R}\right)_{t\xi}wd\xi+\int_{\mathbb{R}}\left(r^{-\mathbf{X}}\phi^{-\mathbf{X}}\right)_{\xi}wd\xi.
		\end{align*}
		Further calculations yield
		\begin{align*}
			\sigma\int_{\mathbb{R}}\phi_{\xi}^{-\mathbf{X}}\phi^{-\mathbf{X}}wd\xi=-\frac{\sigma}{2}\int_{\mathbb{R}}\left(\phi^{-\mathbf{X}}\right)^2w'u_{\xi}^{S}d\xi,
		\end{align*}
		and
		\begin{align*}
			\tau\sigma\int_{\mathbb{R}}r_{\xi}^{-\mathbf{X}}r^{-\mathbf{X}}wd\xi=-\frac{\tau\sigma}{2}\int_{\mathbb{R}}\left(r^{-\mathbf{X}}\right)^2w'u_{\xi}^{S}d\xi
		\end{align*}
		From \cite{ref16}, we can get
		\begin{equation} 
			\begin{aligned}
				&-\int_{\mathbb{R}} (f(\phi^{-\mathbf{X}} + \tilde{u}) - f(\tilde{u})) \phi^{-\mathbf{X}} w d\xi\\
				=&-3\int_{\mathbb{R}}(u^S + u^R - u_m)(u_{\xi}^S + u_{\xi}^R)(\phi^{-\mathbf{X}})^2wd\xi + \frac{3}{2}\int_{\mathbb{R}}(\phi^{-\mathbf{X}})^2(u^R - u_m)^2w'u_{\xi}^Sd\xi\\
				&+\frac{3}{2}\int_{\mathbb{R}}(\phi^{-\mathbf{X}})^2(u_{\xi}^S)^2w'u_{\xi}^Sd\xi + 3\int_{\mathbb{R}}(\phi^{-\mathbf{X}})^2(u^R - u_m)u^Sw'u_{\xi}^Sd\xi\\
				&-\int_{\mathbb{R}}(\phi^{-\mathbf{X}})^3w(u_{\xi}^S + u_{\xi}^R)d\xi + 2\int_{\mathbb{R}}(\phi^{-\mathbf{X}})^3(u^S + u^R - u_m)w'u_{\xi}^Sd\xi+\frac{3}{4}\int_{\mathbb{R}}(\phi^{-\mathbf{X}})^4w'u_{\xi}^Sd\xi. 
			\end{aligned}
		\end{equation}
		and
		\begin{align}\label{11}
			&\int_{\mathbb{R}}\left(r^{-\mathbf{X}}\phi^{-\mathbf{X}}\right)_{\xi}wd\xi=-\int_{\mathbb{R}}r^{-\mathbf{X}}\phi^{-\mathbf{X}}w'u_{\xi}^Sd\xi\nonumber\\
			=&-\int_{\mathbb{R}}\left(-\tau r_{t}^{-\mathbf{X}}+\tau\sigma r_{\xi}^{-\mathbf{X}}+\phi_{\xi}^{-\mathbf{X}}-\tau \dot{\mathbf{X}}(t) \tilde{q}_{\xi}-\tau \left(u^{R}\right)_{t\xi}\right)\phi^{-\mathbf{X}}w'u_{\xi}^Sd\xi\nonumber\\
			=&\tau\int_{\mathbb{R}}r_{t}^{-\mathbf{X}}\phi^{-\mathbf{X}}w'u_{\xi}^Sd\xi-\tau\sigma\int_{\mathbb{R}}r_{\xi}^{-\mathbf{X}}\phi^{-\mathbf{X}}w'u_{\xi}^Sd\xi-\int_{\mathbb{R}}\phi_{\xi}^{-\mathbf{X}}\phi^{-\mathbf{X}}w'u_{\xi}^Sd\xi\nonumber\\
			&+\tau\dot{X}(t)\int_{\mathbb{R}}\tilde{q}_{\xi}\phi^{-\mathbf{X}}w'u_{\xi}^Sd\xi+\tau\int_{\mathbb{R}}\left(u^{R}\right)_{t\xi}\phi^{-\mathbf{X}}w'u_{\xi}^Sd\xi.
		\end{align}
		Now we calculate the righthand side of $(\ref{11})$ term by term.
		\begin{align*}
			&\tau\int_{\mathbb{R}}r_{t}^{-\mathbf{X}}\phi^{-\mathbf{X}}w'u_{\xi}^Sd\xi=\tau\frac{d}{dt}\int_{\mathbb{R}}r^{-\mathbf{X}}\phi^{-\mathbf{X}}w'u_{\xi}^{S}d\xi-\tau\int_{\mathbb{R}}r^{-\mathbf{X}}\phi_{t}^{-\mathbf{X}}w'u_{\xi}^Sd\xi\\
			=&\tau\frac{d}{dt}\int_{\mathbb{R}}r^{-\mathbf{X}}\phi^{-\mathbf{X}}w'u_{\xi}^{S}d\xi\\
			&-\tau\int_{\mathbb{R}}r^{-\mathbf{X}}\left[\sigma\phi_{\xi}^{-\mathbf{X}}-\left(f(\phi^{-\mathbf{X}}+\tilde{u})-f(\tilde{u})\right)_{\xi}-\mathbf{X}(t)(u_{\xi}^S+u_{\xi}^R)+r_{\xi}^{-\mathbf{X}}-F\right]w'u_{\xi}^Sd\xi\\
			=&\tau\frac{d}{dt}\int_{\mathbb{R}}r^{-\mathbf{X}}\phi^{-\mathbf{X}}w'u_{\xi}^{S}d\xi-\tau\sigma\int_{\mathbb{R}}r^{-\mathbf{X}}\phi_{\xi}^{-\mathbf{X}}w'u_{\xi}^Sd\xi+\tau\int_{\mathbb{R}}r^{-\mathbf{X}}\left(f(\phi^{-\mathbf{X}}+\tilde{u})-f(\tilde{u})\right)_{\xi}w'u_{\xi}^Sd\xi\\
			&+\tau \dot{X}(t)\int_{\mathbb{R}}r^{-\mathbf{X}}\left(u_{\xi}^{S}+u_{\xi}^{R}\right)w'u_{\xi}^{S}d\xi-\tau\int_{\mathbb{R}}r^{-\mathbf{X}}r_{\xi}^{-\mathbf{X}}w'u_{\xi}^Sd\xi+\tau\int_{\mathbb{R}}r^{-\mathbf{X}}Fw'u_{\xi}^Sd\xi\\
			=&\tau\frac{d}{dt}\int_{\mathbb{R}}r^{-\mathbf{X}}\phi^{-\mathbf{X}}w'u_{\xi}^{S}d\xi-\tau\sigma\int_{\mathbb{R}}r^{-\mathbf{X}}\phi_{\xi}^{-\mathbf{X}}w'u_{\xi}^Sd\xi+3\tau\int_{\mathbb{R}}r^{-\mathbf{X}}(\phi^{-\mathbf{X}})^2\phi_{\xi}^{-\mathbf{X}}w'u_{\xi}^Sd\xi\\
			&+\tau\int_{\mathbb{R}}r^{-\mathbf{X}}\left[6\tilde{u}\phi^{-\mathbf{X}}\phi_{\xi}^{-\mathbf{X}}+3(\tilde{u})^{2}\phi_{\xi}^{-\mathbf{X}}+3\tilde{u}_{\xi}(\phi^{-\mathbf{X}})^{2}+6\tilde{u}\tilde{u}_{\xi}\phi^{-\mathbf{X}}\right]w'u_{\xi}^{S}d\xi\\
			&+\tau \dot{X}(t)\int_{\mathbb{R}}r^{-\mathbf{X}}\left(u_{\xi}^{S}+u_{\xi}^{R}\right)w'u_{\xi}^{S}d\xi+\tau\int_{\mathbb{R}}(r^{-\mathbf{X}})^{2}\left(w''(u_{\xi}^{S})^{2}+w'u_{\xi\xi}^{S}\right)d\xi+\tau\int_{\mathbb{R}}r^{-\mathbf{X}}Fw'u_{\xi}^Sd\xi\\
			=&\tau\frac{d}{dt}\int_{\mathbb{R}}r^{-\mathbf{X}}\phi^{-\mathbf{X}}w'u_{\xi}^{S}d\xi-\tau\sigma\int_{\mathbb{R}}r^{-\mathbf{X}}\phi_{\xi}^{-\mathbf{X}}w'u_{\xi}^Sd\xi+3\tau\int_{\mathbb{R}}r^{-\mathbf{X}}(\phi^{-\mathbf{X}})^2\phi_{\xi}^{-\mathbf{X}}w'u_{\xi}^Sd\xi\\
			&+6\tau\int_{\mathbb{R}}(u^S+u^R-u_{m})r^{-\mathbf{X}}\phi^{-\mathbf{X}}\phi_{\xi}^{-\mathbf{X}}w'u_{\xi}^Sd\xi+3\tau\int_{\mathbb{R}}(u^S+u^R-u_{m})^{2}r^{-\mathbf{X}}\phi_{\xi}^{-\mathbf{X}}w'u_{\xi}^Sd\xi\\
			&+3\tau\int_{\mathbb{R}}(u_{\xi}^S+u_{\xi}^R)r^{-\mathbf{X}}(\phi^{-\mathbf{X}})^2w'u_{\xi}^Sd\xi+6\tau\int_{\mathbb{R}}(u^S+u^R-u_{m})(u_{\xi}^S+u_{\xi}^R)r^{-\mathbf{X}}\phi^{-\mathbf{X}}w'u_{\xi}^Sd\xi\\
			&+\tau\int_{\mathbb{R}}r^{-\mathbf{X}}Fw'u_{\xi}^Sd\xi+\tau \dot{X}(t)\int_{\mathbb{R}}r^{-\mathbf{X}}\left(u_{\xi}^{S}+u_{\xi}^{R}\right)w'u_{\xi}^{S}d\xi+\tau\int_{\mathbb{R}}(r^{-\mathbf{X}})^{2}\left(w''(u_{\xi}^{S})^{2}+w'u_{\xi\xi}^{S}\right)d\xi,
		\end{align*}
		and
		\begin{align*}
			&-\tau\sigma\int_{\mathbb{R}}r^{-\mathbf{X}}\phi_{\xi}^{-\mathbf{X}}w'u_{\xi}^Sd\xi-\tau\sigma\int_{\mathbb{R}}r_{\xi}^{-\mathbf{X}}\phi^{-\mathbf{X}}w'u_{\xi}^Sd\xi
			=-\tau\sigma\int_{\mathbb{R}}(r^{-\mathbf{X}}\phi^{-\mathbf{X}})_{\xi}w'u_{\xi}^Sd\xi\\
			=&\tau\sigma\int_{\mathbb{R}}r^{-\mathbf{X}}\phi^{-\mathbf{X}}\left[w''(u_{\xi})^2+w'u_{\xi\xi}^S\right]d\xi,
		\end{align*}
		Notice that $u_{\xi}^S=\frac{(u^{S}-u_{-})(u^{S}-u_{m})^{2}}{1+3\tau \sigma\left[(u^{S})^{2}-u_{m}^{2}\right]},$ then 
		\begin{align*}
			&u_{\xi\xi}^S=-3(u_{m}^2-(u^S)^2)u_{\xi}^S\\
			&-\frac{9\tau\sigma\left((u^{S})^{2}-u_{m}^{2}\right)^{2}\left(1+3\tau\sigma\left((u^{S})^{2}-u_{m}^{2}\right)\right)+6\tau\sigma u^{S}(u^{S}+2u_{m})(u_{m}-u^{S})}{\left(1+3\tau\sigma\left((u^{S})^{2}-u_{m}^{2}\right)\right)^{2}}u_{\xi}^S\\
			&=:-3(u_{m}^2-(u^S)^2)u_{\xi}^S-R_{1}^{\tau}(u^{S})u_{\xi}^S.
		\end{align*}
		Thus we have
		\begin{align*}
			&-\int_{\mathbb{R}}\phi_{\xi}^{-\mathbf{X}}\phi^{-\mathbf{X}}w'u_{\xi}^Sd\xi
			=\frac{1}{2}\int_{\mathbb{R}}(\phi^{-\mathbf{X}})^2\left[w''(u_{\xi})^2+w'u_{\xi\xi}^S\right]d\xi\\
			=&\frac{1}{2}\int_{\mathbb{R}}(\phi^{-\mathbf{X}})^2w''(u_{\xi})^2d\xi-\frac{3}{2}\int_{\mathbb{R}}(\phi^{-\mathbf{X}})^2(u_{m}^2-(u^S)^2)w'u_{\xi}^Sd\xi-\frac{1}{2}\int_{\mathbb{R}}(\phi^{-\mathbf{X}})^2R_{1}^{\tau}(u^{S})w'u_{\xi}^Sd\xi,
		\end{align*}
		and since $(u^R)_{t}=f(u^{R})_{x}$, we get $(u^R)_{t}=\sigma u_{\xi}^R+f(u^{R})_{\xi}.$
		\begin{align*}
			&\tau\int_{\mathbb{R}}\left(u^{R}\right)_{t\xi}\phi^{-\mathbf{X}}w'u_{\xi}^Sd\xi
			=\tau\sigma\int_{\mathbb{R}}u_{\xi\xi}^R\phi^{-\mathbf{X}}w'u_{\xi}^Sd\xi+\tau\int_{\mathbb{R}}f(u^{R})_{\xi}\phi^{-\mathbf{X}}w'u_{\xi}^Sd\xi\\
			=&\tau\sigma\int_{\mathbb{R}}u_{\xi\xi}^R\phi^{-\mathbf{X}}w'u_{\xi}^Sd\xi+\tau\int_{\mathbb{R}}\left[6u^R(u_{\xi}^R)^2+3(u^R)^2u_{\xi\xi}^R\right]\phi^{-\mathbf{X}}w'u_{\xi}^Sd\xi,
		\end{align*}
		Substitute them into $(\ref{11})$, then
		\begin{equation}\label{hhx}
			\begin{aligned}
				&-\int_{\mathbb{R}}\left(r^{-\mathbf{X}}\phi^{-\mathbf{X}}\right)_{\xi}wd\xi\\
				:=&-\tau\frac{d}{dt}\int_{\mathbb{R}}r^{-\mathbf{X}}\phi^{-\mathbf{X}}w'u_{\xi}^{S}d\xi - \tau\sigma\int_{\mathbb{R}}r^{-\mathbf{X}}\phi^{-\mathbf{X}}w''\left(u_{\xi}^{S}\right)^{2}d\xi-\tau\sigma\int_{\mathbb{R}}r^{-\mathbf{X}}\phi^{-\mathbf{X}}w'u_{\xi\xi}d\xi\\
				&-3\tau\int_{\mathbb{R}}r^{-\mathbf{X}}(\phi^{-\mathbf{X}})^{2}\phi_{\xi}^{-\mathbf{X}}w'u_{\xi}^{S}d\xi-6\tau\int_{\mathbb{R}}(u^{S}+u^{R}-u_{m})r^{-\mathbf{X}}\phi^{-\mathbf{X}}\phi_{\xi}^{-\mathbf{X}}w'u_{\xi}^{S}d\xi\\
				&-3\tau\int_{\mathbb{R}}(u^{S}+u^{R}-u_{m})^{2}r^{-\mathbf{X}}\phi_{\xi}^{-\mathbf{X}}w'u_{\xi}^{S}d\xi-3\tau\int_{\mathbb{R}}(u_{\xi}^{S}+u_{\xi}^{R})r^{-\mathbf{X}}(\phi^{-\mathbf{X}})^{2}w'u_{\xi}^{S}d\xi\\
				&-6\tau\int_{\mathbb{R}}(u^{S}+u^{R}-u_{m})(u_{\xi}^{S}+u_{\xi}^{R})r^{-\mathbf{X}}\phi^{-\mathbf{X}}w'u_{\xi}^{S}d\xi-\tau\dot{\mathbf{X}}(t)\int_{\mathbb{R}}r^{-\mathbf{X}}(u_{\xi}^{S}+u_{\xi}^{R})w'u_{\xi}^{S}d\xi\\ &-\frac{\tau}{2}\int_{\mathbb{R}}(r^{-\mathbf{X}})w''(u_{\xi}^{S})^{2}d\xi-\frac{\tau}{2}\int_{\mathbb{R}}(r^{-\mathbf{X}})^{2}w'u_{\xi\xi}^Sd\xi -\tau\int_{\mathbb{R}}r^{-\mathbf{X}}F\,w'u_{\xi}^{S}d\xi-\frac{1}{2}\int_{\mathbb{R}}(\phi^{-\mathbf{X}})^{2}w''(u_{\xi}^S)^{2}d\xi\\
				&-\frac{3}{2}\int_{\mathbb{R}}(\phi^{-\mathbf{X}})^{2}w'(u^{S})^{2}u_{\xi}^{S}d\xi+\frac{\sigma}{2}\int_{\mathbb{R}}(\phi^{-\mathbf{X}})^{2}w'u_{\xi}^{S}d\xi+\frac{1}{2}\int_{\mathbb{R}}(\phi^{-\mathbf{X}})^{2}w'R_{1}^{\tau}u_{\xi}^{S}d\xi\\
				&-\tau\dot{\mathbf{X}}(t)\int_{\mathbb{R}}q_{\xi}^{S}\phi^{-\mathbf{X}}w'u_{\xi}^{S}d\xi-\tau\dot{{\mathbf{X}}}(t)\int_{\mathbb{R}}u_{\xi\xi}^{R}\phi^{-\mathbf{X}}w'u_{\xi}^{S}d\xi -\tau\sigma\int_{\mathbb{R}}u_{\xi\xi}^R\phi^{-\mathbf{X}}w'u_{\xi}^Sd\xi\\
				&- 6\tau\int_{\mathbb{R}}(u^{R})(u_{\xi}^{R})^{2}\phi^{-\mathbf{X}}w'u_{\xi}^{S}d\xi-3\tau\int_{\mathbb{R}}(u^{R})^{2}u_{\xi\xi}^{R}\phi^{-\mathbf{X}}w'u_{\xi}^{S}d\xi,
			\end{aligned}
		\end{equation}
		and
		\begin{align*}
			&-\tau\int_{\mathbb{R}}r^{-\mathbf{X}}\left(u^{R}\right)_{t\xi}wd\xi
			=-\tau\int_{\mathbb{R}}r^{-\mathbf{X}}\left(\sigma u_{\xi\xi}^{R}+f(u^R)_{\xi}\right)wd\xi\nonumber\\
			=&-\tau\sigma\int_{\mathbb{R}}r^{-\mathbf{X}}u_{\xi\xi}^Rwd\xi-\tau\int_{\mathbb{R}}r^{-\mathbf{X}}\left[6u^R(u_{\xi}^R)^2+3(u^R)^2u_{\xi\xi}^R\right]wd\xi\nonumber\\
			=&-\tau\sigma\int_{\mathbb{R}}r^{-\mathbf{X}}u_{\xi\xi}^Rwd\xi-6\tau\int_{\mathbb{R}}r^{-\mathbf{X}}u^R(u_{\xi}^R)^2wd\xi-3\tau\int_{\mathbb{R}}r^{-\mathbf{X}}(u^R)^2u_{\xi\xi}^Rwd\xi\nonumber.
		\end{align*}
		So, we finish the proof.
	\end{proof}
	Then we rewrite
		\begin{align}\label{lem3}
			\dot{\mathbf{X}}(t)\mathbf{Y}(t) + J^{good}(t) + J^{bad}(t)=\mathbf{G}^S(t)+\mathbf{G}^R(t)+\mathbf{G}^{SR}(t)+\mathbf{N}(t)+\mathbf{J}(t),
		\end{align}
		where 
		\begin{align}
			\mathbf{G}^{S}(t) :=&\dot{\mathbf{X}}(t) \int_{\mathbb{R}} \phi^{-\mathbf{X}} w u_{\xi}^{S} d\xi+ \int_{\mathbb{R}}  (r^{-\mathbf{X}})^2w d\xi - \frac{3}{4} \int_{\mathbb{R}} (\phi^{-\mathbf{X}})^4 w' u_{\xi}^{S} d\xi\nonumber\\
			&+ \int_{\mathbb{R}} (\phi^{-\mathbf{X}})^2 u_{\xi}^{S} \left[\sigma w' - 3(u^{S})^2 w' + 3u^{S} w - \frac{w''}{2} u_{\xi}^{S}\right] d\xi
			,\label{GS}\\
			\mathbf{G}^{R}(t):=&3\int_{\mathbb{R}}(\phi^{-\mathbf{X}})^{2}wu^{R}u_{\xi}^{R}d\xi, \label{GR}\\
			\mathbf{G}^{SR}(t) := &3 \int_{\mathbb{R}} (\phi^{-X})^2 w(u^R - u_m) u_{\xi}^S d\xi - \frac{3}{2} \int_{\mathbb{R}} (\phi^{-X})^2 (u^R - u_m)^2 w' u_{\xi}^S d\xi,\label{GSR}\\
			\mathbf{N}(t):=&\dot{\mathbf{X}}(t) \int_{\mathbb{R}} \phi^{-\mathbf{X}} w u_{\xi}^{R} d\xi-\frac{\dot{\mathbf{X}}(t)}{2} \int_{\mathbb{R}}\left(\phi^{-\mathbf{X}}\right)^{2} w^{\prime} u_{\xi}^{S} d \xi-\tau\dot{\mathbf{X}}(t)\int_{\mathbb{R}}\phi^{-\mathbf{X}}q_{\xi}^{S}w'u_{\xi}^{S}d\xi\nonumber\\
			&-\tau\dot{\mathbf{X}}(t) \int_{\mathbb{R}}r^{-\mathbf{X}}\left(u_{\xi}^{S}+u_{\xi}^{R}\right)w'u_{\xi}^{S}d\xi+\tau\dot{\mathbf{X}}(t)\int_{\mathbb{R}}r^{-\mathbf{X}}w\tilde{q}_{\xi}d\xi\nonumber\\
			&-\tau\dot{\mathbf{X}}(t)\int_{\mathbb{R}}\phi^{-\mathbf{X}}u_{\xi\xi}^{R}w'u_{\xi}^{S}d\xi-\frac{\tau\dot{\mathbf{X}}(t)}{2}\int_{\mathbb{R}}\left(r^{-\mathbf{X}}\right)^{2} w^{\prime} u_{\xi}^{S} d \xi+R_0^{\tau}(t)\nonumber\\
			&-3\int_{\mathbb{R}}(\phi^{-\mathbf{X}})^2(u^R-u_{m})u^Sw'u_{\xi}^Sd\xi+\int_{\mathbb{R}}(\phi^{-\mathbf{X}})^3w(u_{\xi}^S+u_{\xi}^R)d\xi\nonumber\\
			&-2\int_{\mathbb{R}}(\phi^{-\mathbf{X}})^3(u^S+u^R-u_{m})w'u_{\xi}^Sd\xi+\frac{\tau\sigma}{2}\int_{\mathbb{R}}(r^{-\mathbf{X}})^2w'u_{\xi}^Sd\xi\nonumber\\
			&+\tau\sigma\int_{\mathbb{R}}r^{-\mathbf{X}}u_{\xi\xi}^Rwd\xi+6\tau\int_{\mathbb{R}}r^{-\mathbf{X}}u^R(u_{\xi}^R)^2wd\xi+3\tau\int_{\mathbb{R}}r^{-\mathbf{X}}(u^R)^2u_{\xi\xi}^Rwd\xi,
		\end{align}
		and
		\begin{align}\label{j}
			\mathbf{J}(t):=&3 \int_{\mathbb{R}} (\phi^{-\mathbf{X}})^2 (u^S - u_m) w u_\xi^R d\xi.
		\end{align}
			It is clear that $\mathbf{G}^R(t)$ and $\mathbf{G}^{SR}(t)$ are good terms. The fact that $\mathbf{G}^S(t)$ are good terms follows from Lemma $\ref{lemgs}$, provided we choose the weight function $w$ defined in $(\ref{WF})$ together with the time-dependent shift $\mathbf{X}(t)$ given in $(\ref{shift})$. On the other hand, $\mathbf{N}(t)$ have the bad term, which can be controlled using Lemma $\ref{lem6}$ under the assumption of sufficiently small initial perturbations and rarefaction wave strength. Meanwhile, both $\mathbf{N}(t)$ and $\mathbf{Z}(t)$ can be handled via the wave interaction estimates established in Lemma $\ref{WIE}$ along with the decay properties of the approximate rarefaction wave.
	To prove Lemma $\ref{lem2}$, we begin by estimating the terms $\mathbf{G}^S(t),\mathbf{N}(t),\mathbf{J}(t)$ and $\mathbf{Z}(t)$ as follow.\par
	First, we obtain the estimate for $\mathbf{G}^S(t)$. The proof follows an argument similar to that in \cite{ref16}.
	\begin{lemma}\label{lemgs}
		The following estimate holds for $\mathbf{G}^{S}(t)$:
		\begin{equation} 
			\begin{aligned}
				\mathbf{G}^{S}(t) \geq &\int_{\mathbb{R}}(r^{-\mathbf{X}})^2wd\xi-\left(\frac{5}{6}+\frac{9\tau\sigma u_m^2}{8}\right)\int_{\mathbb{R}} (\phi_{\xi}^{-\mathbf{X}})^{2}w d\xi \\
				&+ \frac{4}{5} u_{m}^{3} \int_{\mathbb{R}} (\phi^{-\mathbf{X}})^{2} u_{\xi}^{S} d\xi+ \frac{25}{64} u_{m}^{2} |\dot{\mathbf{X}}(t)|^{2} + \frac{3}{4} \int_{\mathbb{R}} (\phi^{-\mathbf{x}})^{4} |w'| u_{\xi}^{S} d\xi. 
			\end{aligned}
		\end{equation}
	\end{lemma}
	Next, the estimates of $\mathbf{N}(t)$ is given by the following lemma.
	\begin{lemma}\label{lem6}
		There exists a constant $C>0$ such that
		\begin{align*}
			\left|\mathbf{N}(t)\right|\leq&\left(\frac{u_m^2}{64}+\frac{5\tau u_m^2}{64}\right)|\dot{\mathbf{X}}(t)|^2+\left(C\delta_R+\|\phi^{-\mathbf{X}}\|_{L^{\infty}(\mathbb{R})}\right)\int_{\mathbb{R}}(\phi^{-\mathbf{X}})^2wu_{\xi}^Rd\xi\\
			&+\tau\left(M_1+C\|\phi^{-\mathbf{X}}\|_{L^{\infty}(\mathbb{R})}\right)\int_{\mathbb{R}}(\phi^{-\mathbf{X}})^2u_{\xi}^Sd\xi\\
			&+\tau\left(M_2+C\delta_R^{\frac{1}{3}}+C\|\phi^{-\mathbf{X}}\|_{L^{\infty}(\mathbb{R})}\right)\int_{\mathbb{R}}(r^{-\mathbf{X}})^2wd\xi\\
			&+\tau\left(Cu_m^4+C\delta_R+C\|\phi^{-\mathbf{X}}\|_{L^{\infty}(\mathbb{R})}\right)\int_{\mathbb{R}}(\phi_{\xi}^{-\mathbf{X}})^2wd\xi+C\delta_R^{\frac{1}{3}}(1+t)^{-\frac{4}{3}},
		\end{align*}
		where 
		\begin{align*}
			&M_1=\frac{5\sigma}{4}u_m+126u_m^3+C u_m^6+C u_m^7+Cu_m^{11},\\
			&M_2=Cu_m^4+C u_m^5+Cu_m^8.
		\end{align*}
	\end{lemma}
	\begin{proof}
		We rewrite $\mathbf{N}(t)$ as:
		\begin{align*}
			\mathbf{N}(t):=&\dot{\mathbf{X}}(t) \int_{\mathbb{R}} \phi^{-\mathbf{X}} w u_{\xi}^{R} d\xi-\frac{\dot{\mathbf{X}}(t)}{2} \int_{\mathbb{R}}\left(\phi^{-\mathbf{X}}\right)^{2} w^{\prime} u_{\xi}^{S} d \xi\nonumber\\
			&-3\int_{\mathbb{R}}(\phi^{-\mathbf{X}})^2(u^R-u_{m})u^Sw'u_{\xi}^Sd\xi+\int_{\mathbb{R}}(\phi^{-\mathbf{X}})^3w(u_{\xi}^S+u_{\xi}^R)d\xi\\
			&-2\int_{\mathbb{R}}(\phi^{-\mathbf{X}})^3(u^S+u^R-u_{m})w'u_{\xi}^Sd\xi-\tau\dot{\mathbf{X}}(t)\int_{\mathbb{R}}\phi^{-\mathbf{X}}q_{\xi}^{S}w'u_{\xi}^{S}d\xi\\
			&-\tau\dot{\mathbf{X}}(t) \int_{\mathbb{R}}r^{-\mathbf{X}}\left(u_{\xi}^{S}+u_{\xi}^{R}\right)w'u_{\xi}^{S}d\xi+\tau\dot{\mathbf{X}}(t)\int_{\mathbb{R}}r^{-\mathbf{X}}w\tilde{q}_{\xi}d\xi\\
			&-\tau\dot{\mathbf{X}}(t)\int_{\mathbb{R}}\phi^{-\mathbf{X}}u_{\xi\xi}^{R}w'u_{\xi}^{S}d\xi-\frac{\tau\dot{\mathbf{X}}(t)}{2}\int_{\mathbb{R}}\left(r^{-\mathbf{X}}\right)^{2} w^{\prime} u_{\xi}^{S} d \xi\nonumber\\
			&+\frac{\tau\sigma}{2}\int_{\mathbb{R}}(r^{-\mathbf{X}})^2w'u_{\xi}^Sd\xi+\tau\sigma\int_{\mathbb{R}}r^{-\mathbf{X}}u_{\xi\xi}^Rwd\xi\nonumber\\
			&+6\tau\int_{\mathbb{R}}r^{-\mathbf{X}}u^R(u_{\xi}^R)^2wd\xi+3\tau\int_{\mathbb{R}}r^{-\mathbf{X}}(u^R)^2u_{\xi\xi}^Rwd\xi+R_0^{\tau}(t).
		\end{align*}
		Since $$\frac{15}{8}u_m^2\leq w < \frac{15}{2}u_m^2,\ -\frac{5}{2}u_m\leq w'\leq0,\ 0\leq w''\leq\frac{15}{2}$$ and $|u_{\xi}^S|\leq28u_m^3$ if $\tau<\frac{1}{252\sigma u_m^2}$. Then by the Cauchy inequality and H$\ddot{\text{o}}$lder inequality, we obtain
		\begin{equation}\label{N1}
			\begin{aligned}
				\left|\dot{\mathbf{X}}(t) \int_{\mathbb{R}} \phi^{-\mathbf{X}} w u_{\xi}^{R} d\xi\right|\leq &\varepsilon\left|\dot{{\mathbf{X}}}(t)\right|^2+C(\varepsilon)\left(\int_{\mathbb{R}}\phi^{-\mathbf{X}}wu_{\xi}^Rd\xi\right)^2\\
				\leq&\varepsilon\left|\dot{{\mathbf{X}}}(t)\right|^2+C(\varepsilon)\int_{\mathbb{R}}\left(\phi^{-\mathbf{X}}\right)^2wu_{\xi}^Rd\xi\int_{\mathbb{R}}wu_{\xi}^Rd\xi\\
				\leq&\varepsilon\left|\dot{{\mathbf{X}}}(t)\right|^2+C(\varepsilon)u_{m}^2\delta_R\int_{\mathbb{R}}\left(\phi^{-\mathbf{X}}\right)^2wu_{\xi}^Rd\xi.
			\end{aligned}
		\end{equation}
		Taking $\varepsilon=\frac{u_m^2}{64}.$ By $(\ref{shift})$, we have
		\begin{equation}\label{N2}
			\begin{aligned}
				&\left|-\frac{\dot{\mathbf{X}}(t)}{2} \int_{\mathbb{R}}\left(\phi^{-\mathbf{X}}\right)^{2} w^{\prime} u_{\xi}^{S} d \xi\right|
				=\left| \frac{16}{25u_m^2} \left( \int_{\mathbb{R}} \phi^{-\mathbf{X}} wu_{\xi}^S d\xi \right) \left(\int_{\mathbb{R}}(\phi^{-\mathbf{X}})^2 w' u_{\xi}^s  d\xi \right) \right|\\
				\leq &\frac{16}{25u_m^2} \|\phi^{-\mathbf{X}}\|_{L^\infty(\mathbb{R})} \|w\|_{L^\infty(\mathbb{R})} \|w'\|_{L^\infty(\mathbb{R})} \int_{\mathbb{R}} u_{\xi}^s d\xi \int_{\mathbb{R}} (\phi^{-\mathbf{X}})^2 u_{\xi}^S d\xi\\
				\leq &\frac{48}{25u_m}  \cdot\frac{15}{2}u_m^2\cdot \frac{5}{2}u_m\|\phi^{-\mathbf{X}}\|_{L^\infty(\mathbb{R})} \int_{\mathbb{R}} (\phi^{-\mathbf{X}})^2 u_{\xi}^S d\xi\\
				\leq &C u_m^2 \|\phi^{-\mathbf{X}}\|_{L^{\infty}(\mathbb{R})} \int_{\mathbb{R}} (\phi^{-\mathbf{X}})^2 u_{\xi}^S d\xi.
			\end{aligned}
		\end{equation}
		With this defined of $\delta_R=|u_+-u_m|$, we obtain
		\begin{equation}\label{N3}
			\begin{aligned}
				\left|-3\int_{\mathbb{R}}(\phi^{-\mathbf{X}})^2(u^R-u_{m})u^Sw'u_{\xi}^Sd\xi\right|
				\leq&3\delta_R\|u^S\|_{L^{\infty}(\mathbb{R})}\|w'\|_{L^{\infty}(\mathbb{R})} \int_{\mathbb{R}}(\phi^{-\mathbf{X}})^2u_{\xi}^Sd\xi\\
				\leq&C u_m^2\delta_R\int_{\mathbb{R}}(\phi^{-\mathbf{X}})^2u_{\xi}^Sd\xi.
			\end{aligned}
		\end{equation} 
		Since $|a+b|\leq |a|+|b|,$ we have
		\begin{equation}\label{N4}
			\begin{aligned}
				&\left|\int_{\mathbb{R}}(\phi^{-\mathbf{X}})^3w(u_{\xi}^S+u_{\xi}^R)d\xi\right|\\
				\leq &\|\phi^{-\mathbf{X}}\|_{L^\infty(\mathbb{R})} \|w\|_{L^\infty(\mathbb{R})}\int_{\mathbb{R}} (\phi^{-\mathbf{X}})^2 u_{\xi}^S d\xi+\|\phi^{-\mathbf{X}}\|_{L^\infty(\mathbb{R})}\int_{\mathbb{R}} (\phi^{-\mathbf{X}})^2w u_{\xi}^Rd\xi\\
				\leq&C u_m^2\|\phi^{-\mathbf{X}}\|_{L^\infty(\mathbb{R})}\int_{\mathbb{R}} (\phi^{-\mathbf{X}})^2 u_{\xi}^S d\xi+\|\phi^{-\mathbf{X}}\|_{L^\infty(\mathbb{R})}\int_{\mathbb{R}} (\phi^{-\mathbf{X}})^2w u_{\xi}^Rd\xi,
			\end{aligned}
		\end{equation}
		and
		\begin{equation}\label{N5}
			\begin{aligned}
				&\left|-2\int_{\mathbb{R}}(\phi^{-\mathbf{X}})^3(u^S+u^R-u_{m})w'u_{\xi}^Sd\xi\right|\\
				\leq&\left|2\int_{\mathbb{R}}(\phi^{-\mathbf{X}})^3u^Sw'u_{\xi}^Sd\xi\right|+\left|2\int_{\mathbb{R}}(\phi^{-\mathbf{X}})^3(u^R-u_{m})w'u_{\xi}^Sd\xi\right|\\
				\leq&\left(C u_m^2\|\phi^{-\mathbf{X}}\|_{L^\infty(\mathbb{R})}+C u_m\delta_R\|\phi^{-\mathbf{X}}\|_{L^\infty(\mathbb{R})}\right)\int_{\mathbb{R}} (\phi^{-\mathbf{X}})^2 u_{\xi}^S d\xi.
			\end{aligned}
		\end{equation}
		By Lemma $\ref{lem0}$ and $\eqref{WFP}$, we obtain  $\left|q_{\xi}^S\right|\leq9u_m^2\left|u_{\xi}^S\right|\leq Cu_{m}^5$ and $\left|w'\right|\leq \frac{5}{2}u_m$. Hence, we have 
		\begin{equation*}
			\begin{aligned}
				\left|-\tau\dot{\mathbf{X}}(t)\int_{\mathbb{R}}\phi^{-\mathbf{X}}q_{\xi}^{S}w'u_{\xi}^{S}d\xi\right|
				\leq&C\tau u_m^5\left|\dot{{\mathbf{X}}}(t)\int_{\mathbb{R}}\phi^{-\mathbf{X}}w'u_{\xi}^Sd\xi\right|\\
				\leq&C\tau \left(\varepsilon u_{m}^2\left|\dot{\mathbf{X}}\right|^2+\frac{u_{m}^8}{4\varepsilon}\left(\int_{\mathbb{R}}\phi^{-\mathbf{X}}w'u_{\xi}^Sd\xi\right)^2\right)\\
				\leq&C\tau u_m^2\varepsilon\left|\dot{\mathbf{X}}\right|^2+\frac{C\tau u_m^8}{\varepsilon}\left(\int_{\mathbb{R}}(\phi^{-\mathbf{X}})^2u_{\xi}^Sd\xi\right)\left(\int_{\mathbb{R}}(w')^2u_{\xi}^Sd\xi\right)\\
				\leq&C\tau u_m^2\varepsilon\left|\dot{\mathbf{X}}\right|^2+\frac{C\tau u_m^{11}}{\varepsilon}\int_{\mathbb{R}}(\phi^{-\mathbf{X}})^2u_{\xi}^Sd\xi.
			\end{aligned}
		\end{equation*}
		By choosing a suitable $\varepsilon$, we get 
		\begin{equation}\label{N6}
			\begin{aligned}
				\left|-\tau\dot{\mathbf{X}}(t)\int_{\mathbb{R}}\phi^{-\mathbf{X}}q_{\xi}^{S}w'u_{\xi}^{S}d\xi\right|\leq \frac{\tau}{64}u_m^2\left|\dot{\mathbf{X}}(t)\right|^2+C\tau u_m^{11}\int_{\mathbb{R}}(\phi^{-\mathbf{X}})^2u_{\xi}^Sd\xi.
			\end{aligned}
		\end{equation}
		Note that $\left|\frac{(w')^2}{w}\right|\leq \frac{10}{3}$. Moreover by Lemma $\ref{lem1}$ we have
		\begin{equation}\label{N7}
			\begin{aligned}
				&\left|-\tau\dot{\mathbf{X}}(t) \int_{\mathbb{R}}r^{-\mathbf{X}}\left(u_{\xi}^{S}+u_{\xi}^{R}\right)w'u_{\xi}^{S}d\xi\right|\\
				\leq&\frac{\tau}{32}u_m^2\left|\dot{\mathbf{X}}(t)\right|^2+\frac{16\tau}{u_m^2}\left(\int_{\mathbb{R}}r^{-\mathbf{X}}u_{\xi}^{S}w'u_{\xi}^{S}d\xi\right)^2+\frac{16\tau}{u_m^2}\left(\int_{\mathbb{R}}r^{-\mathbf{X}}u_{\xi}^{R}w'u_{\xi}^{S}d\xi\right)^2\\
				\leq&\frac{\tau}{32}u_m^2\left|\dot{\mathbf{X}}(t)\right|^2+\frac{16\tau}{u_m^2}\left(\int_{\mathbb{R}}(r^{-\mathbf{X}})^2wd\xi\right)\left(\int_{\mathbb{R}}\frac{(w')^2}{w}(u_{\xi}^S)^4d\xi\right)\\
				&+\frac{16\tau}{u_m^2}\left(\int_{\mathbb{R}}(r^{-\mathbf{X}})^2wd\xi\right)\left(\int_{\mathbb{R}}\frac{(w')^2}{w}(u_{\xi}^Su_{\xi}^R)^2d\xi\right)\\
				\leq&\frac{\tau}{32}u_m^2\left|\dot{\mathbf{X}}(t)\right|^2+C\tau u_m^8\int_{\mathbb{R}}(r^{-\mathbf{X}})^2wd\xi+C\delta_R^2\int_{\mathbb{R}}(r^{-\mathbf{X}})^2wd\xi.
			\end{aligned}
		\end{equation}
		Given that $\tilde{q}=q^S+u_{\xi}^R,$ we have
		\begin{equation}\label{N81}
			\begin{aligned}
				\left|\tau\dot{\mathbf{X}}(t)\int_{\mathbb{R}}r^{-\mathbf{X}}w\tilde{q}_{\xi}d\xi\right|\leq \left|\tau\dot{\mathbf{X}}(t)\int_{\mathbb{R}}r^{-\mathbf{X}}wq_{\xi}^Sd\xi\right|+\left|\tau\dot{\mathbf{X}}(t)\int_{\mathbb{R}}r^{-\mathbf{X}}wu_{\xi\xi}^Rd\xi\right|.
			\end{aligned}
		\end{equation}
		By Lemma $\ref{lem1}$ and $(\ref{N6})$, we obtain
		\begin{equation}\label{N82}
			\begin{aligned}
				&\left|\tau\dot{\mathbf{X}}(t)\int_{\mathbb{R}}r^{-\mathbf{X}}w\tilde{q}_{\xi}d\xi\right|\leq\frac{\tau}{64}u_m^2\left|\dot{\mathbf{X}}(t)\right|^2+C\tau u_m^8\int_{\mathbb{R}}(r^{-\mathbf{X}})^2wd\xi,\\
				&\left|\tau\dot{\mathbf{X}}(t)\int_{\mathbb{R}}r^{-\mathbf{X}}wu_{\xi\xi}^Rd\xi\right|\leq\frac{\tau}{64}u_m^2\left|\dot{\mathbf{X}}(t)\right|^2+C\delta_R^2\int_{\mathbb{R}}(r^{-\mathbf{X}})^2wd\xi,
			\end{aligned}
		\end{equation}
		substituting $(\ref{N82})$ into $(\ref{N81})$, we have
		\begin{equation}\label{N8}
			\left|\tau\dot{\mathbf{X}}(t)\int_{\mathbb{R}}r^{-\mathbf{X}}w\tilde{q}_{\xi}d\xi\right|\leq\frac{\tau}{32}u_m^2\left|\dot{\mathbf{X}}(t)\right|^2+\left(C\tau u_m^8+C\delta_{R}^2\right)\int_{\mathbb{R}}(r^{-\mathbf{X}})^2wd\xi.
		\end{equation}
		Similarly, we have
		\begin{equation}\label{N9}
			\left|-\tau\dot{\mathbf{X}}(t)\int_{\mathbb{R}}\phi^{-\mathbf{X}}u_{\xi\xi}^{R}w'u_{\xi}^{S}d\xi\right|\leq\frac{\tau}{64}u_m^2\left|\dot{\mathbf{X}}(t)\right|^2+C\delta_R^2\int_{\mathbb{R}}(\phi^{-\mathbf{X}})^2u_{\xi}^Sd\xi,
		\end{equation}
		\begin{equation}\label{N10}
			\left|-\frac{\tau\dot{\mathbf{X}}(t)}{2}\int_{\mathbb{R}}\left(r^{-\mathbf{X}}\right)^{2} w^{\prime} u_{\xi}^{S} d \xi\right|\leq C\tau u_m^3\|\phi^{-\mathbf{X}}\|_{L^{\infty}(\mathbb{R})}\int_{\mathbb{R}}(r^{-\mathbf{X}})^2wd\xi,
		\end{equation}
		\begin{equation}\label{N11}
			\begin{aligned}
				\left|\frac{\tau\sigma}{2}\int_{\mathbb{R}}(r^{-\mathbf{X}})^2w'u_{\xi}^Sd\xi\right|\leq C\tau u_m^4\int_{\mathbb{R}}(r^{-\mathbf{X}})^2wd\xi,
			\end{aligned}
		\end{equation}
		By the Lemma $(\ref{lem1})$, we have 
		\begin{equation}\label{N12}
			\begin{aligned}
				\left|\tau\sigma\int_{\mathbb{R}}r^{-\mathbf{X}}u_{\xi\xi}^Rwd\xi\right|\leq& \frac{\tau\sigma}{2}\left|\int_{\mathbb{R}}(r^{-\mathbf{X}})^2w(u_{\xi\xi}^R)^\frac{1}{3}d\xi\right|+\frac{\tau\sigma}{2}\left|\int_{\mathbb{R}}w(u_{\xi\xi}^R)^{\frac{5}{3}}d\xi\right|\\
				\leq&C\delta_R^{\frac{1}{3}}\int_{\mathbb{R}}(r^{-\mathbf{X}})^2wd\xi+C\delta_R^{\frac{1}{3}}(1+t)^{-\frac{4}{3}},
			\end{aligned}
		\end{equation}
		Similarly,
		\begin{equation}\label{N13}
			\begin{aligned}
				\left|3\tau\int_{\mathbb{R}}r^{-\mathbf{X}}(u^R)^2u_{\xi\xi}^Rwd\xi\right|\leq C\delta_R^{\frac{1}{3}}\int_{\mathbb{R}}(r^{-\mathbf{X}})^2wd\xi+C\delta_R^{\frac{1}{3}}(1+t)^{-\frac{4}{3}},
			\end{aligned}
		\end{equation}
		and
		\begin{equation}\label{N14}
			\begin{aligned}
				\left|6\tau\int_{\mathbb{R}}r^{-\mathbf{X}}u^R(u_{\xi}^R)^2wd\xi\right|\leq&3\tau|u_+|\left(\int_{\mathbb{R}}(r^{-\mathbf{X}})^2w(u_{\xi}^R)^{\frac{1}{3}}d\xi+\int_{\mathbb{R}}w(u_{\xi}^R)^{\frac{11}{3}}d\xi\right)\\
				\leq&C\delta_R^{\frac{1}{3}}\int_{\mathbb{R}}(r^{-\mathbf{X}})^2wd\xi+C\delta_R(1+t)^{-\frac{8}{3}}.
			\end{aligned}
		\end{equation}
		Finally, we can estimate $R_0^{\tau}(t)$:
		\begin{equation*}
			\begin{aligned}
				R_{0}^{\tau}(t)=&- \tau\sigma\int_{\mathbb{R}}r^{-\mathbf{X}}\phi^{-\mathbf{X}}w''\left(u_{\xi}^{S}\right)^{2}d\xi-\tau\sigma\int_{\mathbb{R}}r^{-\mathbf{X}}\phi^{-\mathbf{X}}w'u^S_{\xi\xi}d\xi\\
				&-3\tau\int_{\mathbb{R}}r^{-\mathbf{X}}(\phi^{-\mathbf{X}})^{2}\phi_{\xi}^{-\mathbf{X}}w'u_{\xi}^{S}d\xi -6\tau\int_{\mathbb{R}}(u^{S}+u^{R}-u_{m})r^{-\mathbf{X}}\phi^{-\mathbf{X}}\phi_{\xi}^{-\mathbf{X}}w'u_{\xi}^{S}d\xi\\
				&-3\tau\int_{\mathbb{R}}(u^{S}+u^{R}-u_{m})^{2}r^{-\mathbf{X}}\phi_{\xi}^{-\mathbf{X}}w'u_{\xi}^{S}d\xi-3\tau\int_{\mathbb{R}}(u_{\xi}^{S}+u_{\xi}^{R})r^{-\mathbf{X}}(\phi^{-\mathbf{X}})^{2}w'u_{\xi}^{S}d\xi\\
				&-6\tau\int_{\mathbb{R}}(u^{S}+u^{R}-u_{m})(u_{\xi}^{S}+u_{\xi}^{R})r^{-\mathbf{X}}\phi^{-\mathbf{X}}w'u_{\xi}^{S}d\xi-\frac{\tau}{2}\int_{\mathbb{R}}(r^{-\mathbf{X}})w''(u_{\xi}^{S})^{2}d\xi\\
				&-\frac{\tau}{2}\int_{\mathbb{R}}(r^{-\mathbf{X}})^{2}w'u_{\xi\xi}^Sd\xi+\frac{1}{2}\int_{\mathbb{R}}(\phi^{-\mathbf{X}})^{2}w'R_{1}^{\tau}u_{\xi}^{S}d\xi-\tau\sigma\int_{\mathbb{R}}u_{\xi\xi}^R\phi^{-\mathbf{X}}w'u_{\xi}^Sd\xi\\
				&- 6\tau\int_{\mathbb{R}}(u^{R})(u_{\xi}^{R})^{2}\phi^{-\mathbf{X}}w'u_{\xi}^{S}d\xi-3\tau\int_{\mathbb{R}}(u^{R})^{2}u_{\xi\xi}^{R}\phi^{-\mathbf{X}}w'u_{\xi}^{S}d\xi,
			\end{aligned}
		\end{equation*}
		Calculate the terms on the right side of the equation in order. we get
		\begin{align*}
			\left|\tau\sigma\int_{\mathbb{R}}r^{-\mathbf{X}}\phi^{-\mathbf{X}}w''\left(u_{\xi}^{S}\right)^{2}d\xi\right|\leq& \frac{\tau\sigma}{2}\left(\int_{\mathbb{R}}(\phi^{-\mathbf{X}})^2\frac{w''}{w}(u^S_{\xi})^3d\xi+\int_{\mathbb{R}}(r^{-\mathbf{X}})^2ww''u^S_{\xi}d\xi\right)\\
			\leq&C\tau u_m^6\int_{\mathbb{R}}(\phi^{-\mathbf{X}})^2u_{\xi}^Sd\xi+C\tau u_m^5\int_{\mathbb{R}}(r^{-\mathbf{X}})^2wd\xi,\\
			\left|\tau\sigma\int_{\mathbb{R}}r^{-\mathbf{X}}\phi^{-\mathbf{X}}w'u^S_{\xi\xi}d\xi\right|\leq&5\tau\sigma u_m^2\left(\int_{\mathbb{R}}(\phi^{-\mathbf{X}})^2\frac{(w')^2}{w}(u^S_{\xi})^2d\xi+\int_{\mathbb{R}}(r^{-\mathbf{X}})^2wd\xi\right)\\
			\leq&C\tau u_m^7\int_{\mathbb{R}}(\phi^{-\mathbf{X}})^2u^S_{\xi}d\xi+C\tau u_m^4\int_{\mathbb{R}}(r^{-\mathbf{X}})^2wd\xi,
		\end{align*}
		and
		\begin{align*}
			&\left|-3\tau\int_{\mathbb{R}}r^{-\mathbf{X}}(\phi^{-\mathbf{X}})^{2}\phi_{\xi}^{-\mathbf{X}}w'u_{\xi}^{S}d\xi\right|\\
			\leq&3\tau\lVert\phi^{-\mathbf{X}}\rVert_{L^{\infty}(\mathbb{R})}^2\left\lVert\frac{w'}{w}\right\rVert_{L^{\infty}(\mathbb{R})}\|u_{\xi}^S\|_{L^{\infty}(\mathbb{R})}\left|\int_{\mathbb{R}}r^{-\mathbf{X}}\phi_{\xi}^{-\mathbf{X}}wd\xi\right|\\
			\leq&C\tau u_m^2\lVert\phi^{-\mathbf{X}}\rVert_{L^{\infty}(\mathbb{R})}^2\left(\int_{\mathbb{R}}(r^{-\mathbf{X}})^2wd\xi+\int_{\mathbb{R}}(\phi_{\xi}^{-\mathbf{X}})^2wd\xi\right),
			\end{align*}
			\begin{align*}
			&\left|-6\tau\int_{\mathbb{R}}(u^{S}+u^{R}-u_{m})r^{-\mathbf{X}}\phi^{-\mathbf{X}}\phi_{\xi}^{-\mathbf{X}}w'u_{\xi}^{S}d\xi\right|\\
			\leq&6\tau\lVert u^S+u^R-u_m\rVert_{L^{\infty}(\mathbb{R})}\lVert\phi^{-\mathbf{X}}\rVert_{L^{\infty}(\mathbb{R})}\left\lVert\frac{w'}{w}\right\rVert_{L^{\infty}(\mathbb{R})}\|u_{\xi}^S\|_{L^{\infty}(\mathbb{R})}\left|\int_{\mathbb{R}}r^{-\mathbf{X}}\phi_{\xi}^{-\mathbf{X}}wd\xi\right|\\
			\leq&C\tau(u_m+\delta_R)u_m^2\lVert\phi^{-\mathbf{X}}\rVert_{L^{\infty}(\mathbb{R})}\left(\int_{\mathbb{R}}(r^{-\mathbf{X}})^2wd\xi+\int_{\mathbb{R}}(\phi_{\xi}^{-\mathbf{X}})^2wd\xi\right).
		\end{align*}
		Similarly, we have
		\begin{align*}
			&\left|-3\tau\int_{\mathbb{R}}(u^{S}+u^{R}-u_{m})^{2}r^{-\mathbf{X}}\phi_{\xi}^{-\mathbf{X}}w'u_{\xi}^{S}d\xi\right|\\
			\leq&3\tau\lVert u^S+u^R-u_m\rVert_{L^{\infty}(\mathbb{R})}^2\left\lVert\frac{w'}{w}\right\rVert_{L^{\infty}(\mathbb{R})}\|u_{\xi}^S\|_{L^{\infty}(\mathbb{R})}\left|\int_{\mathbb{R}}r^{-\mathbf{X}}\phi_{\xi}^{-\mathbf{X}}wd\xi\right|\\
			\leq&C\tau \left(u_m^2+\delta_R\right)u_m^2\left(\int_{\mathbb{R}}(r^{-\mathbf{X}})^2wd\xi+\int_{\mathbb{R}}(\phi_{\xi}^{-\mathbf{X}})^2wd\xi\right),
		\end{align*}
		and
		\begin{align*}
			&\left|-3\tau\int_{\mathbb{R}}(u_{\xi}^{S}+u_{\xi}^{R})r^{-\mathbf{X}}(\phi^{-\mathbf{X}})^{2}w'u_{\xi}^{S}d\xi\right|\\
			\leq&3\tau\|u_{\xi}^S+u^R_{\xi}\|_{L^{\infty}(\mathbb{R})}\lVert\phi^{-\mathbf{X}}\rVert_{L^{\infty}(\mathbb{R})}\left\lVert\frac{w'}{w}\right\rVert_{L^{\infty}(\mathbb{R})}\|u_{\xi}^S\|_{L^{\infty}(\mathbb{R})}\left|\int_{\mathbb{R}}r^{-\mathbf{X}}\phi^{-\mathbf{X}}wu_{\xi}^Sd\xi\right|\\
			\leq&C\tau(u_m^3+\delta_R)u_m^2\lVert\phi^{-\mathbf{X}}\rVert_{L^{\infty}(\mathbb{R})}\int_{\mathbb{R}}(r^{-\mathbf{X}})^2wd\xi+C\tau(u_m^3+\delta_R)u_m^5\lVert\phi^{-\mathbf{X}}\rVert_{L^{\infty}(\mathbb{R})}\int_{\mathbb{R}}(\phi^{-\mathbf{X}})^2u^S_{\xi}d\xi.
		\end{align*}
		Likewise,
		\begin{align*}
			&\left|-6\tau\int_{\mathbb{R}}(u^{S}+u^{R}-u_{m})(u_{\xi}^{S}+u_{\xi}^{R})r^{-\mathbf{X}}\phi^{-\mathbf{X}}w'u_{\xi}^{S}d\xi\right|\\
			\leq& C\tau(u_m+\delta_R)(u_m^3+\delta_R)\int_{\mathbb{R}}(r^{-\mathbf{X}})^2wd\xi+C\tau u_m^3(u_m+\delta_R)(u_m^3+\delta_R)\int_{\mathbb{R}}(\phi^{-\mathbf{X}})^2u^S_{\xi}d\xi.
		\end{align*}
		By $(\ref{WF})$ and $(\ref{WFP})$, we have
		\begin{align*}
			&\left|-\frac{\tau}{2}\int_{\mathbb{R}}(r^{-\mathbf{X}})w''(u_{\xi}^{S})^{2}d\xi-\frac{\tau}{2}\int_{\mathbb{R}}(r^{-\mathbf{X}})^{2}w'u_{\xi\xi}^Sd\xi\right|\\
			\leq&\frac{\tau}{2}\left(\left\lVert\frac{w''}{w}\right\rVert_{L^{\infty}(\mathbb{R})}\|u_{\xi}^S\|_{L^{\infty}(\mathbb{R})}^2+\left\lVert\frac{w'}{w}\right\rVert_{L^{\infty}(\mathbb{R})}\|u_{\xi\xi}^S\|_{L^{\infty}(\mathbb{R})}\right)\int_{\mathbb{R}}(r^{-\mathbf{X}})^2wd\xi
			\leq C\tau u_m^4\int_{\mathbb{R}}(r^{-\mathbf{X}})^2wd\xi.
		\end{align*}
		Notice that if $\tau\leq \frac{1}{18\sigma u_m^2}$, we have
		\begin{align*}
			\left|R_{1}^{\tau}(u^S)\right|\leq \tau\sigma\left(432u_m^3+486u_m^4\right).
		\end{align*}
		Then, we get
		\begin{align*}
			\left|\frac{1}{2}\int_{\mathbb{R}}(\phi^{-\mathbf{X}})^{2}w'R_{1}^{\tau}u_{\xi}^{S}d\xi\right|\leq C\tau u_m^3\left(u_m^3+u_m^4\right)\int_{\mathbb{R}}(\phi^{-\mathbf{X}})^2u^S_{\xi}d\xi.
		\end{align*}
		For these last three items, by Lemma \ref{lem1} and $\|u^S_{\xi}\|_{L^{\infty}(\mathbb{R})}\leq 28u_m^3$, we get
		\begin{align*}
			&\left|-\tau\sigma\int_{\mathbb{R}}u_{\xi\xi}^R\phi^{-\mathbf{X}}w'u_{\xi}^Sd\xi-6\tau\int_{\mathbb{R}}(u^{R})(u_{\xi}^{R})^{2}\phi^{-\mathbf{X}}w'u_{\xi}^{S}d\xi-3\tau\int_{\mathbb{R}}(u^{R})^{2}u_{\xi\xi}^{R}\phi^{-\mathbf{X}}w'u_{\xi}^{S}d\xi\right|\\
			\leq&C \|w'\|_{L^{\infty}(\mathbb{R})}^2\int_{\mathbb{R}}(u^R_{\xi\xi})^2d\xi+\frac{\tau\sigma}{2} \|u^{S}_{\xi}\|_{L^{\infty}(\mathbb{R})}\int_{\mathbb{R}}(\phi^{-\mathbf{X}})^2u^S_{\xi}d\xi+C\|w'\|_{L^{\infty}(\mathbb{R})}^2\int_{\mathbb{R}}(u^{R})^2(u^R_{\xi})^{4}d\xi\\
			&+C\|w'\|_{L^{\infty}(\mathbb{R})}^2\int_{\mathbb{R}}(u^R)^4(u^R_{\xi\xi})^2d\xi+\frac{9}{2}\tau \|u^S_{\xi}\|_{L^{\infty}(\mathbb{R})}\int_{\mathbb{R}}(\phi^{-\mathbf{X}})^2u^S_{\xi}d\xi\\
			\leq&\tau\left(\frac{5\sigma}{4}u_m+126u_m^3\right)\int_{\mathbb{R}}(\phi^{-\mathbf{X}})^2u^S_{\xi}d\xi+Cu_m^2\delta_R^{\frac{2}{3}}(1+t)^{-\frac{4}{3}}.
		\end{align*}
		Hence, we have
		\begin{equation}\label{N15}
			\begin{aligned}
				\left|R_0^{\tau}(t)\right|
				\leq&\tau(\frac{5\sigma}{4}u_m+126u_m^3+C u_m^{6}+Cu_m^7+C\|\phi^{-\mathbf{X}}\|_{L^{\infty}(\mathbb{R})})\int_{\mathbb{R}}(\phi^{-\mathbf{X}})^2u_{\xi}^Sd\xi\\
				&+C\tau\left(u_m^4+ u_m^5+\delta_R+\|\phi^{-\mathbf{X}}\|_{L^{\infty}(\mathbb{R})}\right)\int_{\mathbb{R}}(r^{-\mathbf{X}})^2wd\xi\\
				&+C\tau\left(u_m^4+\delta_R+\|\phi^{-\mathbf{X}}\|_{L^{\infty}(\mathbb{R})}\right)\int_{\mathbb{R}}(\phi_{\xi}^{-\mathbf{X}})^2wd\xi+C\delta_R^{\frac{1}{3}}(1+t)^{-\frac{4}{3}}.
			\end{aligned}
		\end{equation}
	\end{proof}
	Estimating $\mathbf{J}(t)$ requires us to account for the wave interactions. Given that the shifted Oleinik shock and the shifted rarefaction wave are permanently attached, their interaction is inherently intricate. Thus, we have the following lemma, see \cite{ref16}.
	\begin{lemma}\label{WIE}
		For any $\varepsilon\in(0,1)$, there exists a positive constant $C$ (is independent of $\delta_R, \tau$ and $T$), such that
		\begin{equation} 
			\begin{aligned}
				&\int_{-\infty}^{0} |u^{S} - u_{m}|u_{\xi}^{R} d\xi \leq C\delta_{S}\delta_{R}^{\frac{2\varepsilon}{2+\varepsilon}} (1+t)^{-1+\varepsilon}, \\
				&\int_{0}^{+\infty} |u^{S} - u_{m}|u_{\xi}^{R} d\xi \leq C\delta_{S}^{2\varepsilon-1} \delta_{R}^{\varepsilon} (1+t)^{-1+\varepsilon}, \\
				&\int_{-\infty}^{0} |u^{R} - u_{m}|u_{\xi}^{S} d\xi \leq C\delta_{S}\delta_{R}^{\frac{2\varepsilon}{2+\varepsilon}} (1+t)^{-1+\varepsilon},  \\
				&\int_{0}^{[f'(u_{+}) - \sigma](t+1)} |u^{R} - u^{r}|u_{\xi}^{S} d\xi \leq C\delta_{S}\delta_{R}^{\varepsilon} (1+t)^{-1+\varepsilon},  \\
				&\int_{0}^{[f'(u_{+}) - \sigma](t+1)} |u^{r} - u_{m}|u_{\xi}^{S} d\xi \leq C\delta_{S}^{-1+2\varepsilon}\delta_{R}^{\varepsilon} (1+t)^{-1+\varepsilon} \ln^{1-\varepsilon}(1 + C\delta_{S}^{2}\delta_{R}t),  \\
				&\int_{[f'(u_{+}) - \sigma](t+1)}^{+\infty} |u^{R} - u_{m}|u_{\xi}^{S} d\xi \leq C\delta_{S}\delta_{R}(1 + C\delta_{S}^{2}\delta_{R}t)^{-1},
			\end{aligned} 
		\end{equation}
		where
		\begin{align*}
			u^{R} = u^{R}(1 + t, \xi + \sigma t; u_{m}, u_{+}) \text{ and } u^{r} = u^{r}\left( \frac{\xi + \sigma(1+t)}{1+t}; u_{m}, u_{+} \right).
		\end{align*}
	\end{lemma}
	Note that $\delta_R$ is small. Thus we present an estimate for $\mathbf{J}(t)$ in the following lemma.
	\begin{lemma}\label{J}
		For $\mathbf{J}(t)$ is defined in $(\ref{j}),$ we have
		\begin{equation}
			\begin{aligned}
				\left|\mathbf{J}(t)\right|\leq \frac{1}{1800}\int_{\mathbb{R}}(\phi_{\xi}^{-\mathbf{X}})^2wd\xi+\frac{2u_m^3}{15}\int_{\mathbb{R}}(\phi^{-\mathbf{X}})^2u_{\xi}^Sd\xi+Cu_m^2\|\phi^{-\mathbf{X}}\|_{L^2(\mathbb{R})}^2\delta_R^{\frac{1}{3}}(1+t)^{-\frac{4}{3}}.
			\end{aligned}
		\end{equation}
	\end{lemma}
	\begin{proof}
		\begin{equation}
			\begin{aligned}
				\left|\mathbf{J}(t)\right|=&\left|3 \int_{\mathbb{R}} (\phi^{-\mathbf{X}})^2 (u^S - u_m) w u_\xi^R d\xi\right|\\
				\leq&3\left\lVert\phi^{-\mathbf{X}}\sqrt{w}\right\rVert_{L^{\infty}(\mathbb{R})}^2\int_{\mathbb{R}}\left|u^S-u_m\right|u_{\xi}^Rd\xi\\
				\leq&C\left\lVert\left(\phi^{-\mathbf{X}}\sqrt{w}\right)_{\xi}\right\rVert_{L^{2}(\mathbb{R})}\cdot\left\lVert\phi^{-\mathbf{X}}\sqrt{w}\right\rVert_{L^{2}(\mathbb{R})}\int_{\mathbb{R}}\left|u^S-u_m\right|u_{\xi}^Rd\xi.
			\end{aligned}
		\end{equation}
		By the Minkowski inequality, we have \begin{align*}
			\left\lVert\left(\phi^{-\mathbf{X}}\sqrt{w}\right)_{\xi}\right\rVert_{L^{2}(\mathbb{R})}
			\leq&\left\lVert\phi_{\xi}^{-\mathbf{X}}\sqrt{w}\right\rVert_{L^{2}(\mathbb{R})}+\left\lVert\phi^{-\mathbf{X}}\frac{w'}{2\sqrt{w}}u_{\xi}^S\right\rVert_{L^{2}(\mathbb{R})}\\
			\leq&\left\lVert\phi_{\xi}^{-\mathbf{X}}\sqrt{w}\right\rVert_{L^{2}(\mathbb{R})}+\frac{\sqrt{30}}{6}\left\lVert\phi^{-\mathbf{X}}u_{\xi}^S\right\rVert_{L^{2}(\mathbb{R})}.
		\end{align*}
		Hence, we have 
		\begin{equation}\label{J1}
			\begin{aligned}
				\left|\mathbf{J}(t)\right|\leq &C\left\lVert\phi_{\xi}^{-\mathbf{X}}\sqrt{w}\right\rVert_{L^{2}(\mathbb{R})}\left\lVert\phi^{-\mathbf{X}}\sqrt{w}\right\rVert_{L^{2}(\mathbb{R})}\int_{\mathbb{R}}\left|u^S-u_m\right|u_{\xi}^Rd\xi\\
				&+C\left\lVert\phi^{-\mathbf{X}}u_{\xi}^S\right\rVert_{L^{2}(\mathbb{R})}\left\lVert\phi^{-\mathbf{X}}\sqrt{w}\right\rVert_{L^{2}(\mathbb{R})}\int_{\mathbb{R}}\left|u^S-u_m\right|u_{\xi}^Rd\xi\\
				\leq&\frac{1}{1800}\left\lVert\phi_{\xi}^{-\mathbf{X}}\sqrt{w}\right\rVert_{L^{2}(\mathbb{R})}^2+\frac{1}{210}\left\lVert\phi^{-\mathbf{X}}u_{\xi}^S\right\rVert_{L^{2}(\mathbb{R})}^2\\
				&+C\left\lVert\phi^{-\mathbf{X}}\sqrt{w}\right\rVert_{L^{2}(\mathbb{R})}^2\left(\int_{\mathbb{R}}\left|u^S-u_m\right|u_{\xi}^Rd\xi\right)^2\\
				\leq&\frac{1}{1800}\int_{\mathbb{R}}(\phi_{\xi}^{-\mathbf{X}})^2wd\xi+\frac{2u_m^3}{15}\int_{\mathbb{R}}(\phi^{-\mathbf{X}})^2u_{\xi}^Sd\xi\\
				&+Cu_m^2\lVert\phi^{-\mathbf{X}}\rVert_{L^{2}(\mathbb{R})}^2\left(\int_{\mathbb{R}}\left|u^S-u_m\right|u_{\xi}^Rd\xi\right)^2,
			\end{aligned}
		\end{equation}
		Applying Lemma $\ref{WIE},$ let $\varepsilon=\frac{1}{3}$ and $\delta_{R}=\delta_{S}^2$, we get
		\begin{equation}\label{J2}
			\begin{aligned}
				\left(\int_{\mathbb{R}}\left|u^S-u_m\right|u_{\xi}^Rd\xi\right)^2=&\left(\int_{-\infty}^0\left|u^S-u_m\right|u_{\xi}^Rd\xi+\int_{0}^{+\infty}\left|u^S-u_m\right|u_{\xi}^Rd\xi\right)^2\\
				\leq&C\left(\delta_{R}^{\frac{2}{7}}+\delta_R^{\frac{1}{6}}\right)^2(1+t)^{-\frac{4}{3}}
				\leq C\delta_{R}^{\frac{1}{3}}(1+t)^{-\frac{4}{3}}.
			\end{aligned}
		\end{equation}
		This completes the proof of Lemma $\ref{J}$. 
	\end{proof}
	Before estimating $\mathbf{Z}(t)$, we give this lemma to estimate $\partial_{\xi}^kF,\ k=0,1,2$:
	\begin{lemma}\label{lemf}
		There exists a positive constant $C$, such that for $k=0,1,2$
		\begin{equation}\label{kf}
			\int_0^t\int_{\mathbb{R}}\left|\partial_{\xi}^kF\right|^2d\xi dt\leq Cu_m\delta_R^{\frac{1}{6}}= C\delta_{R}^{\frac{2}{3}}.
		\end{equation}
	\end{lemma} 
	\begin{proof}
		Notice that 
		\begin{align*}
			\partial_{\xi}^kF&= \partial_{\xi}^{k+1}\left[f(\tilde{u}) - f(u^R) - f((u^S)_{\xi})\right]- \partial_{\xi}^{k+2}(u^R)\\
			&= \partial_{\xi}^k\left[\left(f'(\tilde{u}) - f'(u^S)\right)(u^S)_{\xi}+ \left(f'(\tilde{u}) - f'(u^R)\right)(u^R)_{\xi}\right] -  \partial_{\xi}^{k+2}(u^R).
		\end{align*}
		If $k=0,$ we get
		\begin{equation*}
			\int_{\mathbb{R}} |F|^2 d\xi \leq C(u_m^2+\delta_{R}^2) \int_{\mathbb{R}} |u^S - u_m|^2 (u_\xi^R)^2 d\xi + C(u_m^2+\delta_{R}^2) \int_{\mathbb{R}} |u^R - u_m|^2 (u_\xi^S)^2 d\xi + C \int_{\mathbb{R}} |u_{\xi\xi}^R|^2 d\xi.
		\end{equation*}
		By the Lemma $\ref{lem1}$ and $\ref{WIE}$,  we get
		\begin{align*}
			\int_{\mathbb{R}} |u^S - u_m|^2 (u_{\xi}^R)^2 d\xi &\leq C \|u_{\xi}^R\|_{L^{\infty}(\mathbb{R})} \int_{\mathbb{R}} |u^S - u_m| u_{\xi}^R d\xi \\
			&\leq C \delta_R^{\frac{3}{5}}(1+t)^{-\frac{2}{5}} \delta_R^{\frac{1}{6}}(1+t)^{-\frac{2}{3}}= C\delta_R^{\frac{23}{30}} (1+t)^{-\frac{6}{5}},
		\end{align*}
		and
		\begin{align}\label{rms}
			&\int_{\mathbb{R}} |u^R - u_m|^2 (u_{\xi}^S)^2 d\xi \nonumber\\
			\leq& \int_{-\infty}^{0} |u^R - u_m|^2 (u_{\xi}^S)^2 d\xi + \int_{0}^{[f'(u_+)-\sigma](1+t)} |u^R - u_m|^2 (u_{\xi}^S)^2 d\xi \  \\
			&+ \int_{[f'(u_+)-\sigma](1+t)}^{+\infty} |u^R - u_m|^2 (u_{\xi}^S)^2 d\xi.\nonumber
		\end{align}
		By Lemma $\ref{lem1}$, we have
		\begin{align*}
			\int_{-\infty}^{0} |u^R - u_m|^2 (u_\xi^S)^2 d\xi \leq C \delta_R^{\frac{4}{7}} (1+t)^{-\frac{4}{3}}.
		\end{align*}
		For $\int_{0}^{[f'(u_{+})-\sigma](1+t)} |u^{R}-u_{m}|^{2}(u_{\xi}^{S})^{2}d\xi$, we have
		\begin{align*}
			&\int_{0}^{[f'(u_{+})-\sigma](1+t)} |u^{R}-u_{m}|^{2}(u_{\xi}^{S})^{2}d\xi\\
			&\leq \int_{0}^{[f'(u_{+})-\sigma](1+t)} |u^{R}-u^{r}|^{2}(u_{\xi}^{S})^{2}d\xi + \int_{0}^{[f'(u_{+})-\sigma](1+t)} |u^{r}-u_{m}|^{2}(u_{\xi}^{S})^{2}d\xi.
		\end{align*}
		By the Lemma $\ref{lem1}$, we get
		\begin{align*}
			&\int_{0}^{[f'(u_{+})-\sigma](1+t)} |u^{R}-u^{r}|^{2}(u_{\xi}^{S})^{2}d\xi\\
			\leq&\int_{0}^{[f'(u_{+})-\sigma](1+t)}|u^{R}(1+t,\xi+\sigma t)-u^{R}(1+t, \xi+\sigma(1+t))|^{2}(u_{\xi}^{S})^{2}d\xi\\
			&+\int_{0}^{[f'(u_{+})-\sigma](1+t)}|u^{R}(1+t, \xi+\sigma(1+t))-u^r|^{2}(u_{\xi}^{S})^{2}d\xi\\
			\leq& C\|u_{\xi}^R\|_{L^{\infty}(\mathbb{R})}^{2} \int_{0}^{+\infty} (u_{\xi}^S)^{2} d\xi + C_{\varepsilon} \delta_{R}^{2\varepsilon} (1+t)^{-2+2\varepsilon} \int_{0}^{+\infty} (u_{\xi}^S)^{2} d\xi\\
			\leq& C \delta_{R}^{\frac{2}{3}} (1+t)^{-\frac{4}{3}},
		\end{align*}
		and
		\begin{equation*}
			\begin{aligned}
				&\int_{0}^{[f'(u_{+})-\sigma](1+t)}|u^{r}-u_{m}|^{2}(u_{\xi}^{S})^{2}d\xi\\
				=&\int_{0}^{[f'(u_{+})-\sigma](1+t)}\left|(f')^{-1}\left(\frac{\xi+\sigma(1+t)}{1+t}\right)-(f')^{-1}(\sigma)\right|^{2}(u_{\xi}^{S})^{2}d\xi\\
				\leq& C\int_{0}^{[f'(u_{+})-\sigma](1+t)}\frac{\xi^{2}\delta_{S}^{6}}{(1+t)^{2}(C\delta_{S}^{2}\xi+1)^{4}}d\xi\\
				\leq& C\frac{1}{(1+t)^{2}}\int_{0}^{[f'(u_{+})-\sigma](1+t)}\frac{1}{C\xi+1}d\xi\\
				\leq& C\frac{1}{(1+t)^{2}}\ln(1+C\delta_{R}t)\leq C\delta_R^{\frac{2}{3}}\frac{\ln^{\frac{1}{3}}(1+C\delta_{R}t)}{(1+t)^{\frac{4}{3}}}.
			\end{aligned}
		\end{equation*}
		So, we get
		\begin{align*}
			\int_{0}^{[f'(u_{+})-\sigma](1+t)} |u^{R}-u_{m}|^{2}(u_{\xi}^{S})^{2}d\xi\leq C \delta_{R}^{\frac{2}{3}} (1+t)^{-\frac{4}{3}}.
		\end{align*}
		For $\int_{[f'(u_{+})-\sigma](1+t)}^{+\infty} |u^{R}-u_{m}|^{2} (u_{\xi}^{S})^{2} d\xi$, we have
		\begin{align*}
			&\int_{[f'(u_{+})-\sigma](1+t)}^{+\infty} |u^{R}-u_{m}|^{2} (u_{\xi}^{S})^{2} d\xi \\
			&\leq \delta_{R}^{2} \|u_{\xi}^{S}\|_{L^{\infty}([f'(u_{+})-\sigma](1+t), +\infty)} \int_{[f'(u_{+})-\sigma](1+t)}^{+\infty} u_{\xi}^{S} d\xi
			\leq  \frac{C\delta_{R}^{2}}{(1+C\delta_{R}t)^{2}}.
		\end{align*} 
		Substituting the above inequalities into $(\ref{rms})$, we get
		\begin{align*}
			\int_{\mathbb{R}} |u^R - u_m|^2 (u_{\xi}^S)^2 d\xi\leq C\delta_R^{\frac{4}{7}}(1+t)^{-\frac{4}{3}},
		\end{align*}
		and
		\begin{equation*} 
			\int_{\mathbb{R}} |u_{\xi \xi}^{R}|^2 d\xi \leq C \|u_{\xi \xi}^{R}\|_{L^\infty(\mathbb{R})}^{\frac{2}{3}} \|u_{\xi \xi}^{R}\|_{L^\infty(\mathbb{R})}^{\frac{1}{3}} (1+t)^{-1} \leq C\delta_R^{\frac{2}{3}} (1+t)^{-\frac{4}{3}}=Cu_m\delta_{R}^{\frac{1}{6}}(1+t)^{-\frac{4}{3}}.
		\end{equation*}
		Hence, we have
		\begin{align*}
			\int_{\mathbb{R}} |F|^2 d\xi\leq Cu_m\delta_R^{\frac{1}{6}}(1+t)^{-\frac{4}{3}}.
		\end{align*}\par
		If $k=1,$ we notice that
		\begin{align*}
			F_{\xi}=&\left(f''(\tilde{u})-f''(u^S)\right)(u^S_{\xi})^2+\left(f''(\tilde{u})-f''(u^R)\right)(u^R_{\xi})^2+2f''(\tilde{u})u^S_{\xi}u^R_{\xi}\\
			&+\left(f'(\tilde{u})-f'(u^R)\right)u^R_{\xi\xi}+\left(f'(\tilde{u})-f'(u^S)\right)u^S_{\xi\xi}-u^R_{\xi\xi\xi}\\
			=&6(u^R-u_m)(u_{\xi}^S)^2+6(u^S-u_m)(u^R_{\xi})^2+12(u^S+u^R-u_m)u_{\xi}^Su_{\xi}^R\\
			&+3(u^S-u_m)(u^S+2u^R-u_m)u_{\xi\xi}^R+3(u^R-u_m)(2u^S+u^R-u_m)u_{\xi\xi}^S-u_{\xi\xi\xi}^R.
		\end{align*}
		By $(a+b)^2\leq2a^2+2b^2$, we get
		\begin{equation}\label{Fxi}
			\begin{aligned}
				&\int_{\mathbb{R}}\left|F_{\xi}\right|^2d\xi
				\leq C\int_{\mathbb{R}}(u^R-u_m)^2(u_{\xi}^S)^4d\xi+C\int_{\mathbb{R}}(u^S-u_m)^2(u^R_{\xi})^4d\xi\\
				&\ \ +C\int_{\mathbb{R}}(u^S+u^R-u_m)^2(u_{\xi}^Su_{\xi}^R)^2d\xi+C\int_{\mathbb{R}}(u^S-u_m)^2(u^S+2u^R-u_m)^2(u_{\xi\xi}^R)^2d\xi\\
				&\ \ +C\int_{\mathbb{R}}(u^R-u_m)^2(2u^S+u^R-u_m)^2(u_{\xi\xi}^S)^2d\xi+C\int_{\mathbb{R}}(u^R_{\xi\xi\xi})^2d\xi.
			\end{aligned}
		\end{equation}
		It is easy to check that
		\begin{align*}
			&\int_{\mathbb{R}}(u^R-u_m)^2(u_{\xi}^S)^4d\xi\leq Cu_m^6\int_{\mathbb{R}}(u^R-u_m)^2(u_{\xi}^S)^2d\xi\leq Cu_m^6\delta_R^{\frac{4}{7}}(1+t)^{-\frac{4}{3}},\\
			&\int_{\mathbb{R}}(u^S-u_m)^2(u^R_{\xi})^4d\xi\leq C\delta_{R}^2\int_{\mathbb{R}}(u^S-u_m)^2(u^R_{\xi})^2d\xi\leq Cu_m^4\delta_R^{\frac{43}{55}} (1+t)^{-\frac{6}{5}},\\
			&\int_{\mathbb{R}}(u^R-u_m)^2(2u^S+u^R-u_m)^2(u_{\xi\xi}^S)^2d\xi\leq Cu_m^4\int_{\mathbb{R}}\int_{\mathbb{R}}(u^R-u_m)^2(u_{\xi}^S)^2d\xi\leq Cu_m^4\delta_R^{\frac{4}{7}}(1+t)^{-\frac{4}{3}}.
		\end{align*}
		For $\int_{\mathbb{R}}(u^S+u^R-u_m)^2(u_{\xi}^Su_{\xi}^R)^2d\xi$, we get
		\begin{align*}
			\int_{\mathbb{R}}(u^S+u^R-u_m)^2(u_{\xi}^Su_{\xi}^R)^2d\xi&\leq C(u_m^2+\delta_{R}^2)\int_{\mathbb{R}}(u_{\xi}^Su_{\xi}^R)^2d\xi\leq Cu_m^2\lVert u_{\xi}^R\rVert_{L^{\infty}(\mathbb{R})}\int_{\mathbb{R}}(u_{\xi}^S)^{\frac{1}{5}}u_{\xi}^Rd\xi.
		\end{align*}
		Notice that $\int_{\mathbb{R}}(u_{\xi}^S)^{\frac{1}{5}}u_{\xi}^Rd\xi=\int_{-\infty}^0(u_{\xi}^S)^{\frac{1}{5}}u_{\xi}^Rd\xi+\int_{0}^{+\infty}(u_{\xi}^S)^{\frac{1}{5}}u_{\xi}^Rd\xi$, by Lemma $\ref{lem1}$, we have $$\lVert u_{\xi}^R\rVert_{L^{\infty}(\mathbb{R})}\leq C(1+t)^{-1}$$ 
		and
		\begin{align*}
			\int_{-\infty}^0(u_{\xi}^S)^{\frac{1}{5}}u_{\xi}^Rd\xi\leq C\int_{-\infty}^0u_{\xi}^Rd\xi=C\left|u^R(1+t,\sigma t)-u_m\right|\leq C\delta_R^{\frac{4}{7}}(1+t)^{-\frac{1}{5}}.
		\end{align*}
		By H$\ddot{\text{o}}$ler inequality, we get 
		\begin{align*}
			\int_{0}^{+\infty}(u_{\xi}^S)^{\frac{1}{5}}u_{\xi}^Rd\xi&\leq \left(\int_{0}^{+\infty}\left((u_{\xi}^S)^{\frac{1}{5}}\right)^5d\xi\right)^{\frac{1}{5}}\left(\int_{0}^{+\infty}(u_{\xi}^R)^{\frac{5}{4}}d\xi\right)^{\frac{4}{5}}\leq C\delta_{S}^{\frac{9}{5}}\delta_R^{\frac{4}{5}}(1+t)^{-\frac{1}{5}}.
		\end{align*}
		Hence, we have
		\begin{align*}
			\int_{\mathbb{R}}(u^S+u^R-u_m)^2(u_{\xi}^Su_{\xi}^R)^2d\xi\leq Cu_m^2\delta_R^{\frac{4}{7}}(1+t)^{-\frac{6}{5}}.
		\end{align*}
		By Lemma $(\ref{lem1})$, we get
		\begin{align*}
			&\int_{\mathbb{R}}(u^S-u_m)^2(u^S+2u^R-u_m)^2(u_{\xi\xi}^R)^2d\xi\leq Cu_m^2\int_{\mathbb{R}}(u_{\xi\xi}^R)^2d\xi\leq Cu_m^2\delta_R^{\frac{2}{3}} (1+t)^{-\frac{4}{3}},\\
			&\int_{\mathbb{R}}(u^R_{\xi\xi\xi})^2d\xi\leq C \|u_{\xi\xi \xi}^{R}\|_{L^\infty(\mathbb{R})}^{\frac{2}{3}} \|u_{\xi \xi\xi}^{R}\|_{L^\infty(\mathbb{R})}^{\frac{1}{3}} (1+t)^{-1} \leq C\delta_R^{\frac{2}{3}} (1+t)^{-\frac{4}{3}}=Cu_m\delta_R^{\frac{1}{6}} (1+t)^{-\frac{4}{3}}.
		\end{align*} 
		Then, we get 
		\begin{align*}
			\int_{\mathbb{R}}\left|F_{\xi}\right|^2d\xi\leq Cu_m\delta_R^{\frac{1}{6}}(1+t)^{-\frac{6}{5}}.
		\end{align*}\par
		If $k=2$, we notice that 
		\begin{align*}
			F_{\xi\xi}=&18(u_{\xi}^S)^2u_{\xi}^R+18u_{\xi}^S(u_{\xi}^R)^2+18(u^S-u_m)u_{\xi}^Ru_{\xi\xi}^R+18(u^R-u_m)u_{\xi}^Su_{\xi\xi}^S\\
			&+18(u^S+u^R-u_m)u_{\xi}^Su_{\xi\xi}^R+18(u^S+u^R-u_m)u_{\xi}^Ru_{\xi\xi}^S-u_{\xi\xi\xi\xi}^R\\
			&+3(u^S-u_m)(u^S+2u^R-u_m)u_{\xi\xi\xi}^R+3(u^R-u_m)(2u^S+u^R-u_m)u_{\xi\xi\xi}^S.
		\end{align*}
		Therefore, we have
		\begin{align*}
			&\int_{\mathbb{R}}(F_{\xi\xi})^2d\xi
			\leq C\int_{\mathbb{R}}(u_{\xi}^S)^4(u_{\xi}^R)^2d\xi+C\int_{\mathbb{R}}(u_{\xi}^S)^2(u_{\xi}^R)^4d\xi+C\int_{\mathbb{R}}(u^S-u_m)^2(u_{\xi}^R)^2(u_{\xi\xi}^R)^2d\xi\\
			&+C\int_{\mathbb{R}}(u^R-u_m)^2(u_{\xi}^S)^2(u_{\xi\xi}^S)^2d\xi+C\int_{\mathbb{R}}(u^S+u^R-u_m)^2(u_{\xi}^S)^2(u_{\xi\xi}^R)^2d\xi\\
			&+C\int_{\mathbb{R}}(u^S+u^R-u_m)^2(u_{\xi}^R)^2(u_{\xi\xi}^S)^2d\xi+C\int_{\mathbb{R}}(u^S-u_m)^2(u^S+2u^R-u_m)^2(u_{\xi\xi\xi}^R)^2d\xi\\
			&+C\int_{\mathbb{R}}(u^R-u_m)^2(2u^S+u^R-u_m)^2(u_{\xi\xi\xi}^S)^2d\xi+C\int_{\mathbb{R}}(u_{\xi\xi\xi\xi}^R)^2d\xi.
		\end{align*}
		Clearly,
		\begin{align*}
			&\int_{\mathbb{R}}(u_{\xi}^S)^4(u_{\xi}^R)^2d\xi\leq Cu_m^6\int_{\mathbb{R}}(u_{\xi}^S)^2(u_{\xi}^R)^2d\xi\leq Cu_m^6\delta_R^{\frac{4}{7}}(1+t)^{-\frac{6}{5}},\\
			&\int_{\mathbb{R}}(u_{\xi}^S)^2(u_{\xi}^R)^4d\xi\leq C\delta_{R}^2\int_{\mathbb{R}}(u_{\xi}^S)^2(u_{\xi}^R)^2d\xi\leq Cu_m^4\delta_R^{\frac{4}{7}}(1+t)^{-\frac{6}{5}},\\
			&\int_{\mathbb{R}}(u^S-u_m)^2(u_{\xi}^R)^2(u_{\xi\xi}^R)^2d\xi\leq C\delta_{R}^2\int_{\mathbb{R}}(u^S-u_m)^2(u_{\xi}^R)^2d\xi\leq Cu_m^4\delta_R^{\frac{43}{55}} (1+t)^{-\frac{6}{5}},\\
			&\int_{\mathbb{R}}(u^R-u_m)^2(u_{\xi}^S)^2(u_{\xi\xi}^S)^2d\xi\leq Cu_m^{10}\int_{\mathbb{R}}(u^R-u_m)^2(u_{\xi}^S)^2d\xi\leq Cu_m^{10}\delta_R^{\frac{4}{7}}(1+t)^{-\frac{4}{3}},\\
			&\int_{\mathbb{R}}(u^S+u^R-u_m)^2(u_{\xi}^S)^2(u_{\xi\xi}^R)^2d\xi\leq Cu_m^6\int_{\mathbb{R}}(u_{\xi\xi}^R)^2d\xi\leq Cu_m^6\delta_R^{\frac{2}{3}} (1+t)^{-\frac{4}{3}},\\
			&\int_{\mathbb{R}}(u^S+u^R-u_m)^2(u_{\xi}^R)^2(u_{\xi\xi}^S)^2d\xi\leq C(u_m^2+\delta_{R}^2)\int_{\mathbb{R}}(u_{\xi}^S)^2(u_{\xi}^R)^2d\xi\leq Cu_m^2\delta_R^{\frac{4}{7}}(1+t)^{-\frac{6}{5}},\\
			&\int_{\mathbb{R}}(u^R_{\xi\xi\xi\xi})^2d\xi\leq C \|u_{\xi\xi\xi \xi}^{R}\|_{L^\infty(\mathbb{R})}^{\frac{2}{3}} \|u_{\xi \xi\xi\xi}^{R}\|_{L^\infty(\mathbb{R})}^{\frac{1}{3}} (1+t)^{-1} \leq Cu_m\delta_R^{\frac{1}{6}} (1+t)^{-\frac{4}{3}},\\
			&\int_{\mathbb{R}}(u^S-u_m)^2(u^S+2u^R-u_m)^2(u_{\xi\xi\xi}^R)^2d\xi\leq Cu_m^2\int_{\mathbb{R}}(u_{\xi\xi\xi}^R)^2d\xi\leq Cu_m^2\delta_R^{\frac{2}{3}}(1+t)^{-\frac{4}{3}},\\
			&\int_{\mathbb{R}}(u^R-u_m)^2(2u^S+u^R-u_m)^2(u_{\xi\xi\xi}^S)^2d\xi\leq Cu_m^2\int_{\mathbb{R}}(u^R-u_m)^2(u_{\xi}^S)^2d\xi \leq Cu_m^2\delta_R^{\frac{4}{7}}(1+t)^{-\frac{4}{3}}.
		\end{align*}
		Hence, we have 
		\begin{align*}
			\int_{\mathbb{R}}(F_{\xi\xi})^2d\xi\leq Cu_m\delta_R^{\frac{1}{6}}(1+t)^{-\frac{6}{5}}.
		\end{align*}
		This completes the proof of Lemma \ref{lemf}.
	\end{proof}
	Finally, it remains to estimate $\mathbf{Z}(t).$ 
	\begin{lemma}\label{lem7}
		$\mathbf{Z}(t)$ is defined in Lemma $(\ref{lem2})$, it holds that for $\forall t\in \left[0,T\right]$
		\begin{equation}
			\begin{aligned}
				\left|\int_{0}^t\mathbf{Z}(s)ds\right|\leq&\frac{112\sqrt{\tau}}{3}u_m^2\int_{\mathbb{R}}w\left(\phi_{0}^2+\tau r_0^2\right)d\xi+\frac{224\sqrt{\tau}}{3}u_m^2\int_{\mathbb{R}}w\eta(U^{-\mathbf{X}}|\tilde{U})d\xi.\\
				&+\frac{1}{1800}\int_{0}^t\int_{\mathbb{R}} \left(\phi_{\xi}^{-\mathbf{X}}\right)^2wd\xi ds+\frac{2u_m^3}{15}\int_{0}^t\int_{\mathbb{R}}\left(\phi^{-\mathbf{X}}\right)^2u_{\xi}^Sd\xi ds\\
				&+C\tau u_m^4\int_{0}^t\int_{\mathbb{R}}(r^{-\mathbf{X}})^2wd\xi ds+Cu_m^2\| \phi^{-\mathbf{X}} \|_{L^2(\mathbb{R})}^{\frac{2}{3}}\delta_{R}^{\frac{8}{33}}.
			\end{aligned}
		\end{equation}
	\end{lemma}
	\begin{proof}
		\begin{equation}
			\begin{aligned}
				\left|\int_{0}^t\mathbf{Z}(s)ds\right|
				\leq&\left| \int_{0}^t\int_{\mathbb{R}} F \phi^{-\mathbf{X}} w d\xi ds\right|+\tau\left|\int_{0}^t\int_{\mathbb{R}}r^{-\mathbf{X}}Fw'u_{\xi}^{S}d\xi ds\right|\\
				&+\tau\left|\int_{\mathbb{R}}r^{-\mathbf{X}}\phi^{-\mathbf{X}}w'u_{\xi}^{S}d\xi\right|+\tau\left|\int_{\mathbb{R}}r_0\phi_0w'u_{\xi}^{S}d\xi\right|.
			\end{aligned}	
		\end{equation}
		By Sobolev inequality, combined with Cauchy inequality and Minkowski inequality, we have
		\begin{equation}
			\begin{aligned}
				&\left| \int_{\mathbb{R}} F \phi^{-\mathbf{X}} w d\xi \right| \leq \| \phi^{-\mathbf{X}}w \|_{L^{\infty}(\mathbb{R})} \int_{\mathbb{R}} |F| d\xi \\
				\leq &C\| \phi^{-\mathbf{X}}w \|_{L^2(\mathbb{R})}^{\frac{1}{2}} \| \left(\phi^{-\mathbf{X}}w\right)_{\xi} \|_{L^2(\mathbb{R})}^{\frac{1}{2}} \int_{\mathbb{R}} |F| d\xi\\
				\leq&C\| \phi^{-\mathbf{X}}w \|_{L^2(\mathbb{R})}^{\frac{1}{2}}\left(\|\phi_{\xi}^{-\mathbf{X}}w\|_{L^2(\mathbb{R})}^{\frac{1}{2}}+\|\phi^{-\mathbf{X}}w_{\xi}\|_{L^2(\mathbb{R})}^{\frac{1}{2}}\right)\int_{\mathbb{R}} |F| d\xi\\
				\leq&C\| \phi^{-\mathbf{X}}w \|_{L^2(\mathbb{R})}^{\frac{1}{2}}\|\phi_{\xi}^{-\mathbf{X}}w\|_{L^2(\mathbb{R})}^{\frac{1}{2}}\int_{\mathbb{R}} |F| d\xi+C\| \phi^{-\mathbf{X}}w \|_{L^2(\mathbb{R})}^{\frac{1}{2}}\|\phi^{-\mathbf{X}}w_{\xi}\|_{L^2(\mathbb{R})}^{\frac{1}{2}}\int_{\mathbb{R}} |F| d\xi\\
				\leq& \frac{1}{13500u_m^2}\| \phi_{\xi}^{-\mathbf{X}}w \|_{L^2(\mathbb{R})}^2+\frac{2}{2625u_m^2}\|\phi^{-\mathbf{X}}w'u_{\xi}^S\|_{L^2(\mathbb{R})}^2 + C u_m^2 \| \phi^{-\mathbf{X}} \|_{L^2(\mathbb{R})}^{\frac{2}{3}} \left( \int_{\mathbb{R}} |F| d\xi \right)^{\frac{4}{3}}\\
				\leq& \frac{1}{1800}\int_{\mathbb{R}} \left(\phi_{\xi}^{-\mathbf{X}}\right)^2wd\xi+\frac{2u_m^3}{15}\int_{\mathbb{R}}\left(\phi^{-\mathbf{X}}\right)^2u_{\xi}^Sd\xi+Cu_m^2 \| \phi^{-\mathbf{X}} \|_{L^2(\mathbb{R})}^{\frac{2}{3}} \left( \int_{\mathbb{R}} |F| d\xi \right)^{\frac{4}{3}},
			\end{aligned}
		\end{equation}
		where
		\begin{equation} \label{F}
			\int_{0}^t\left( \int_{\mathbb{R}} |F| d\xi \right)^{\frac{4}{3}}ds \leq C\delta_R^{\frac{8}{33}},
		\end{equation}
		see \cite{ref16}. Thus, we obtain
		\begin{equation}\label{Z1}
			\begin{aligned}
				\left| \int_{0}^t\int_{\mathbb{R}} F \phi^{-\mathbf{X}} w d\xi ds\right| \leq&\frac{1}{1800}\int_{0}^t\int_{\mathbb{R}} \left(\phi_{\xi}^{-\mathbf{X}}\right)^2wd\xi ds+\frac{2u_m^3}{15}\int_{0}^t\int_{\mathbb{R}}\left(\phi^{-\mathbf{X}}\right)^2u_{\xi}^Sd\xi ds\\
				&+Cu_m^2\| \phi^{-\mathbf{X}} \|_{L^2(\mathbb{R})}^{\frac{2}{3}}\delta_{R}^{\frac{8}{33}}.
			\end{aligned}
		\end{equation}
		Similarly,
		\begin{equation}\label{Z2}
			\begin{aligned}
				&\tau\left|\int_{0}^t\int_{\mathbb{R}}r^{-\mathbf{X}}Fw'u_{\xi}^{S}d\xi ds\right|\\
				\leq&\tau\left\lVert\frac{w'}{w}\right\rVert_{L^{\infty}(\mathbb{R})}\lVert u_{\xi}^S\rVert_{L^{\infty}(\mathbb{R})}\left|\int_{0}^t\int_{\mathbb{R}}r^{-\mathbf{X}}Fwd\xi ds\right|\\
				\leq&\frac{112\tau}{3}u_m^2\left(\int_{0}^t\int_{\mathbb{R}}(r^{-\mathbf{X}})^2w^2d\xi ds+\int_{0}^t\int_{\mathbb{R}}|F|^2d\xi ds\right)\\
				\leq&280\tau u_m^4\int_{0}^t\int_{\mathbb{R}}(r^{-\mathbf{X}})^2wd\xi ds+\frac{112\tau}{3}u_m^2\int_{0}^t\int_{\mathbb{R}}|F|^2d\xi ds,
			\end{aligned}
		\end{equation}
		where $\int_{0}^t\int_{\mathbb{R}}|F|^2d\xi ds\leq C\delta_{R}^{\frac{4}{7}}$, by Lemma \ref{lemf}. Finally, 
		\begin{equation}\label{Z3}
			\begin{aligned}
				\tau\left|\int_{\mathbb{R}}r^{-\mathbf{X}}\phi^{-\mathbf{X}}w'u_{\xi}^{S}d\xi\right|\leq&\sqrt{\tau}\left\lVert\frac{w'}{w}\right\rVert_{L^{\infty}(\mathbb{R})}\|u_{\xi}^{S}\|_{L^{\infty}(\mathbb{R})}\left|\int_{\mathbb{R}}\sqrt{\tau}r^{-\mathbf{X}}\phi^{-\mathbf{X}}wd\xi\right|\\
				\leq&\frac{112\sqrt{\tau}}{3}u_m^2\left(\int_{\mathbb{R}}\tau(r^{-\mathbf{X}})^2wd\xi+\int_{\mathbb{R}}(\phi^{-\mathbf{X}})^2wd\xi\right)\\
				=&\frac{224\sqrt{\tau}}{3}u_m^2\int_{\mathbb{R}}w\eta(U^{-\mathbf{X}}|\tilde{U})d\xi,
			\end{aligned}
		\end{equation}
		and
		\begin{equation}\label{Z4}
			\begin{aligned}
				\tau\left|\int_{\mathbb{R}}r_0\phi_0w'u_{\xi}^{S}d\xi\right|\leq\frac{112\sqrt{\tau}}{3}u_m^2\int_{\mathbb{R}}w\left(\phi_{0}^2+\tau r_0^2\right)d\xi.
			\end{aligned}
		\end{equation}
		By combining results $(\ref{Z1})$ through $(\ref{Z4})$, we prove Lemma $\ref{lem7}$.
	\end{proof}
	Thus, using Lemma $\ref{lem2.5},\ref{lem3}$ and integrating $(\ref{yds})$ with respect to $t$, we can obtain $L^2$ estimates as following:
	\begin{equation}\label{zero}
		\begin{aligned}
			&\int_{\mathbb{R}}w\eta(U^{-\mathbf{X}}|\tilde{U})d\xi+ \left(2-2k_{1}\right)\int_0^t\int_{\mathbb{R}}(r^{-\mathbf{X}})^2wd\xi ds + \frac{25u_m^2}{64}\int_0^t |\dot{\mathbf{X}}(s)|^2 ds \\
			&+ \frac{4u_m^3}{5}\int_0^t \int_{\mathbb{R}} (\phi^{-\mathbf{X}})^2u^S_{\xi}d\xi ds+3u_m\int_0^t \int_{\mathbb{R}} (\phi^{-\mathbf{X}})^2wu^R_\xi d\xi ds+2\int_0^t\mathbf{G}^{SR}(s)ds \\
			\leq& Cu_m^2 (\|\phi_0\|_{L^2}^2+\tau\|r_0\|_{L^2}^2 )+2k_2\int_{0}^t\int_{\mathbb{R}}\left(\phi_{\xi}^{-\mathbf{X}}\right)^2wd\xi ds+Cu_m^2\delta_R^{\frac{8}{33}}.
		\end{aligned}
	\end{equation} 
	Here $\tau,u_{m},k_1,k_2$ satisfy the following inequality:
	\begin{equation}\label{k1k2}
		\left\{\begin{array}{l}
			\frac{224\sqrt{\tau}}{3}u_m^2\leq \frac{1}{2} ,
			\\C\delta _R+\|\phi^{-\mathbf{X}}\|_{L^{\infty }(\mathbb{R})}\leq \frac{3u_m}{2}, 
			\\\frac{u_m^2}{64}+\frac{5\tau u_m^2}{64}\leq \frac{25u_m^2}{128},  
			\\\tau\left(M_1+C\|\phi^{-\mathbf{X}}\|_{L^{\infty }(\mathbb{R})} \right)+\frac{4u_m^3}{15}\leq \frac{2u_m^3}{5} ,
			\\\tau\left(M_2+C\delta _R^{\frac{1}{3}}+C\|\phi^{-\mathbf{X}}\|_{L^{\infty }(\mathbb{R})}\right)+C\tau u_m^4\leq k_1\leq1,
			\\\tau\left(C u_m^4+C\delta _R+C\|\phi^{-\mathbf{X}}\|_{L^{\infty }(\mathbb{R})}\right)+\frac{751}{900}\leq k_2,  
		\end{array}\right.
	\end{equation}
	where $M_{1},M_{2}$ are given in Lemma $\ref{lem6}$. For a given constant state $u_{-}<0$, we can choose $$\delta_{0}=4u_m=-2u_{-}>0,\ \varepsilon_{2}=\delta_{0}^2.$$ By the smallness of $\delta_{0}$ and $\delta_{S}=3u_m<\delta_{0}$ together with \eqref{tau}, we get that \eqref{k1k2}$_1$ and \eqref{k1k2}$_3$ are satisfied. Using the smallness of $\varepsilon_{2},\delta_{R}=\delta_{S}^2$ and $\|\phi^{-\mathbf{X}}\|_{L^{\infty }(\mathbb{R})}\leq \varepsilon_{2}$, we get that \eqref{k1k2}$_2$ also holds.  These two inequalities \eqref{k1k2}$_4$ and \eqref{k1k2}$_5$ will be used to determine suitable values of $k_1,k_2$. 
	
	Changing of variable $\xi\to\xi+\mathbf{X}(t),$ we have obtained the proof of Lemma $\ref{lem2}.$
	
	\subsection{High-order energy estimates}
	\begin{lemma}\label{lem8}
		 There exist constants $\delta_{0},\varepsilon_{3}>0$, if the rarefaction wave and shock wave strength $\delta_{R}=\delta_{S}^2,\delta_{S}\leq \delta_{0}$ and $$\sup_{0\leq t\leq T} \|\phi^{-\mathbf{X}}(t,\cdot), \sqrt{\tau}r^{-\mathbf{X}}(t,\cdot)\|_{H^2(\mathbb{R})} \leq \varepsilon_{3},$$ then there exists a constant $C>0$ such that for $0\leq t\leq T$, we have 
		\begin{equation}\label{HEE}
			\begin{aligned}
				&\|\partial_{\xi}\phi^{-\mathbf{X}}\sqrt{w}\|_{H^1(\mathbb{R})}^2+\tau\|\partial_{\xi}r^{-\mathbf{X}}\sqrt{w}\|_{H^1(\mathbb{R})}^2+\int_{0}^t\|\partial_{\xi}r^{-\mathbf{X}}\sqrt{w}\|_{H^1(\mathbb{R})}^2 ds\\
				\leq&Cu_m^2\left(\|\partial_{\xi}^k \phi_{0}\|_{L^{2}(\mathbb{R})}^2+\|\partial_{\xi}^kr_{0}\|_{L^{2}(\mathbb{R})}^2\right)+\frac{u_m^2}{27}\int_{0}^t\left|\dot{\mathbf{X}}(s)\right|^2ds\\
				&+\left(Cu_m^3+C\delta_{R}\right)\int_{0}^t\int_{\mathbb{R}}(\phi^{-\mathbf{X}})^2wu_{\xi}^Rd\xi ds+(Cu_m^3+A_1)u_m^2\int_{0}^t\int_{\mathbb{R}}(\phi^{-\mathbf{X}})^2u_{\xi}^Sd\xi ds\\
				&+\left(Cu_m^2+A_1+C\|\phi_{\xi}^{-\mathbf{X}}\|_{L^{\infty}(\mathbb{R})}\right)\int_{0}^t\|\partial_{\xi}\phi^{-\mathbf{X}}\sqrt{w}\|_{H^1(\mathbb{R})}^2 ds+Cu_m^2\delta_{R}^{\frac{1}{6}},
			\end{aligned}
		\end{equation}
		where  $A_1:=C\delta_R+C\|\phi^{-\mathbf{X}}\|_{L^{\infty}(\mathbb{R})}$.
	\end{lemma}
	\begin{proof}
		Applying $\partial_{\xi}^k(k=1,2)$ to the system $(\ref{PFC})$ and changing $\xi\to\xi-\mathbf{X}(t)$, we get that
		\begin{equation}\label{HPFC}
			\left\{\begin{array}{l}\partial_{\xi}^k\phi_{t}^{-\mathbf{X}}-\sigma \partial_{\xi}^{k+1}\phi^{-\mathbf{X}}+\partial_{\xi}^{k+1}\left(f\left(\phi^{-\mathbf{X}}+\tilde{u}\right)-f\left(\tilde{u}\right)\right)\\
				\ \ \ \ \ \ \ \ \ \ +\dot{\mathbf{X}}(t)\left(\partial_{\xi}^{k+1}u^{S}+\partial_{\xi}^{k+1}u^{R}\right)-\partial_{\xi}^{k+1}r^{-\mathbf{X}}=-\partial_{\xi}^kF, \\ 
				\tau \partial_{\xi}^kr_{t}^{-\mathbf{X}}-\tau\sigma\partial_{\xi}^{k+1} r^{-\mathbf{X}}+\partial_{\xi}^kr^{-\mathbf{X}}-\partial_{\xi}^{k+1}\phi^{-\mathbf{X}}+\tau \dot{\mathbf{X}}(t) \partial_{\xi}^{k+1}\tilde{q}+\tau \partial_{\xi}^{k+1}\left(u^{R}\right)_{t}=0.\end{array}\right. 
		\end{equation}
		Multiplying the above equations by $w\partial_{\xi}^k\phi^{-\mathbf{X}},\ w\partial_{\xi}^kr^{-\mathbf{X}}$, and integrating over $\mathbb{R}$, we have
		\begin{equation}\label{gjfc}
			\begin{aligned}
				\frac{1}{2}\frac{d}{dt}\int_{\mathbb{R}}\left[\left(\partial_{\xi}^k\phi^{-\mathbf{X}}\right)^2+\tau \left(\partial_{\xi}^kr^{-\mathbf{X}}\right)^2\right]wd\xi+\int_{\mathbb{R}} \left(\partial_{\xi}^kr^{-\mathbf{X}}\right)^2wd\xi=\sum_{i=1}^{8}R_i^k,
			\end{aligned}
		\end{equation}
		where
		\begin{align*}
			&R_1^k=-\frac{\sigma}{2}\int_{\mathbb{R}}\left(\partial_{\xi}^k\phi^{-\mathbf{X}}\right)^2w'u_{\xi}^Sd\xi,\\
			&R_2^k=-\int_{\mathbb{R}}\partial_{\xi}^{k+1}\left(f\left(\phi^{-\mathbf{X}}+\tilde{u}\right)-f\left(\tilde{u}\right)\right)\partial_{\xi}^k\phi^{-\mathbf{X}}wd\xi,\\
			&R_3^k=-\dot{\mathbf{X}}(t)\int_{\mathbb{R}}\left(\partial_{\xi}^{k+1}u^{S}+\partial_{\xi}^{k+1}u^{R}\right)\partial_{\xi}^k\phi^{-\mathbf{X}}wd\xi,\\
			&R_4^k=-\int_{\mathbb{R}}\partial_{\xi}^kr^{-\mathbf{X}}\partial_{\xi}^k\phi^{-\mathbf{X}}w'u_{\xi}^Sd\xi,\ R_5^k=-\int_{\mathbb{R}}\partial_{\xi}^kF\partial_{\xi}^k\phi^{-\mathbf{X}}wd\xi,\\
			&R_6^k=-\frac{\tau\sigma}{2}\int_{\mathbb{R}}(\partial_{\xi}^kr^{-\mathbf{X}})^2w'u_{\xi}^Sd\xi,\ R_7^k=-\tau \dot{\mathbf{X}}(t) \int_{\mathbb{R}}\partial_{\xi}^{k+1}\tilde{q}\partial_{\xi}^kr^{-\mathbf{X}}wd\xi,\\
			&R_8^k=-\tau \int_{\mathbb{R}}\partial_{\xi}^{k+1}\left(u^{R}\right)_{t}\partial_{\xi}^kr^{-\mathbf{X}}wd\xi.
		\end{align*}
		It is easy to check that 
		\begin{equation}
			\begin{aligned}
				&\left|R_1^k\right|\leq Cu_m^4\int_{\mathbb{R}}\left(\partial_{\xi}^k\phi^{-\mathbf{X}}\right)^2wd\xi,\\
				&\left|R_3^k\right|\leq\frac{u_m^2}{81}\left|\dot{\mathbf{X}}(t)\right|^2+\left(Cu_m^5+C\delta_{R}\right)\int_{\mathbb{R}}\left(\partial_{\xi}^k\phi^{-\mathbf{X}}\right)^2wd\xi,\\
				&\left|R_4^k\right|\leq Cu_m^3\int_{\mathbb{R}}\left(\partial_{\xi}^kr^{-\mathbf{X}}\right)^2wd\xi+Cu_m^3\int_{\mathbb{R}}\left(\partial_{\xi}^k\phi^{-\mathbf{X}}\right)^2wd\xi,\\
				&\left|R_6^k\right|\leq  C\tau u_m^4\int_{\mathbb{R}}\left(\partial_{\xi}^kr^{-\mathbf{X}}\right)^2wd\xi,\\
				&\left|R_7^k\right|\leq\frac{\tau u_m^2}{192}\left|\dot{\mathbf{X}}(t)\right|^2+\left(C u_m+C\delta_{R}\right)\int_{\mathbb{R}}\left(\partial_{\xi}^kr^{-\mathbf{X}}\right)^2wd\xi.
			\end{aligned}
		\end{equation}
		For $R_2^k$, since
		\begin{align*}
			&\partial_{\xi}^{k+1}\left(f\left(\phi^{-\mathbf{X}}+\tilde{u}\right)-f\left(\tilde{u}\right)\right)\partial_{\xi}^k\phi^{-\mathbf{X}}\\
			=&\left(\frac{1}{2}f'(\phi^{-\mathbf{X}}+\tilde{u})(\partial_{\xi}^{k}\phi^{-\mathbf{X}})^{2}\right)_{\xi}\\
			&+\partial_{\xi}^{k+1}(f(\phi^{-\mathbf{X}}+\tilde{u})-f(\tilde{u}))\partial_{\xi}^{k}\phi^{-\mathbf{X}}-\left(\frac{1}{2}f'(\phi^{-\mathbf{X}}+\tilde{u})(\partial_{\xi}^{k}\phi^{-\mathbf{X}})^{2}\right)_{\xi}\\
			=:&\left(\frac{1}{2}f'(\phi^{-\mathbf{X}}+\tilde{u})(\partial_{\xi}^{k}\phi^{-\mathbf{X}})^{2}\right)_{\xi}+K_{k}(t,\xi).
		\end{align*}
		Thus, we get 
		\begin{align*}
			&\left|R_2^k\right|
			=\left|\int_{\mathbb{R}}\partial_{\xi}^{k+1}\left(f\left(\phi^{-\mathbf{X}}+\tilde{u}\right)-f\left(\tilde{u}\right)\right)\partial_{\xi}^k\phi^{-\mathbf{X}}wd\xi\right|\\
			\leq&\left|\int_{\mathbb{R}}\left(\frac{1}{2}f'(\phi^{-\mathbf{X}}+\tilde{u})(\partial_{\xi}^{k}\phi^{-\mathbf{X}})^{2}\right)_{\xi}wd\xi\right|+\left|\int_{\mathbb{R}}K_{k}(t,\xi)wd\xi\right|\\
			\leq&\left(Cu_m^2+C\delta_R+C\|\phi^{-\mathbf{X}}\|_{L^{\infty}(\mathbb{R})}^2\right)\int_{\mathbb{R}}(\partial_{\xi}^{k}\phi^{-\mathbf{X}})^{2}wd\xi+\left|\int_{\mathbb{R}}K_{k}(t,\xi)wd\xi\right|.
		\end{align*}
		Then we estimate $K_k(t,\xi)$, for $k=1,$ we have
		\begin{align*}
			&K_1(t,\xi)=\partial_{\xi}^{2}(f(\phi^{-\mathbf{X}}+\tilde{u})-f(\tilde{u}))\phi_{\xi}^{-\mathbf{X}}-(\frac{1}{2}f'(\phi^{-\mathbf{X}}+\tilde{u})(\phi_{\xi}^{-\mathbf{X}})^{2})_{\xi}\\
			=&(f''(\phi^{-\mathbf{X}}+\tilde{u})-f''(\tilde{u}))(\tilde{u}_{\xi})^{2}\phi_{\xi}^{-\mathbf{X}}+2f''(\phi^{-\mathbf{X}}+\tilde{u})\tilde{u}_{\xi}(\phi_{\xi}^{-\mathbf{X}})^{2}+f''(\phi^{-\mathbf{X}}+\tilde{u})(\phi_{\xi}^{-\mathbf{X}})^{3}\\
			&+(f'(\phi^{-\mathbf{X}}+\tilde{u})-f'(\tilde{u}))\tilde{u}_{\xi\xi}\phi_{\xi}^{-\mathbf{X}}-\frac{1}{2}f''(\phi^{-\mathbf{X}}+\tilde{u})(\phi_{\xi}^{-\mathbf{X}}+\tilde{u}_{\xi})(\phi_{\xi}^{-\mathbf{X}})^{2}\\
			=&6(\tilde{u}_{\xi})^2\phi^{-\mathbf{X}}\phi_{\xi}^{-\mathbf{X}}+9(\phi^{-\mathbf{X}}+\tilde{u})\tilde{u}_{\xi}(\phi_{\xi}^{-\mathbf{X}})^{2}+3(\phi^{-\mathbf{X}}+\tilde{u})(\phi_{\xi}^{-\mathbf{X}})^{3}+3(\phi^{-\mathbf{X}}+2\tilde{u})\phi^{-\mathbf{X}}\tilde{u}_{\xi\xi}\phi_{\xi}^{-\mathbf{X}},
		\end{align*}
		By the Lemma $\ref{lem1}$, we have
		\begin{align*}
			&\left|\int_{\mathbb{R}}K_1(t,\xi)wd\xi\right|
			\leq\left| 6\int_{\mathbb{R}}(\tilde{u}_{\xi})^2\phi^{-\mathbf{X}}\phi_{\xi}^{-\mathbf{X}}wd\xi\right|+\left|9\int_{\mathbb{R}}(\phi^{-\mathbf{X}}+\tilde{u})\tilde{u}_{\xi}(\phi_{\xi}^{-\mathbf{X}})^{2}wd\xi\right|\\
			&\quad+\left|3\int_{\mathbb{R}}(\phi^{-\mathbf{X}}+\tilde{u})(\phi_{\xi}^{-\mathbf{X}})^{3}wd\xi\right|+\left|3\int_{\mathbb{R}}(\phi^{-\mathbf{X}}+2\tilde{u})\phi^{-\mathbf{X}}\tilde{u}_{\xi\xi}\phi_{\xi}^{-\mathbf{X}}wd\xi\right|\\
			\leq&3\|\tilde{u}_{\xi}\|_{L^{\infty}(\mathbb{R})}\int_{\mathbb{R}}(\phi^{-\mathbf{X}})^2w\tilde{u}_{\xi}d\xi+3\|\tilde{u}_{\xi}\|_{L^{\infty}(\mathbb{R})}^2\int_{\mathbb{R}}(\phi_{\xi}^{-\mathbf{X}})^2wd\xi\\
			&+9\|\tilde{u}_{\xi}\|_{L^{\infty}(\mathbb{R})}\left(\|\phi^{-\mathbf{X}}\|_{L^{\infty}(\mathbb{R})}+\|\tilde{u}\|_{L^{\infty}(\mathbb{R})}\right)\int_{\mathbb{R}}(\phi_{\xi}^{-\mathbf{X}})^{2}wd\xi+C\|\phi_{\xi}^{-\mathbf{X}}\|_{L^{\infty}(\mathbb{R})}\int_{\mathbb{R}}(\phi_{\xi}^{-\mathbf{X}})^{2}wd\xi\\
			&+15u_m^2\left(\|\phi^{-\mathbf{X}}\|_{L^{\infty}(\mathbb{R})}+2\|\tilde{u}\|_{L^{\infty}(\mathbb{R})}\right)\|u^S_{\xi}\|_{L^{\infty}(\mathbb{R})}\int_{\mathbb{R}}(\phi_{\xi}^{-\mathbf{X}})^{2}wd\xi\\
			&+15u_m^2\left(\|\phi^{-\mathbf{X}}\|_{L^{\infty}(\mathbb{R})}+2\|\tilde{u}\|_{L^{\infty}(\mathbb{R})}\right)\int_{\mathbb{R}}(\phi^{-\mathbf{X}})^2wu_{\xi}^Sd\xi\\
			&+3\left(\|\phi^{-\mathbf{X}}\|_{L^{\infty}(\mathbb{R})}+2\|\tilde{u}\|_{L^{\infty}(\mathbb{R})}\right)\frac{\|\phi^{-\mathbf{X}}\|_{L^{\infty}(\mathbb{R})}}{2}\left(\int_{\mathbb{R}}(\phi_{\xi}^{-\mathbf{X}})^{2}wd\xi+Cu_m^2\delta_{R}^{\frac{2}{3}}(1+t)^{-\frac{4}{3}}\right)\\
			\leq&(Cu_m^3+A_1)u_m^2\int_{\mathbb{R}}(\phi^{-\mathbf{X}})^2u^S_{\xi}d\xi+(Cu_m^3+C\delta_R)\int_{\mathbb{R}}(\phi^{-\mathbf{X}})^2wu^R_{\xi}d\xi\\
			&+(Cu_m^4+A_1+C\|\phi_{\xi}^{-\mathbf{X}}\|_{L^{\infty}(\mathbb{R})})\int_{\mathbb{R}}(\phi_{\xi}^{-\mathbf{X}})^{2}wd\xi+Cu_m^2\delta_{R}^{\frac{2}{3}}(1+t)^{-\frac{4}{3}},
		\end{align*}
		where
		\begin{align*}
			&\left|3\int_{\mathbb{R}}(\phi^{-\mathbf{X}}+2\tilde{u})\phi^{-\mathbf{X}}\tilde{u}_{\xi\xi}\phi_{\xi}^{-\mathbf{X}}wd\xi\right|\\
			\leq&3\left(\|\phi^{-\mathbf{X}}\|_{L^{\infty}(\mathbb{R})}+2\|\tilde{u}\|_{L^{\infty}(\mathbb{R})}\right)\left[\int_{\mathbb{R}}10u_m^2u_{\xi}^S\phi^{-\mathbf{X}}\phi_{\xi}^{-\mathbf{X}}wd\xi+\int_{\mathbb{R}}u_{\xi\xi}^R\phi^{-\mathbf{X}}\phi_{\xi}^{-\mathbf{X}}wd\xi\right]\\
			\leq&15u_m^2\left(\|\phi^{-\mathbf{X}}\|_{L^{\infty}(\mathbb{R})}+2\|\tilde{u}\|_{L^{\infty}(\mathbb{R})}\right)\left[\int_{\mathbb{R}}(\phi^{-\mathbf{X}})^2wu_{\xi}^Sd\xi+\|u^S_{\xi}\|_{L^{\infty}(\mathbb{R})}\int_{\mathbb{R}}(\phi_{\xi}^{-\mathbf{X}})^2wd\xi\right]\\
			&+\frac{3}{2}\left(\|\phi^{-\mathbf{X}}\|_{L^{\infty}(\mathbb{R})}+2\|\tilde{u}\|_{L^{\infty}(\mathbb{R})}\right)\|\phi^{-\mathbf{X}}\|_{L^{\infty}(\mathbb{R})}\left[\int_{\mathbb{R}}(\phi_{\xi}^{-\mathbf{X}})^2wd\xi+\int_{\mathbb{R}}(u_{\xi\xi}^R)^2wd\xi\right].
		\end{align*}
		Similarly, for $k=2$, we have
		\begin{align*}
			K_2(t,\xi)=&\left[f'''(\phi^{-\mathbf{X}}+\tilde{u})-f'''(\tilde{u})\right](\tilde{u}_{\xi})^3\phi^{-\mathbf{X}}_{\xi\xi}+3f'''(\phi^{-\mathbf{X}}+\tilde{u})(\tilde{u}_{\xi})^{2}\phi_{\xi}^{-\mathbf{X}}\phi_{\xi\xi}^{-\mathbf{X}}
			\\
			&+3f'''(\phi^{-\mathbf{X}}+\tilde{u})\tilde{u}_{\xi}(\phi_{\xi}^{-\mathbf{X}})^{2}\phi_{\xi\xi}^{-\mathbf{X}}+f'''(\phi^{-\mathbf{X}}+\tilde{u})(\phi_{\xi}^{-\mathbf{X}})^{3}\phi_{\xi\xi}^{-\mathbf{X}}\\
			&+3[f''(\phi^{-\mathbf{X}}+\tilde{u})-f''(\tilde{u})]\tilde{u}_{\xi}\tilde{u}_{\xi\xi}\phi_{\xi\xi}^{-\mathbf{X}}+\frac{5}{2}f''(\phi^{-\mathbf{X}}+\tilde{u})\tilde{u}_{\xi}(\phi_{\xi\xi}^{-\mathbf{X}})^{2}\\
			&+3f''(\phi^{-\mathbf{X}}+\tilde{u})\tilde{u}_{\xi\xi}\phi_{\xi}^{-\mathbf{X}}\phi_{\xi\xi}^{-\mathbf{X}}+\frac{5}{2}f''(\phi^{-\mathbf{X}}+\tilde{u})\phi_{\xi}^{-\mathbf{X}}(\phi_{\xi\xi}^{-\mathbf{X}})^{2}\\
			&+[f'(\phi^{-\mathbf{X}}+\tilde{u})-f'(\tilde{u})]\tilde{u}_{\xi\xi\xi}\phi_{\xi\xi}^{-\mathbf{X}}\\
			=&18(\tilde{u}_{\xi})^{2}\phi_{\xi}^{-\mathbf{X}}\phi_{\xi\xi}^{-\mathbf{X}}+18\tilde{u}_{\xi}(\phi_{\xi}^{-\mathbf{X}})^{2}\phi_{\xi\xi}^{-\mathbf{X}}+6(\phi_{\xi}^{-\mathbf{X}})^{3}\phi_{\xi\xi}^{-\mathbf{X}}+18\tilde{u}_{\xi}\tilde{u}_{\xi\xi}\phi^{-\mathbf{X}}\phi_{\xi\xi}^{-\mathbf{X}}\\
			&+15(\phi^{-\mathbf{X}}+\tilde{u})\tilde{u}_{\xi}(\phi_{\xi\xi}^{-\mathbf{X}})^{2}+18(\phi^{-\mathbf{X}}+\tilde{u})\tilde{u}_{\xi\xi}\phi_{\xi}^{-\mathbf{X}}\phi_{\xi\xi}^{-\mathbf{X}}\\
			&+15(\phi^{-\mathbf{X}}+\tilde{u})\phi_{\xi}^{-\mathbf{X}}(\phi_{\xi\xi}^{-\mathbf{X}})^{2}+3(\phi^{-\mathbf{X}}+2\tilde{u})\phi^{-\mathbf{X}}\tilde{u}_{\xi\xi\xi}\phi_{\xi\xi}^{-\mathbf{X}}.
		\end{align*}
		It is easy to show
		\begin{align*}
			&\int_{\mathbb{R}}18(\tilde{u}_{\xi})^{2}\phi_{\xi}^{-\mathbf{X}}\phi_{\xi\xi}^{-\mathbf{X}}wd\xi\leq 9\|\tilde{u}_{\xi}\|_{L^{\infty}(\mathbb{R})}^2\int_{\mathbb{R}}(\phi_{\xi}^{-\mathbf{X}})^2wd\xi+9\|\tilde{u}_{\xi}\|_{L^{\infty}(\mathbb{R})}^2\int_{\mathbb{R}}(\phi_{\xi\xi}^{-\mathbf{X}})^2wd\xi,\\
			&\int_{\mathbb{R}}18\tilde{u}_{\xi}(\phi_{\xi}^{-\mathbf{X}})^{2}\phi_{\xi\xi}^{-\mathbf{X}}wd\xi\leq 9\|\tilde{u}_{\xi}\|_{L^{\infty}(\mathbb{R})}\|\phi^{-\mathbf{X}}_{\xi}\|_{L^{\infty}(\mathbb{R})}\int_{\mathbb{R}}(\phi_{\xi}^{-\mathbf{X}})^2wd\xi\\&\ \ \ \ \ \ \ \ \ \ \ \ \ \ \ \ \ \ \ \ \ \ \ \ \ \ \ \ \ \ \ +9\|\tilde{u}_{\xi}\|_{L^{\infty}(\mathbb{R})}\|\phi^{-\mathbf{X}}_{\xi}\|_{L^{\infty}(\mathbb{R})}\int_{\mathbb{R}}(\phi_{\xi\xi}^{-\mathbf{X}})^2wd\xi,\\
			&\int_{\mathbb{R}}6(\phi_{\xi}^{-\mathbf{X}})^{3}\phi_{\xi\xi}^{-\mathbf{X}}wd\xi\leq 3\|\phi^{-\mathbf{X}}_{\xi}\|_{L^{\infty}(\mathbb{R})}^2\int_{\mathbb{R}}(\phi_{\xi}^{-\mathbf{X}})^2wd\xi+3\|\phi^{-\mathbf{X}}_{\xi}\|_{L^{\infty}(\mathbb{R})}^2\int_{\mathbb{R}}(\phi_{\xi\xi}^{-\mathbf{X}})^2wd\xi,\\
			&\int_{\mathbb{R}}18\tilde{u}_{\xi}\tilde{u}_{\xi\xi}\phi^{-\mathbf{X}}\phi_{\xi\xi}^{-\mathbf{X}}wd\xi\leq9\|\tilde{u}_{\xi\xi}\|_{L^{\infty}(\mathbb{R})}\int_{\mathbb{R}}(\phi^{-\mathbf{X}})^2w\left(u^S_{\xi}+u^R_{\xi}\right)d\xi\\
			&\ \ \ \ \ \ \ \ \ \ \ \ \ \ \ \ \ \ \ \ \ \ \ \ \ \ \ \ \ \ \ +9\|\tilde{u}_{\xi}\|_{L^{\infty}(\mathbb{R})}\|\tilde{u}_{\xi\xi}\|_{L^{\infty}(\mathbb{R})}\int_{\mathbb{R}}(\phi_{\xi\xi}^{-\mathbf{X}})^2wd\xi,\\
			&\int_{\mathbb{R}}15(\phi^{-\mathbf{X}}+\tilde{u})\tilde{u}_{\xi}(\phi_{\xi\xi}^{-\mathbf{X}})^{2}wd\xi\leq15\|\tilde{u}_{\xi}\|_{L^{\infty}(\mathbb{R})}\left(\|\phi^{-\mathbf{X}}+\tilde{u}\|_{L^{\infty}(\mathbb{R})}\right)\int_{\mathbb{R}}(\phi_{\xi\xi}^{-\mathbf{X}})^{2}wd\xi,\\
			&\int_{\mathbb{R}}18(\phi^{-\mathbf{X}}+\tilde{u})\tilde{u}_{\xi\xi}\phi_{\xi}^{-\mathbf{X}}\phi_{\xi\xi}^{-\mathbf{X}}wd\xi\leq 9\|\tilde{u}_{\xi\xi}\|_{L^{\infty}(\mathbb{R})}\left(\|\phi^{-\mathbf{X}}+\tilde{u}\|_{L^{\infty}(\mathbb{R})}\right)\int_{\mathbb{R}}(\phi_{\xi}^{-\mathbf{X}})^{2}wd\xi\\
			&\ \ \ \ \ \ \ \ \ \ \ \ \ \ \ \ \ \ \ \ \ \ \ \ \ \ \ \ \ \ \ \ \ \ \ \ \ \ \ +9\|\tilde{u}_{\xi\xi}\|_{L^{\infty}(\mathbb{R})}\left(\|\phi^{-\mathbf{X}}+\tilde{u}\|_{L^{\infty}(\mathbb{R})}\right)\int_{\mathbb{R}}(\phi_{\xi\xi}^{-\mathbf{X}})^{2}wd\xi,\\
			&\int_{\mathbb{R}}15(\phi^{-\mathbf{X}}+\tilde{u})\phi_{\xi}^{-\mathbf{X}}(\phi_{\xi\xi}^{-\mathbf{X}})^{2}wd\xi\leq15\left(\|\phi^{-\mathbf{X}}+\tilde{u}\|_{L^{\infty}(\mathbb{R})}\right)\|\phi_{\xi}^{-\mathbf{X}}\|_{L^{\infty}(\mathbb{R})}\int_{\mathbb{R}}(\phi_{\xi\xi}^{-\mathbf{X}})^{2}wd\xi,
		\end{align*}
		and
		\begin{align*}
			&\int_{\mathbb{R}}3(\phi^{-\mathbf{X}}+2\tilde{u})\phi^{-\mathbf{X}}\tilde{u}_{\xi\xi\xi}\phi_{\xi\xi}^{-\mathbf{X}}wd\xi\\
			\leq&3\left(\|\phi^{-\mathbf{X}}+2\tilde{u}\|_{L^{\infty}(\mathbb{R})}\right)\int_{\mathbb{R}}\left(u^S_{\xi\xi\xi}+u^R_{\xi\xi\xi}\right)\phi^{-\mathbf{X}}\phi_{\xi\xi}^{-\mathbf{X}}wd\xi\\
			\leq&\frac{3}{2}\left(\|\phi^{-\mathbf{X}}+2\tilde{u}\|_{L^{\infty}(\mathbb{R})}\right)\|u^S_{\xi\xi\xi}\|_{L^{\infty}(\mathbb{R})}\int_{\mathbb{R}}(\phi^{-\mathbf{X}})^2wu^S_{\xi}d\xi\\
			&+\frac{3}{2}\left(\|\phi^{-\mathbf{X}}+2\tilde{u}\|_{L^{\infty}(\mathbb{R})}\right)\left\lVert\frac{u^S_{\xi\xi\xi}}{u^S_{\xi}}\right\rVert_{L^{\infty}(\mathbb{R})}\int_{\mathbb{R}}(\phi_{\xi\xi}^{-\mathbf{X}})^{2}wd\xi\\
			&+\frac{3}{2}\left(\|\phi^{-\mathbf{X}}+2\tilde{u}\|_{L^{\infty}(\mathbb{R})}\right)\|\phi^{-\mathbf{X}}\|_{L^{\infty}(\mathbb{R})}\int_{\mathbb{R}}(\phi_{\xi\xi}^{-\mathbf{X}})^{2}wd\xi\\
			&+\frac{3}{2}\left(\|\phi^{-\mathbf{X}}+2\tilde{u}\|_{L^{\infty}(\mathbb{R})}\right)\|\phi^{-\mathbf{X}}\|_{L^{\infty}(\mathbb{R})}\int_{\mathbb{R}}(u^R_{\xi\xi\xi})^2wd\xi.
		\end{align*}
		Then we get
		\begin{align*}
			\left|\int_{\mathbb{R}}K_2(t,\xi)wd\xi\right|\leq& (Cu_m+A_1)u_m^7\int_{\mathbb{R}}(\phi^{-\mathbf{X}})^2u^S_{\xi}d\xi+(Cu_m^5+C\delta_R)\int_{\mathbb{R}}(\phi^{-\mathbf{X}})^2wu^R_{\xi}d\xi\\
			&+\left(Cu_m^6+A_1+C\|\phi_{\xi}^{-\mathbf{X}}\|_{L^{\infty}(\mathbb{R})}\right)\int_{\mathbb{R}}(\phi_{\xi}^{-\mathbf{X}})^2wd\xi\\
			&+\left(Cu_m^4+A_1+C\|\phi_{\xi}^{-\mathbf{X}}\|_{L^{\infty}(\mathbb{R})}\right)\int_{\mathbb{R}}(\phi_{\xi\xi}^{-\mathbf{X}})^2wd\xi+Cu_m^2\delta_{R}^{\frac{2}{3}}(1+t)^{-\frac{4}{3}}.
		\end{align*}
		Hence, we have
		\begin{align*}
			\left|R_2^k\right|\leq &(Cu_m^3+A_1)\int_{\mathbb{R}}(\phi^{-\mathbf{X}})^2wu^R_{\xi}d\xi+(Cu_m^3+A_1)u_m^2\int_{\mathbb{R}}(\phi^{-\mathbf{X}})^2u^S_{\xi}d\xi\\
			&+\left(Cu_m^2+A_1+C\|\phi_{\xi}^{-\mathbf{X}}\|_{L^{\infty}(\mathbb{R})}\right)\sum_{i=1}^{k}\int_{\mathbb{R}}(\partial_{\xi}^i\phi^{-\mathbf{X}})^2wd\xi+Cu_m^2\delta_R^{\frac{2}{3}}(1+t)^{-\frac{4}{3}}.
		\end{align*}
		For $R_5^k,$ we have
		\begin{align*}
			\left|R_5^k\right|=\left|\int_{\mathbb{R}}\partial_{\xi}^kF\partial_{\xi}^k\phi^{-\mathbf{X}}wd\xi\right|
			\leq&Cu_m\int_{\mathbb{R}}(\partial_{\xi}^k\phi^{-\mathbf{X}})^2wd\xi+Cu_m\int_{\mathbb{R}}\left|\partial_{\xi}^kF\right|^2d\xi\\
			\leq&Cu_m\int_{\mathbb{R}}(\partial_{\xi}^k\phi^{-\mathbf{X}})^2wd\xi+Cu_m^2\delta_{R}^{\frac{1}{6}}(1+t)^{-\frac{6}{5}}. 
		\end{align*}
		For $R_8^k$, we have $(u^R)_t=\sigma u^R_{\xi}+f(u^R)_{\xi},$ then
		\begin{align*}
			\left|R_8^k\right|=&\left|-\tau\sigma\int_{\mathbb{R}}\partial_{\xi}^{k+2}u^R\partial_{\xi}^kr^{-\mathbf{X}}wd\xi-\tau\int_{\mathbb{R}}\partial_{\xi}^{k+2}\left((u^R)^3\right)\partial_{\xi}^{k}r^{-\mathbf{X}}wd\xi\right|\\
			\leq&\frac{\tau\sigma^2}{2}\int_{\mathbb{R}}\left(\partial_{\xi}^{k+2}u^R\right)^2d\xi+\frac{15\tau}{2}u_m^2\int_{\mathbb{R}}\left(\partial_{\xi}^kr^{-\mathbf{X}}\right)^2wd\xi+\frac{\tau}{2}\int_{\mathbb{R}}\left(\partial_{\xi}^{k+2}\left((u^R)^3\right)\right)^2d\xi\\
			\leq&C\tau u_m^2\int_{\mathbb{R}}\left(\partial_{\xi}^kr^{-\mathbf{X}}\right)^2wd\xi+C\delta_{R}^{\frac{2}{3}}(1+t)^{-\frac{4}{3}},
		\end{align*}
		where we use Lemma $\ref{lem1}$. Without loss of generality, we set $\delta_{0}=4u_m,\ \varepsilon_{3}=\varepsilon_{2}=\delta_{0}^2.$ Given the smallness of $\delta_{0}$ and $\varepsilon_{3}$, the conclusion of Lemma \ref{lem8} follows upon integrating \eqref{gjfc} from $0$ to $t$.
	\end{proof}
	
	\subsection{Dissipative estimates}
	In this section, we prove dissipative estimates. To this end, we first introduce the following lemma.
	\begin{lemma}\label{lem9}
		There exist positive constants $\delta_{0},\varepsilon_{4}$, if the rarefaction wave and shock wave strength $\delta_{R},\delta_{S}\leq\delta_{0}$, and $$\sup_{0\leq t\leq T} \|\phi^{-\mathbf{X}}(t,\cdot), \sqrt{\tau}r^{-\mathbf{X}}(t,\cdot)\|_{H^2(\mathbb{R})} \leq \varepsilon_{4},$$ then there exist positive constants $C,\varepsilon_*$ such that for $0\leq t\leq T,$ the following properties hold true:
		\begin{equation}\label{de}
			\begin{aligned}
				&\int_{0}^t\|\partial_{\xi}\phi^{-\mathbf{X}}\sqrt{w}\|_{H^1(\mathbb{R})}^2ds\leq Cu_m^2\left(\|r_0\|_{H^1(\mathbb{R})}^2+\|\partial_{\xi}\phi_{0}\|_{H^1}^2\right)+6\tau u_m^2\int_{0}^t\left|\dot{\mathbf{X}}(s)\right|^2ds+Cu_m^2\delta_{R}^{\frac{4}{7}}\\
				&+\varepsilon_*\tau\|r^{-\mathbf{X}}\sqrt{w}\|_{H^1(\mathbb{R})}^2+C(\varepsilon_*)\tau\|\partial_{\xi}\phi^{-\mathbf{X}}\sqrt{w}\|_{H^1(\mathbb{R})}^2+\tau (Cu_m^3+A_1)\int_{0}^t\int_{\mathbb{R}}(\phi^{-\mathbf{X}})^2wu^R_{\xi}d\xi ds\\
				&+\tau (Cu_m^2+A_1)u_m^2\int_{0}^t\int_{\mathbb{R}}(\phi^{-\mathbf{X}})^2u^S_{\xi}d\xi ds+\bigg(\frac{8}{7}+3\tau+C\tau u_m^3+\tau A_1\bigg)\int_{0}^t\|\partial_{\xi}r^{-\mathbf{X}}\sqrt{w}\|_{H^1(\mathbb{R})}^2ds.
			\end{aligned}
		\end{equation}
		where $A_1=C\delta_R+C\|\phi^{-\mathbf{X}}\|_{L^{\infty}(\mathbb{R})}$.
	\end{lemma}
	\begin{proof}
		Multiplying the equation $(\ref{HPFC})_2$ by $-w\partial_{\xi}^{k+1}\phi^{-\mathbf{X}}$ for $k=0,1$, then integrating over $(0,t)\times\mathbb{R}$, we have
		\begin{align*}
			\int_{0}^t\int_{\mathbb{R}}(\partial_{\xi}^{k+1}\phi^{-\mathbf{X}})^2wd\xi ds=\sum_{i=1}^{5}N_i^k,
		\end{align*}
		where
		\begin{align*}
			&N_1^k=\tau\int_{0}^t\int_{\mathbb{R}}\partial_{\xi}^kr^{-\mathbf{X}}_t\partial_{\xi}^{k+1}\phi^{-\mathbf{X}}wd\xi ds,\\
			&N_2^k=-\tau\sigma\int_{0}^t\int_{\mathbb{R}}\partial_{\xi}^{k+1}r^{-\mathbf{X}}\partial_{\xi}^{k+1}\phi^{-\mathbf{X}}wd\xi ds,\\
			&N_3^k=\int_{0}^t\int_{\mathbb{R}}\partial_{\xi}^kr^{-\mathbf{X}}\partial_{\xi}^{k+1}\phi^{-\mathbf{X}}wd\xi ds,\\
			&N_4^k=\tau\dot{\mathbf{X}}\int_{0}^t\int_{\mathbb{R}}\partial_{\xi}^{k+1}\tilde{q}\partial_{\xi}^{k+1}\phi^{-\mathbf{X}}wd\xi ds,\\
			&N_5^k=\tau\int_{0}^t\int_{\mathbb{R}}\partial_{\xi}^{k+1}(u^R)_t\partial_{\xi}^{k+1}\phi^{-\mathbf{X}}wd\xi ds.
		\end{align*}
		Using the equation $(\ref{PFC})_1$, we get 
		\begin{align*}
			&N_1^k+N_2^k\\
			=&\tau\int_{\mathbb{R}}\partial_{\xi}^kr^{-\mathbf{X}}\partial_{\xi}^{k+1}\phi^{-\mathbf{X}}wd\xi-\tau\int_{\mathbb{R}}\partial_{\xi}^kr_0\partial_{\xi}^{k+1}\phi_{0}wd\xi+\tau\sigma\int_{0}^t\int_{\mathbb{R}}\partial_{\xi}^kr^{-\mathbf{X}}\partial_{\xi}^{k+1}\phi^{-\mathbf{X}}w'u^S_{\xi}d\xi ds\\
			&-\tau\int_{0}^t\int_{\mathbb{R}}\partial_{\xi}^kr^{-\mathbf{X}}\partial_{\xi}^{k+1}\left[-F-\left(f(\phi^{-\mathbf{X}}+\tilde{u})-f(\tilde{u})\right)_{\xi}-\dot{\mathbf{X}}(s)\left(u^S_{\xi}+u^R_{\xi}\right)+r^{-\mathbf{X}}_{\xi}\right]wd\xi ds\\
			=&\tau\int_{\mathbb{R}}\partial_{\xi}^kr^{-\mathbf{X}}\partial_{\xi}^{k+1}\phi^{-\mathbf{X}}wd\xi-\tau\int_{\mathbb{R}}\partial_{\xi}^kr_0\partial_{\xi}^{k+1}\phi_{0}wd\xi+\tau\int_0^t\int_{\mathbb{R}}\partial_{\xi}^kr^{-\mathbf{X}}\partial_{\xi}^{k+1}Fwd\xi ds\\
			&+\tau\int_0^t\int_{\mathbb{R}}\partial_{\xi}^kr^{-\mathbf{X}}\partial_{\xi}^{k+2}\left(f(\phi^{-\mathbf{X}}+\tilde{u})-f(\tilde{u})\right)wd\xi ds+\tau\sigma\int_{0}^t\int_{\mathbb{R}}\partial_{\xi}^kr^{-\mathbf{X}}\partial_{\xi}^{k+1}\phi^{-\mathbf{X}}w'u^S_{\xi}d\xi ds\\
			&+\tau\int_0^t\dot{\mathbf{X}}(s)\int_{\mathbb{R}}\partial_{\xi}^kr^{-\mathbf{X}}\partial_{\xi}^{k+2}(u^S+u^R)wd\xi ds-\tau\int_0^t\int_{\mathbb{R}}\partial_{\xi}^kr^{-\mathbf{X}}\partial_{\xi}^{k+2}r^{-\mathbf{X}}wd\xi ds.
		\end{align*} 
		By Young's inequality, we have
		\begin{align*}
			\tau\int_0^t\int_{\mathbb{R}}\partial_{\xi}^kr^{-\mathbf{X}}\partial_{\xi}^{k+1}Fwd\xi ds
			\leq&\frac{\tau}{2}\int_{0}^t\int_{\mathbb{R}}(\partial_{\xi}^kr^{-\mathbf{X}})^2wd\xi ds+Cu_m^2\delta_R^{\frac{4}{7}},
		\end{align*}
		and
		\begin{align*}
			&\tau\sigma\int_{0}^t\int_{\mathbb{R}}\partial_{\xi}^kr^{-\mathbf{X}}\partial_{\xi}^{k+1}\phi^{-\mathbf{X}}w'u^S_{\xi}d\xi ds\\
			\leq&\frac{\tau\sigma}{2}\left\lVert\frac{w'}{w}\right\rVert_{L^{\infty}(\mathbb{R})}\|u^S_{\xi}\|_{L^{\infty}(\mathbb{R})}\int_{0}^t\int_{\mathbb{R}}(\partial_{\xi}^kr^{-\mathbf{X}})^2wd\xi ds\\
			&+\frac{\tau\sigma}{2}\left\lVert\frac{w'}{w}\right\rVert_{L^{\infty}(\mathbb{R})}\|u^S_{\xi}\|_{L^{\infty}(\mathbb{R})}\int_{0}^t\int_{\mathbb{R}}(\partial_{\xi}^{k+1}\phi^{-\mathbf{X}})^2wd\xi ds\\
			\leq&C\tau u_m^4\int_{0}^t\int_{\mathbb{R}}(\partial_{\xi}^kr^{-\mathbf{X}})^2wd\xi ds+C\tau u_m^4\int_{0}^t\int_{\mathbb{R}}(\partial_{\xi}^{k+1}\phi^{-\mathbf{X}})^2wd\xi ds,
		\end{align*}
		and
		\begin{align*}
			&-\tau\int_0^t\int_{\mathbb{R}}\partial_{\xi}^kr^{-\mathbf{X}}\partial_{\xi}^{k+2}r^{-\mathbf{X}}wd\xi ds\leq\tau\int_{0}^t\int_{\mathbb{R}}(\partial_{\xi}^{k+1}r^{-\mathbf{X}})^2wd\xi ds+C\tau u_m^4\int_0^t\int_{\mathbb{R}}(\partial_{\xi}^kr^{-\mathbf{X}})^2wd\xi,
		\end{align*}
		\begin{align*}
			&\tau\int_0^t\dot{\mathbf{X}}(s)\int_{\mathbb{R}}\partial_{\xi}^kr^{-\mathbf{X}}\partial_{\xi}^{k+2}(u^S+u^R)wd\xi ds\\
			\leq&\tau u_m^2\int_{0}^t\left|\dot{\mathbf{X}}(s)\right|^2ds+\tau\left(Cu_m^{12}+C\delta_R\right)\int_{0}^t\int_{\mathbb{R}}(\partial_{\xi}^kr^{-\mathbf{X}})^2wd\xi ds.
		\end{align*}
		Notice that 
		\begin{align*}
			\tau\int_{\mathbb{R}}\partial_{\xi}^kr^{-\mathbf{X}}\partial_{\xi}^{k+1}\phi^{-\mathbf{X}}wd\xi\leq \varepsilon_*\tau\int_{\mathbb{R}}(\partial_{\xi}^kr^{-\mathbf{X}})^2wd\xi+C(\varepsilon_*)\tau\int_{\mathbb{R}}(\partial_{\xi}^{k+1}\phi^{-\mathbf{X}})^2wd\xi.
		\end{align*}
		For $k=0$, we notice that 
		\begin{align*}
			&\tau\int_0^t\int_{\mathbb{R}}r^{-\mathbf{X}}\left(f(\phi^{-\mathbf{X}}+\tilde{u})-f(\tilde{u})\right)_{\xi\xi}wd\xi ds\\
			=&6\tau\int_0^t\int_{\mathbb{R}}(\tilde{u}_{\xi})^2r^{-\mathbf{X}}\phi^{-\mathbf{X}}wd\xi ds+12\tau\int_{0}^t\int_{\mathbb{R}}(\phi^{-\mathbf{X}}+\tilde{u})\tilde{u}_{\xi}r^{-\mathbf{X}}\phi^{-\mathbf{X}}_{\xi}wd\xi ds\\
			&+6\tau\int_{0}^t\int_{\mathbb{R}}(\phi^{-\mathbf{X}}+\tilde{u})(\phi^{-\mathbf{X}}_{\xi})^3wd\xi ds+3\tau\int_{0}^t\int_{\mathbb{R}}(\phi^{-\mathbf{X}}+2\tilde{u})\tilde{u}_{\xi\xi}r^{-\mathbf{X}}\phi^{-\mathbf{X}}wd\xi ds\\
			&-6\tau\int_{0}^t\int_{\mathbb{R}}(\phi^{-\mathbf{X}}+\tilde{u})(\phi_{\xi}^{-\mathbf{X}}+\tilde{u}_{\xi})r^{-\mathbf{X}}\phi_{\xi}^{-\mathbf{X}}wd\xi ds\\
			&-3\tau\int_{0}^t\int_{\mathbb{R}}(\phi^{-\mathbf{X}}+\tilde{u})^2r^{-\mathbf{X}}\phi_{\xi}^{-\mathbf{X}}w'u_{\xi}^Sd\xi ds-3\tau\int_{0}^t\int_{\mathbb{R}}(\phi^{-\mathbf{X}}+\tilde{u})^2r_{\xi}^{-\mathbf{X}}\phi_{\xi}^{-\mathbf{X}}wd\xi ds\\
			\leq& \tau(Cu_m^6+A_1)\int_0^t\int_{\mathbb{R}}(r^{-\mathbf{X}})^2wd\xi ds+\tau(Cu_m^2+A_1)u_m^2\int_0^t\int_{\mathbb{R}}(\phi^{-\mathbf{X}})^2u^S_{\xi}d\xi ds\\
			&+\tau(Cu_m^3+A_1)\int_0^t\int_{\mathbb{R}}(\phi^{-\mathbf{X}})^2wu^R_{\xi}d\xi ds+\tau(Cu_m^2+A_1)\int_0^t\int_{\mathbb{R}}(r_{\xi}^{-\mathbf{X}})^2wd\xi ds\\
			&+\tau\left(Cu_m^2+A_1+C\|\phi_{\xi}^{-\mathbf{X}}\|_{L^{\infty}(\mathbb{R})}\right)\int_0^t\int_{\mathbb{R}}(\phi_{\xi}^{-\mathbf{X}})^2wd\xi ds+Cu_m^2\delta_R^{\frac{2}{3}}.
		\end{align*}
		Similarly, for $k=1$, we have
		\begin{align*}
			&\tau\int_0^t\int_{\mathbb{R}}r_{\xi}^{-\mathbf{X}}\left(f(\phi^{-\mathbf{X}}+\tilde{u})-f(\tilde{u})\right)_{\xi\xi\xi}wd\xi ds\\
			\leq&\tau(Cu_m^3+A_1)u_m^2\int_0^t\int_{\mathbb{R}}(\phi^{-\mathbf{X}})^2u^S_{\xi}d\xi ds+\tau\left(Cu_m^3+C\delta_R\right)\int_0^t\int_{\mathbb{R}}(\phi^{-\mathbf{X}})^2wu^R_{\xi}d\xi ds\\
			&+\tau\left(Cu_m^3+A_1+C\|\phi_{\xi}^{-\mathbf{X}}\|_{L^{\infty}(\mathbb{R})}\right)\int_0^t\int_{\mathbb{R}}(r_{\xi}^{-\mathbf{X}})^2wd\xi ds\\
			&+\tau\left(Cu_m^2+A_1+C\|\phi_{\xi}^{-\mathbf{X}}\|_{L^{\infty}(\mathbb{R})}\right)\int_0^t\int_{\mathbb{R}}(\phi_{\xi}^{-\mathbf{X}})^2wd\xi ds\\
			&+\tau(Cu_m^2+A_1)\int_0^t\int_{\mathbb{R}}(r_{\xi\xi}^{-\mathbf{X}})^2wd\xi ds+C\delta_R^{\frac{2}{3}}+\tau(Cu_m^2+A_1)\int_0^t\int_{\mathbb{R}}(\phi_{\xi\xi}^{-\mathbf{X}})^2wd\xi ds.
		\end{align*}
		For $N_3^k$, we have
		\begin{align*}
			N_3^k&\leq\frac{1}{2}\int_{0}^t\int_{\mathbb{R}}(\partial_{\xi}^kr^{-\mathbf{X}})^2wd\xi ds+\frac{1}{2}\int_{0}^t\int_{\mathbb{R}}(\partial_{\xi}^{k+1}\phi^{-\mathbf{X}})^2wd\xi ds.
		\end{align*}
		Since $\tilde{q}_{\xi}=q^S_{\xi}+u^R_{\xi\xi}$, we get
		\begin{align*}
			N_4^k=&\tau\int_{0}^t\dot{\mathbf{X}}(s)\int_{\mathbb{R}}\partial_{\xi}^{k+1}\tilde{q}\partial_{\xi}^{k+1}\phi^{-\mathbf{X}}wd\xi ds\\
			\leq&\tau u_m^2\int_{0}^t\left|\dot{\mathbf{X}}(s)\right|^2ds+\tau\left(Cu_m^7+C\delta_R\right)\int_{0}^t\int_{\mathbb{R}}(\partial_{\xi}^{k+1}\phi^{-\mathbf{X}})^2wd\xi ds,
		\end{align*}
		and
		\begin{align*}
			N_5^k=&\tau\int_{0}^t\int_{\mathbb{R}}\partial_{\xi}^{k+1}(u^R)_t\partial_{\xi}^{k+1}\phi^{-\mathbf{X}}wd\xi ds
			\leq C\tau u_m^2\int_{0}^t\int_{\mathbb{R}}(\partial_{\xi}^{k+1}\phi^{-\mathbf{X}})wd\xi ds+Cu_m^2\delta_R^{\frac{2}{3}}.
		\end{align*}
		Without loss of generality, let $\delta_{0}=4u_m,\ \varepsilon_{4}=\delta_{0}^2$, by Lemma \ref{lem2}, Lemma \ref{lem9} is then proved by integrating the above inequalities with respect to $t$.
	\end{proof}
	Finally, Combining Lemma \ref{lem2} with Lemmas \ref{lem8} and \ref{lem9}, the parameters $k_1,k_2,\tau,\varepsilon_*$ are chosen to satisfy
	\begin{align}\label{nk1k2}
		\begin{cases}
			2k_2\varepsilon_*\leq\frac{1}{4},\\
			2k_2C(\varepsilon_*)\tau=\frac{k_2\tau}{2\varepsilon_*}\leq \frac{1}{2},\\
			2k_2\tau(Cu_m^3+A_1)\leq \frac{3}{2}u_m,\\
			2k_2u_m^2(Cu_m^2+A_1)\leq \frac{2u_m^3}{5},\\
			2k_2(\frac{8}{7}+3\tau+C\tau u_m^4+\tau A_1)<2-2k_1, 
		\end{cases}
	\end{align}
	where $A_1$ is given by Lemma \ref{lem8}. There exists a family of solutions $(k_1,k_2,\tau,\varepsilon_*)=(\frac{1}{1574},\frac{1513}{1800},\frac{1}{1574},\frac{225}{1513})$ for which both inequalities \eqref{k1k2} and \eqref{nk1k2} are satisfied.
	
	Notice that $\|\phi\|_{L^{\infty}(\mathbb{R})}=\|\phi^{-\mathbf{X}}\|_{L^{\infty}(\mathbb{R})}$, $\|\phi_{\xi}\|_{L^{\infty}(\mathbb{R})}=\|\phi_{\xi}^{-\mathbf{X}}\|_{L^{\infty}(\mathbb{R})}$ and $$\sup_{0\leq t\leq T} \|\phi^{-\mathbf{X}}, \sqrt{\tau}r^{-\mathbf{X}}\|_{H^2(\mathbb{R})}=\sup_{0\leq t\leq T} \|\phi, \sqrt{\tau}r\|_{H^2(\mathbb{R})}.$$  Changing of variable $\xi\to\xi+\mathbf{X}(t)$, we take $\delta_{1}=\delta_{0}=4u_m$ and set $\varepsilon_{1}=\delta_{0}^2$. This completes the proof of Proposition \ref{p1}.
	
\end{document}